\documentclass[12pt]{amsart}
\usepackage[toc,page]{appendix}

\usepackage{fancyhdr}
\newcommand{\nc}{\newcommand}
\newcommand{\rc}{\renewcommand}

\usepackage[ansinew]{inputenc}
\usepackage{color}
\usepackage{amscd}
\usepackage{enumerate,latexsym}
\usepackage{amsmath,amssymb}
\usepackage{graphicx}
\usepackage{amsmath,amssymb}
\usepackage{amsxtra}
\usepackage{amsfonts}
\usepackage{enumerate}

\definecolor{b}{rgb}{.1,.1,.7}
\definecolor{rr}{rgb}{.8,0,.3}
\definecolor{g}{rgb}{0,.5,0}
\definecolor{pp}{rgb}{.5,0,.7}
\definecolor{r}{rgb}{.6,0,.3}
\definecolor{y}{rgb}{.9,.99,.9}

\newcommand{\bbb}{\textcolor{b}}

\newcommand{\bblock}{\begin{block}}
\newcommand{\eblock}{\end{block}}

\def\R{\mathbb{R}}

\newcommand{\bit}{\begin{itemize}}
\newcommand{\een}{\end{enumerate}}
\newcommand{\eit}{\end{itemize}}
\newcommand{\ed}{\end{document}}

\nc{\Aut}{{	\operatorname{Aut}	}}
\nc{\codim}{{	\operatorname{codim}	}}
\nc{\Ob}{{	\operatorname{Ob}	}}
\nc{\PGL}{{	\operatorname{PGL}	}}
\nc{\supp}{{	\operatorname{supp}	}}
\nc{\tr}{{	\operatorname{tr}	}}

\nc{\Rep}{{	{\cal{R}}ep		}}
\nc{\one}{{	\mbox{\bf{1}}		}}
\nc{\iso}{	\overset{\sim}{\lra}	}

\nc{\nen}{\newenvironment}

\nc{\pr}{\protect}

\nc{\nn}{{\newline}}
\nc{\np}{{\newpage}}	

\nc{\lab}{	\label}
\nc{\npp}{{	\newpage\setcounter{page}{0}	}}
\nc{\setpa}{		\setcounter{part}		}
\nc{\setse}{		\setcounter{section}	}
\nc{\setsus}{		\setcounter{subsection}		}
\nc{\setsss}{		\setcounter{subsubsection}	}
\nc{\setpage}{		\setcounter{page}	}

\nc{\nfd}{ $$\text{ This version is preliminary and approximate, 
		             it is not for distribution. }$$	}

\nc{\noi}{{\noindent}}
\nc{\pf}{{	\noindent {\em Proof.}		}}
					
\nc{\epf}{ \fbox{\bf QED}	}

\nc{\heart}{{\tiny \cen{\tiny $\heartsuit $ }	}} 
\nc{\cont}{\tableofcontents}

\nc{\sbr}{{	\smallpagebreak	}}
\nc{\mbr}{{	\medpagebreak	}}
\nc{\bbr}{{	\bigpagebreak	}}

\nc{\bib}{		}

\rc{\b}{ 	\big         			}  
\nc{\lam}[1]{{ 	\text{\large $#1$	}	}}  
\nc{\smm}[1]{{ 	\text{\small $#1$	}	}}  
\nc{\fom}[1]{{ 	\text{\footnotesize $#1$	}	}}  
\nc{\tinm}[1]{{ \text{\tiny $#1$	}	}}  

\nc{\bu}{ \bullet         }  			
\nc{\bbu}{ \aa{\bbb \bullet}         }  	
\nc{\bus}{{	^\bullet	}}	 	
\nc{\bui}{{	_\bullet	}}	 	

\nc{\bem}{{	\begin{em}	}}
\nc{\eem}{{	\end{em} 	}}

\nc{\bbox}{{	\blackbox	}}	
\nc{\bx}{	\boxed	}		
\nc{\tbx}[1]{{\boxed{\tx{#1}}}}		

\nc{\mmbox}[1]{{	\mbox{$#1$}	}}	
\nc{\tbox}[1]{{		\mbox{\tx{#1}}	}}
 	
\nc{\ot}{		\leftarrow			}
\nc{\tto}{		\longrightarrow			}
\nc{\ott}{		\longleftarrow			}
\nc{\too}[1]{{		\aa{#1}\rightarrow			}}
\nc{\oot}[1]{{		\aa{#1}\leftarrow			}}
\nc{\ttoo}[1]{{		\aa{#1}\longrightarrow			}}
\nc{\oott}[1]{{		\aa{#1}\longleftarrow			}}

\nc{\Too}[2]{{		\aa{#1}{\bb{#2}\rightarrow}		}}
\nc{\ooT}[2]{{		\aa{#1}{\bbb{#2}\leftarrow}		}}
\nc{\TToo}[2]{{		\aa{#1}{\bb{#2}\longrightarrow}		}}
\nc{\ooTT}[2]{{		\aa{#1}{\bbb{#2}\longleftarrow}		}}
\nc{\toot}[2]{{		\aa{#1}{\bb{#2}\rightleftarrows}	}}
\nc{\ttoot}[2]{{	\aa{#1}{\bb{#2}\rightleftrightarrows}	}}

\nc{\ra}{{	\rightarrow		}}
\nc{\laa}{{	\leftarrow	}}	
\nc{\lra}{{\longrightarrow}}

\nc{\lr}{{\leftrightarrow}}     	
\nc{\lrs}{{\rightleftarrows}}     	

\nc{\imp}{{\Rightarrow}}        	
\nc{\impp}{{\Leftarrow}}        	
\nc{\eq}{{\Leftrightarrow}}        	
\nc{\impl}{{\Longrightarrow}}        	
\nc{\imppl}{{\Longleftarrow}}        	
\nc{\eql}{{\Longleftrightarrow}}        	
	\nc{\Ra}{{\Rightarrow}}         	
	\nc{\LRa}{{\Leftrightarrow}}        	

\nc{\inj}{{\pr	\hookrightarrow	}}    		
\nc{\injj}{{\pr	\hookleftarrow	}}    		

\nc{\sur}{{	\twoheadrightarrow	}}	
\nc{\surr}{{	\twoheadleftarrow	}}	
			
\nc{\mm}{{	\mapsto		}}     		
\nc{\mmm}{{	\leftarrow\shortmid }}		
\nc{\ainj}[1]{{\aa{#1}{\pr\hookrightarrow}	}}    	
\nc{\ainjj}[1]{{\aa{#1}{\pr\hookleftarrow}	}}    	

\nc{\asur}[1]{{	\aa{#1}\twoheadrightarrow	}}	
\nc{\asurr}[1]{{\aa{#1}\twoheadleftarrow	}}	
			
\nc{\amm}[1]{{	\aa{#1}\mapsto		}}     	
\nc{\ammm}[1]{{	\aa{#1}\leftarrow\shortmid }}	

\nc{\syp}[1]{	^{ (#1) }		} 	
\nc{\up}[1]{	^{ (#1) }		} 	
\nc{\lp}[1]{	_{ (#1) }		}	
\nc{\hp}[1]{	^{ [#1] }		}	
\nc{\cle}{\preceq}		
\nc{\cl}{\prec}			
\nc{\cge}{\succeq}		
\nc{\cg}{\succ}			

\nc{\bb}{	\pr\underset 	}           
\rc{\aa}{ 	\pr\overset 	}            

\nc{\indd}{{ ${} \ \ \ \ \  \ \        {} $	}}	
\nc{\inddd}{{ 	\indd\indd			}}	
\nc{\nnd}{{ 	\nn  \indd 			}}	
\nc{\nndb}{{ 	\nn  \indd $\bullet$		}}	
		
\nc{\bce}{	\begin{center}	}
\nc{\ece}{	\end  {center}	}
\nc{\cen}[1]{	\begin{center}	{  #1}	\end  {center}	}

\nc{\bss}{{\backslash}}           		
\nc{\barr}{ 	\overline 	}      		
\nc{\ud}{	\underline	}		

\nc{\ti}{\tilde}              
\nc{\tii}{\widetilde}         

\nc{\hatt}{\widehat}				
\nc{\hata}{{	\bbb{ \hat{} }		}}	

\nc{\ch}{\check}              			
\nc{\cha}{{ 	\bbb{ \check{} }	}}      

\nc{\sub}{{	\subseteq	}}         
\nc{\subb}{{	\supseteq	}}         
\nc{\nsub}{{	\nsubseteq	}}         
\nc{\nsubb}{{	\nsupseteq	}}         %

\nc{\nin}{{	\notin	}}

\nc{\lb}{\langle}             				
\nc{\rb}{\rangle}

\nc{\lB}{	\left(	}             			
\nc{\rB}{	\right)	}
\nc{\BBl}{{	\bbb{ \left( \right.}	}}             	
\nc{\BBr}{{	\bbb{ \left. \right)}	}}

\nc{\Pa}[2]{ {\lb} #1 {,} #2 {\rb} }				

\nc{\cD}[1]{ \tx{ $$\CD {#1} \endCD $$ }  }		

\nc{\mat} {		\left(		\matrix	}	
\nc{\emat}{		\endmatrix	\right)	}
\nc{\sm} {		\left(		\smallmatrix	}	
\nc{\esm}{		\endsmallmatrix	\right)	}
\nc{\smat} {		\left(		\smallmatrix	}	
\nc{\esmat}{		\endsmallmatrix	\right)	}

\nc{\matr} {		\left[		\matrix	}	
\nc{\ematr}{		\endmatrix	\right]	}
\nc{\smr} {		\left[		\smallmatrix	}	
\nc{\esmr}{		\endsmallmatrix	\right]	}
\nc{\smatr} {		\left[		\smallmatrix	}	
\nc{\esmatr}{		\endsmallmatrix	\right]	}

\nc{\imat} {		\left.		\matrix	}	
\nc{\eimat}{		\endmatrix	\right.	}
\nc{\ism} {		\left.		\smallmatrix	}	
\nc{\eism}{		\endsmallmatrix	\right.	}

\nc{\ca}{		\left\{		\smallmatrix	}	
\nc{\eca}{		\endsmallmatrix	\right\}	}
\nc{\Ca}{		\left\{		\matrix		}	
\nc{\Eca}{		\endmatrix	\right.		}	
\nc{\eCa}{		\endmatrix	\right\}	}	

\nc{\com}{	\begin{diagram}	}
\nc{\ecom}{	  \end{diagram}	}

\nc{\tab}{	\begin{tabular}		}
\nc{\etab}{	\end{tabular}		}	
\nc{\hl}{{	\hline			}}

\nc{\Eq}{	\begin{equation}	}
\nc{\Eeq}{	\end{equation}	}
\nc{\aln}{	\begin{align}	}
\nc{\ealn}{	\end{align}	}

\nc{\Rpart}{	\rc{\thepart}{\Roman{part}}	}
\nc{\Apart}{	\rc{\thepart}{\arabic{part}}	}
		
\nc{\rref}[2]{\ref{#1}.\ref{#2}}

\nc{\pa}[1]{ 	\part{#1}		}

\nc{\se}[1]{ 	\section{\bf#1}		}
\nc{\ses}[1]{ 	\section*{\bf#1}		}
\nc{\sus}{ 	\subsection		}
\nc{\sss}{ 	\subsubsection		}

\nc{\Lem}{ 	\subsection{Lemma}		}
\nc{\lem}{ 	\subsubsection{Lemma}		}
\nc{\slem}{ 	\subsubsection*{Lemma}		}
\nc{\sublem}{ 	\subsubsection{ Sublemma}	}
\nc{\ssublem}{ \subsubsection*{ Sublemma}	}

\nc{\Lemm}{ 	\subsection{Lemma}		}
\nc{\lemm}{ 	\subsubsection{Lemma}		}
\nc{\slemm}{ 	\subsubsection*{Lemma}		}
\nc{\sublemm}{ 	\subsubsection{ Sublemma}	}
\nc{\ssublemm}{ \subsubsection*{ Sublemma}	}

\nc{\Pro}{ 	\subsection{Proposition}	}
\nc{\pro}{ 	\subsubsection{Proposition}	}
\nc{\spro}{ 	\subsubsection*{Proposition}	}

\nc{\Cor}{ 	\subsection{Corollary}		}
\nc{\cor}{ 	\subsubsection{Corollary}	}
\nc{\scor}{ 	\subsubsection*{Corollary}	}

\nc{\Corr}{ 	\subsection{Corollary}		}
\nc{\corr}{ 	\subsubsection{Corollary}	}
\nc{\scorr}{ 	\subsubsection*{Corollary}	}

\nc{\Theo}{ 	\subsection{Theorem}		}		
\nc{\theo}{ 	\subsubsection{Theorem}		}
\nc{\stheo}{ 	\subsubsection*{Theorem}	}

\nc{\pretheo}{ 	\subsubsection{Pretheorem}	}

\nc{\rem}{ 	\subsubsection{Remark}		}
\nc{\srem}{ 	\subsubsection*{Remark}	}

\nc{\rems}{ 	\subsubsection{Remarks}		}

\nc{\srems}{ 	\subsubsection*{Remarks}	}

\nc{\Def}{ 	\subsection{Definition}		}
\nc{\ddef}{ 	\subsubsection{Definition}	}

\nc{\comm}{ 	\subsubsection{Comment}		}
\nc{\scomm}{ 	\subsubsection*{Comment}	}
\nc{\comms}{ 	\subsubsection{Comments}		}
\nc{\scomms}{ 	\subsubsection*{Comments}	}

\nc{\claim}{ 	\subsubsection{Claim}		}
\nc{\sclaim}{ 	\subsubsection*{Claim}	}

\nc{\nota}{ 	\subsubsection{Notation}	}

\nc{\conj}{ 	\subsubsection{Conjecture}	}
\nc{\sconj}{ 	\subsubsection*{Conjecture}	}

\nc{\ex}{ 	\subsubsection{Example}		}
\nc{\sex}{ 	\subsubsection*{Example}	}
\nc{\exs}{ 	\subsubsection{Examples}	}
\nc{\sexs}{ 	\subsubsection*{Examples}	}
\nc{\Ex}{ 	\subsection{Example}		}
\nc{\sEx}{ 	\subsection*{Example}	}
\nc{\Exs}{ 	\subsection{Examples}	}
\nc{\sExs}{ 	\subsection*{Examples}	}

\nc{\que}{ 	\subsubsection{Question}	}
\nc{\ques}{ 	\subsubsection{Questions}	}
\nc{\sque}{ 	\subsubsection*{Question}	}
\nc{\sques}{ 	\subsubsection*{Questions}	}

\nc{\bi}{	\begin{itemize}\item		}
\rc{\i}{	\item			}
\nc{\ei}{ \end{itemize}	} 
\nc{\ben}{	\begin{enumerate}\item		}
\nc{\ftt}[1]{{\footnote{#1}}}
\nc{\fttt}[1]{{$^($\footnote{#1}$^)$}}
\nc{\bftt}[1]{\footnote{#1}}
\nc{\f}[1]{ \fbox{$ $}\footnote{ \fbox{!}#1 }\fbox{$ $}		}

\nc{\Ao}{{	\A^1	}}
\nc{\Po}{{	\P^1	}}
\nc{\So}{{	S^1	}}

\nc{\h}{{	\hslash	}}	

\nc{\All}{{	\forall		}}
\nc{\Exx}{{	\exists 	}}
\nc{\yy}{\infty}                       
\nc{\ys}{{  \frac{\infty}{2}  }}

\nc{\ii}{{i\in I}}
\nc{\ww}{{w\in W}}
\nc{\SES}[5]{{	0 @>>> {#1} @>{#2}>> {#3} @>{#4}>> {#5} @>>> 0	}}
\nc{\Ses}[3] {{	0 @>>> {#1} @>>>     {#2} @>>>     {#3} @>>> 0	}}
\nc{\pl}{{\oplus}}              		
\nc{\tim}{{\times}}             
\nc{\btim}{{\boxtimes}}
\nc{\ltim}{\ltimes}                  	%
\nc{\rtim}{\rtimes}			%
\nc{\ltr}{\triangleleft}        %
\nc{\rtr}{\triangleright}       %

\nc{\ten}{{	\otimes		}}            
\nc{\Lten}{{	\aa{L}\otimes	}}            
\nc{\Ltim}{{	\aa{L}\times	}}            
\nc{\Lcap}{{	\aa{L}\cap	}}            
\nc{\tenA}{	\bb{A}\ten	}
\nc{\tenB}{	\bb{B}\ten	}
\nc{\tenZ}{	\bb{\Z}\ten	}
\nc{\tenR}{	\bb{\R}\ten	}
\nc{\tenC}{	\bb{\C}\ten	}
\nc{\tenk}{	\bb{\k}\ten	}
\nc{\bten}{{\boxtimes}}         		

\nc{\con}{{ @>{\protect\cong}>> }}  	
\nc{\conl}{{ 	@>{\cong}>>	}}  	
\nc{\conn}{{    @<{\cong}<<  	}}  	
\nc{\Con}{{	\equiv		}}	
\nc{\appr}{{	\sim		}}	
\nc{\eqr}{{	\sim		}}	
\nc{\equi}{{	\sim		}}	

\nc{\fra}{ 	\frac	}     	
\nc{\ffr}[2]{{ 	\text{\footnotesize $\frac{#1}{#2}$	}	}}  

\nc{\ha}{{ \frac{1}{2} }}     		
	\nc{\half}{{ \frac{1}{2} }}    	

\nc{\ci}{{\circ}}               
\nc{\cd }{{\cdot}}            	
\nc{\cddd}{{\cdot\cdot\cdot}}	

\nc{\ox}{{	\OO_X		}}               
\nc{\omx}{{	\om_X		}}               
\nc{\Omx}{{	\Om_X^1		}}               
\nc{\Coh}{{	\CC oh		}}               %
\nc{\qcoh}{{	q\CC oh		}}               %

\nc{\xt}{{	X_*(T)		}}
\nc{\Xt}{{	X^*(T)		}}

\nc{\cfm}{{	co\fm		}}	

\nc{\cupp}{\bigcup}             
\nc{\capp}{\bigcap}
\nc{\pll}{\bigoplus}

\nc{\pii}{\prod}                
\nc{\ppii}{\bigprod}            
\nc{\cci}{\sqcup}              
\nc{\ccii}{\bigsqcup}

\nc{\wwe}{\bigwedge}            
\nc{\cce}{\bigcoprod}           

\nc{\aaa}{	\stackerel	}	

\nc{\edd}{{ \end{document}	}} 

\nc{\tx}{	\text		}		
\nc{\df}{{ \protect\overset{ \text{def}}= 	}}		
\nc{\dff}{{ \ \df\				}}		
\nc{\inv}{{ {}^{-1}      }}			

\nc{\thh}{	^{\text{th}}	}                     	
\nc{\st}{	^{\text{st}}	}                     	
\nc{\nd}{	^{\text{nd}}	}                     	
\nc{\rd}{	^{\text{rd}}	}                     	

\nc{\pmo}{{ 	\pm 1		}}
\nc{\mpo}{{ 	\mp 1		}}

\nc{\htt}{  \text{ht}}				

\nc{\emp}{{   \emptyset}}      			

\nc{\cowe}{{	\vee	}}			
\nc{\we}{{\wedge}}				
\nc{\wee}{{	\aa{\bullet}\wedge	}}		
\nc{\wetwo}{{     \pr\overset{2}\wedge       }}	

\nc{\limp}{{	\pr\underset {\leftarrow} \lim		}}	
\nc{\Limp}{{	\pr\underset {\leftarrow} {\bbb\lim}	}}	
\nc{\limi}{{	\pr\underset {\rightarrow}\lim		}}      
\nc{\Limi}{{	\pr\underset {\rightarrow}{\bbb\lim}	}}	
\nc{\llim}[1]{	 \bb{#1}\lim        	}   
\nc{\llimp}[1]{ \bb{#1}{ \pr\underset {\leftarrow} \lim       } }
\nc{\LLimp}[1]{ \bb{#1}{ \pr\underset {\leftarrow} {\bbb\lim} } } 	
\nc{\llimi}[1]{ \bb{#1}{ \pr\underset {\rightarrow}\lim       } }
\nc{\LLimi}[1]{ \bb{#1}{ \pr\underset {\rightarrow}{\bbb\lim} } }	
						
\let\Bbb\mathbb
\nc{\ppn}{{ {\Bbb P}^n }}            		
\nc{\pt}{	{ \text{pt} }	}		
				
\nc{\qlb}{{ \barr{{\Bbb Q}_l} }}      		
\nc{\ffq}{{  {\Bbb F}_q  }}           		
\nc{\ffp}{{  {\Bbb F}_p  }}           		
\nc{\tw}{   {}^{(1)}	}		

\nc{\Ab}{{ 	\AA b 		}}      		%
\nc{\Set}{{ 	\SS et 		}}      		%
\nc{\Top}{{ 	\TT op 		}}      		%
\nc{\Pic}{{ 	\tx{Pic}	}}      		%

\nc{\del}{{\partial }}
\nc{\delb}{{\partial }}
\nc{\dd}[2]{	\fra{d{#1}}{d{#2}}		}
\nc{\ddel}[2]{	\fra{\del{#1}}{\del {#2}}	}
                                        
\nc{\Spec}{{ 	\text{Spec}      		}} 
\nc{\Specf}{{ 	\text{Specf}      		}} 
\nc{\Spf}{{ 	\text{Spf}      		}} 
\nc{\hk}{{     \text{hyperk\"ahler} 	}}
\nc{\susy}{{\text{supersymmetry}}}

\nc{\ie}{{,\ \     \text{i.e.,}\ \ 	}}
\nc{\iif}{{\ \     \text{if}\ \ 	}}
\nc{\aand}{{\ \ \  \text{and}\ \ \ 	}}
\nc{\hence}{{\ \ \ \text{hence}\ \ \ 	}}
\nc{\while}{{\ \ \ \text{while}\ \ \ 	}}
\nc{\with}{{\ \ \  \text{with}\ \ \ 	}}
\nc{\oor}{{\ \     \text{or}\ \ 	}}
\nc{\foor}{{\ \     \text{for}\ \ 	}}
\nc{\suchthat}{{\ \     \text{such that}\ \ 	}}

\nc{\rk}{{\operatorname{rk}}}
\nc{\Ker}{{\operatorname{Ker}}}
\nc{\Coker}{{\operatorname{Coker}}}
\rc{\Im}{{ 	\text{Im} 	}}
\nc{\rank}{{	\ \text{rank} 	}}
\nc{\Res}{{	\  \text{Res}   }}
\nc{\Hom}{{\operatorname{Hom}}}
\nc{\End}{{	\text{End}	}}
\nc{\RHom}{{	\text{RHom}	}}
\nc{\HHom}{{	\text{$\HH$om}	}}
\nc{\RHHom}{{	\text{R$\HH$om} }}
\nc{\RGa}{{	\text{R$\Ga$}	}}
\nc{\EEnd}{{	\text{$\EE nd$}	}}
\nc{\AAut}{{	\text{$\AA ut$}	}}
\nc{\Ext}{{\operatorname{Ext}}}
\nc{\Tor}{{\operatorname{Tor}}}
\nc{\Der}{{	\text{Der}	}}

\nc{\ord	}{{ \text{ord} }}			
\nc{\divv	}{{ \text{div} }}			
\nc{\Lie	}{{ \text{Lie} }}

\nc{\timA} {{   \pr\underset{A}\tim             }}
\nc{\timB} {{   \pr\underset{B}\tim             }}
\nc{\timC} {{   \pr\underset{C}\tim             }}
\nc{\timG} {{   \pr\underset{G}\tim             }}
\nc{\timH} {{   \pr\underset{H}\tim             }}
\nc{\timN} {{   \pr\underset{N}\tim             }}
\nc{\timP}{{    \pr\underset{P}\tim             }}
\nc{\timQ}{{    \pr\underset{Q}\tim             }}
\nc{\timS} {{   \pr\underset{S}\tim             }}
\nc{\timT} {{   \pr\underset{T}\tim             }}
\nc{\timU} {{   \pr\underset{U}\tim             }}
\nc{\timV} {{   \pr\underset{V}\tim             }}
\nc{\timX} {{   \pr\underset{X}\tim             }}
\nc{\timY} {{   \pr\underset{Y}\tim             }}
\nc{\timZ} {{   \pr\underset{Z}\tim             }}

\nc{\ab}{{       ^{\text{ab}}   		}}
\nc{\af}{{       ^{\text{aff}}  		}}

\nc{\cod}{\text{codim}}	

\rc{\AA}{{\cal A}}
\nc{\BB}{{\cal B}} 
\nc{\CC}{{\cal C}}
\nc{\DD}{{\cal D}}
\nc{\EE}{{\cal E}}
\nc{\FF}{{\cal F}}
\nc{\GG}{{\cal G}}
\nc{\HH}{{\cal H}}
\nc{\II}{{\cal I}}
\nc{\JJ}{{\cal J}}
\nc{\KK}{{\cal K}}
\nc{\LL}{{\cal L}}
\nc{\MM}{{\cal M}}
\nc{\NN}{{\cal N}}
\nc{\OO}{{\cal O}}
\nc{\PP}{{\cal P}}
\nc{\QQ}{{\cal Q}}
\nc{\RR}{{\cal R}}
\rc{\SS}{{\cal S}}
\nc{\TT}{{\cal T}}
\nc{\UU}{{\cal U}}
\nc{\VV}{{\cal V}}
\nc{\WW}{{\cal W}}
\nc{\ZZ}{{\cal Z}}
\nc{\XX}{{\cal X}}
\nc{\YY}{{\cal Y}}

\nc{\A}{{\Bbb A }}
\nc{\B}{{\Bbb B}}
\nc{\C}{{\Bbb C}}
		\nc{\cc}{{\Bbb C}}
\nc{\Cs}{{\Bbb C^*}}
		\nc{\cs}{{\Bbb C^*}}
		\nc{\ccs}{{\Bbb C^*}}
\nc{\D}{{\Bbb D}}
\nc{\E}{{\Bbb E}}
\nc{\F}{{\Bbb F}}
\nc{\G}{{\Bbb G}}
	\nc{\hH}{{\Bbb H}}
\nc{\I}{{\Bbb I}}
\nc{\J}{{\Bbb J}}
\nc{\K}{{\Bbb K}}
	\nc{\lL}{{\Bbb L}}
\nc{\M}{{\Bbb M}}
\nc{\N}{{\Bbb N}}
	\nc{\oO}{{\Bbb O}}
	\nc{\pP}{{\Bbb P}}      
\nc{\Q}{{\Bbb Q}}
	\nc{\sS}{{\Bbb S}}
\nc{\T}{{\Bbb T}}
\nc{\U}{{\Bbb U}}
\nc{\V}{{\Bbb V}}
\nc{\W}{{\Bbb W}}
\nc{\Z}{{\Bbb Z}}
\nc{\X}{{\Bbb X}}
\nc{\Y}{{\Bbb Y}}

\let\P\pP

\nc{\fA}{{\frak A}}
\nc{\fB}{{\frak B}}
\nc{\fC}{{\frak C}}
\nc{\fD}{{\frak D}}
\nc{\fE}{{\frak E}}
\nc{\fF}{{\frak F}}
\nc{\fG}{{\frak G}}
\nc{\fH}{{\frak H}}
\nc{\fI}{{\frak I}}
\nc{\fJ}{{\frak J}}
\nc{\fK}{{\frak K}}
\nc{\fL}{{\frak L}}
\nc{\fM}{{\frak M}}
\nc{\fN}{{\frak N}}
\nc{\fO}{{\frak O}}
\nc{\fP}{{\frak P}}
\nc{\fQ}{{\frak Q}}
\nc{\fR}{{\frak R}}
\nc{\fS}{{\frak S}}
\nc{\fT}{{\frak T}}
\nc{\fU}{{\frak U}}
\nc{\fV}{{\frak V}}
\nc{\fW}{{\frak W}}
\nc{\fZ}{{\frak Z}}
\nc{\fX}{{\frak X}}
\nc{\fY}{{\frak Y}}
\nc{\fa}{{\frak a}}
\nc{\fb}{{\frak b}}
\nc{\fc}{{\frak c}}
\nc{\fd}{{\frak d}}
\nc{\fe}{{\frak e}}
\nc{\ff}{{\frak f}}
\nc{\fg}{{\frak g}}
\nc{\fh}{{\frak h}}
\nc{\fiI}{{\frak i}}  
	\nc{\ffi}{{\frak i}}  
\nc{\fj}{{\frak j}}
\nc{\fk}{{\frak k}}
\nc{\fl}{{\frak{l}}}
\nc{\fm}{{\frak m}}
\nc{\fn}{{\frak n}}
\nc{\fo}{{\frak o}}
\nc{\fp}{{\frak p}}
\nc{\fq}{{\frak q}}
\nc{\fr}{{\frak r}}
\nc{\fs}{{\frak s}}
\nc{\ft}{{\frak t}}
\nc{\fu}{{\frak u}}
\nc{\fv}{{\frak v}}
\nc{\fw}{{\frak w}}
\nc{\fz}{{\frak z}}
\nc{\fx}{{\frak x}}
\nc{\fy}{{\frak y}}

\nc{\al}{{\alpha }}
\nc{\be}{{\beta }}
\nc{\ga}{{\gamma }}
\nc{\de}{{\delta }}
\nc{\ep}{{\varepsilon }}
\nc{\vap}{{\epsilon }}

\nc{\ze}{{\zeta }}
\nc{\et}{{\eta }}
\rc{\th}{{\theta }}
\nc{\vth}{{\vartheta }}

\nc{\io}{{\iota }}
\nc{\ka}{{\kappa }}
\nc{\la}{{\lambda }}
\nc{\vpi}{{	\varpi		}}
\nc{\vrho}{{	\varrho		}}
\nc{\si}{{	\sigma 		}}
\nc{\ups}{{	\upsilon 	}}
\nc{\vphi}{{	\varphi 	}}
\nc{\om}{{	\omega 		}}

\nc{\Ga}{{\Gamma }}
\nc{\De}{{\Delta }}
\nc{\nab}{{\nabla}}
\nc{\na}{{\nabla}}
\nc{\Th}{{\Theta }}
\nc{\La}{{\Lambda }}
\nc{\Si}{{\Sigma }}
\nc{\Ups}{{\Upsilon }}
\nc{\Om}{{\Omega }}

\nc{\Aa}{{	\text{A}	}}
\nc{\Bb}{{	\text{B}	}}
\nc{\Cc}{{	\text{C}	}}
\nc{\Dd}{{	\text{D}	}}
\nc{\Ee}{{	\text{E}	}}
\nc{\Ff}{{	\text{F}	}}
\nc{\Gg}{{	\text{G}	}}
\nc{\Hh}{{	\text{H}	}}
\nc{\Ii}{{	\text{I}	}}
\nc{\Jj}{{	\text{J}	}}
\nc{\Kk}{{	\text{K}	}}
\nc{\Ll}{{	\text{L}	}}
\nc{\Mm}{{	\text{M}	}}
\nc{\Nn}{{	\text{N}	}}
\nc{\Oo}{{	\text{O}	}}
\nc{\Pp}{{	\text{P}	}}
\nc{\Qq}{{	\text{Q}	}}
\nc{\Rr}{{	\text{R}	}}
\nc{\Ss}{{	\text{S}	}}
\nc{\Tt}{{	\text{T}	}}
\nc{\Uu}{{	\text{U}	}}
\nc{\Vv}{{	\text{V}	}}
\nc{\Ww}{{	\text{W}	}}
\nc{\Zz}{{	\text{Z}	}}
\nc{\Xx}{{	\text{X}	}}
\nc{\Yy}{{	\text{Y}	}}

\nc{\bGa}{{	\bbb{\Ga}	}}
\nc{\bA}{{	\bbb{A}		}}
\nc{\bB}{{	\bbb{B}		}}
\nc{\bC}{{	\bbb{C}		}}
\nc{\bD}{{	\bbb{D}		}}
\nc{\bE}{{	\bbb{E}	}}
\nc{\bF}{{	\bbb{F}	}}
\nc{\bG}{{	\bbb{G}	}}
\nc{\bH}{{	\bbb{H}	}}
\nc{\bI}{{	\bbb{I}	}}
\nc{\bJ}{{	\bbb{J}	}}
\nc{\bK}{{	\bbb{K}	}}
\nc{\bL}{{	\bbb{L}	}}
\nc{\bM}{{	\bbb{M}	}}
\nc{\bN}{{	\bbb{N}	}}
\nc{\bO}{{	\bbb{O}	}}
\nc{\bP}{{	\bbb{P}	}}
\nc{\bQ}{{	\bbb{Q}	}}
\nc{\bR}{{	\bbb{R}	}}
\nc{\bS}{{	\bbb{S}	}}
\nc{\bT}{{	\bbb{T}	}}
\nc{\bU}{{	\bbb{U}	}}
\nc{\bV}{{	\bbb{V}	}}
\nc{\bW}{{	\bbb{W}	}}
\nc{\bX}{{	\bbb{X}	}}
\nc{\bY}{{	\bbb{Y}	}}
\nc{\bZ}{{	\bbb{Z}	}}
\nc{\ba}{{	\bbb{a}	}}
			\nc{\bbbb}{{	\bbb{b}	}}
\nc{\bc}{{	\bbb{c}	}}
\nc{\bd}{{	\bbb{d}	}}
			\nc{\bbe}{{	\bbb{e}	}}
			\nc{\bbf}{{	\bbb{f}	}}
\nc{\bg}{{	\bbb{g}	}}
\nc{\bh}{{	\bbb{h}	}}
			\nc{\bbi}{{	\bbb{i}	}}
\nc{\bj}{{	\bbb{j}	}}
			\nc{\bbk}{{	\bbb{k}	}}
\nc{\bl}{{	\bbb{l}	}}
\nc{\bm}{{	\bbb{m}	}}
\nc{\bn}{{	\bbb{n}	}}
\nc{\bo}{{	\bbb{o}	}}
\nc{\bp}{{	\bbb{p}	}}
\nc{\bq}{{	\bbb{q}	}}
\nc{\br}{{	\bbb{r}	}}
\nc{\bs}{{	\bbb{s}	}}
\nc{\bt}{{	\bbb{t}	}}
			\nc{\bbbu}{{	\bbb{u}	}}
\nc{\bv}{{	\bbb{v}	}}
\nc{\bw}{{	\bbb{w}	}}
\nc{\bxx}{{	\bbb{x}	}}
\nc{\by}{{	\bbb{y}	}}
\nc{\bz}{{	\bbb{z}	}}

\nc{\sA}{{\mathsf A}}
\nc{\sB}{{\mathsf B}}
\nc{\sC}{{\mathsf C}}
\nc{\sD}{{\mathsf D}}
\nc{\sE}{{\mathsf E}}
\nc{\sF}{{\mathsf F}}
\nc{\sG}{{\mathsf G}}
\nc{\sH}{{\mathsf H}}
\nc{\sI}{{\mathsf I}}
\nc{\sJ}{{\mathsf J}}
\nc{\sK}{{\mathsf K}}
\nc{\sL}{{\mathsf L}}
\nc{\sM}{{\mathsf M}}
\nc{\sN}{{\mathsf N}}
\nc{\sO}{{\mathsf O}}
\nc{\sP}{{\mathsf P}}
\nc{\sQ}{{\mathsf Q}}
\nc{\sR}{{\mathsf R}}
\rc{\sS}{{\mathsf S}}
\nc{\sT}{{\mathsf T}}
\nc{\sU}{{\mathsf U}}
\nc{\sV}{{\mathsf V}}
\nc{\sW}{{\mathsf W}}
\nc{\sX}{{\mathsf X}}
\nc{\sY}{{\mathsf Y}}
\nc{\sZ}{{\mathsf R}}
\nc{\sa}{{\mathsf a}}
\rc{\sb}{{\mathsf b}}
\rc{\sc}{{\mathsf c}}
\nc{\sd}{{\mathsf d}}
\nc{\sg}{{\mathsf g}}
\nc{\sh}{{\mathsf h}}
\nc{\sj}{{\mathsf j}}
\nc{\sk}{{\mathsf k}}
\nc{\sn}{{\mathsf n}}
\nc{\so}{{\mathsf o}}
\nc{\sq}{{\mathsf q}}
\nc{\sr}{{\mathsf r}}
\nc{\su}{{\mathsf u}}
\nc{\sv}{{\mathsf v}}
\nc{\sw}{{\mathsf w}}
\nc{\sx}{{\mathsf x}}
\nc{\sy}{{\mathsf y}}
\nc{\sz}{{\mathsf z}}

\nc{\toc}{{ 	\small{\tableofcontents} }}
\nc{\addl}{	\addcontentsline{toc}{subsection}	}

\nc{\all}{{ ^{(\alpha)} }}
\nc{\bee}{{ ^{(\beta)} }}
\nc{\gaa}{{ ^{(\gamma)} }}
\nc{\nnnn}{{ ^{ ( n ) } }}     
\nc{\nnn}{{ ^{ [ n ] } }}     
\nc{\GK}{{  	G(\KK)		}}
\nc{\GO}{{  	G(\OO)		}}

\nc{\Kh}{\tx{Ka\"hler\ }}
\nc{\Khs}{\tx{Ka\"hler structure\ }}
\nc{\Khss}{\tx{Ka\"hler structures\ }}

\nc{\GKh}{\tx{Generalized Ka\"hler\ }}
\nc{\GKs}{\tx{Generalized Ka\"hler structure\ }}
\nc{\GKss}{\tx{Generalized Ka\"hler structures\ }}
\nc{\gKs}{\tx{Generalized Ka\"hler structure\ }}
\nc{\gKss}{\tx{Generalized Ka\"hler structures\ }}

\nc{\sYM}{\text{super Young-Mills\ }}
\nc{\tFT}{\text{topological Field Theory\ }}
\rc{\top}{\tx{topological\ }}
\rc{\Top}{\tx{Topological\ }}
\nc{\TFT}{\text{Topological Field Theory\ }}
\nc{\TQFT}{\text{Topological Quantum Field Theory\ }}
\nc{\TQFTs}{\text{Topological Quantum Field Theories\ }}
\nc{\QFT}{\text{Quantum Field Theory\ }}
\nc{\QFTs}{\text{Quantum Field Theories\ }}
\nc{\FT}{\text{Field Theory\ }}

\nc{\HM}{\text{Hitchin moduli\ }}
\nc{\Hf}{\text{Hitchin fibration\ }}
\nc{\Wi}{\text{Wilson\ }}
\nc{\Wo}{\text{Wilson operator\ }}
\nc{\Wos}{\text{Wilson operators\ }}
\nc{\tH}{\text{t'Hooft\ }}
\nc{\tHo}{\text{t'Hooft operator\ }}
\nc{\Ho}{\text{t'Hooft operator\ }}
\nc{\tHos}{\text{t'Hooft operators\ }}
\nc{\Hos}{\text{t'Hooft operators\ }}
\nc{\Sd}{\text{S-duality\ }}
\nc{\Ld}{\text{Langlands duality\ }}
\nc{\Be}{\text{Bogomolny equations\ }}
\rc{\d}{\text{duality\ }}
\nc{\He}{\text{Hecke\ }}
\nc{\Heo}{\text{Hecke operators\ }}
\nc{\Hem}{\text{Hecke modifications\ }}
\nc{\Bgs}{\tx{Bogomolny equations\ }}
\nc{\Bg}{\tx{Bogomolny equation\ }}
\nc{\Hk}{{\text{Hyperk$\ddot{a}$hler} }}

\nc{\eqq}{{ 	\ =\			}}
\nc{\Cy}{{ 	C_\yy	}}
\nc{\Ay}{{ 	A_\yy	}}
\nc{\Ly}{{ 	L_\yy	}}
\nc{\Cm}{{ 	C_m	}}
\nc{\Am}{{ 	A_m	}}
\nc{\Lm}{{ 	L_m	}}

\rc{\Cy}{{	Cat_\yy			}}
\nc{\Cey}{{	Cat^{ex}_\yy		}}
\nc{\Cyt}{{	\tii{Cat}_\yy		}}

\nc{\Gm}{{	G_m			}}

\nc{\Uo}{{	U(1)	}}
\nc{\sqt}{{	\sqrt{2}	}}

\nc{\BGB}{{B\bss G/B}}

\nc{\Glb}{{ \barr{\GG_\la}	}}

\nc{\mut}{{	\mu_2	}}

\nc{\dx}{{	\dot x		}}
\nc{\ddx}{{	\ddot x		}}

\nc{\dy}{{	\dot y		}}
\nc{\ddy}{{	\ddot y		}}

\nc{\du}{{	\dot u		}}
\nc{\ddu}{{	\ddot u		}}

\nc{\Cyy}{	C^\yy		}

\nc{\hh}{{	\hatt\fh		}}
\nc{\hhp}{{	\hatt\fh_+	}}
\nc{\hhm}{{	\hatt\fh_-	}}
\nc{\jh}{{	J_{\fh}		}}
\nc{\jhs}{{	J_{\fh}		}}
\nc{\negg}{{ _{<0} 	}}
\nc{\pp}{{ _{>0} 	}}

\nc{\zb}{{\bar z}}
\let\d\del
\nc{\dbar}{{\bar \del}}

\nc{\p}{\psi}
\nc{\pb}{{	\bar\psi	}}

\nc{\pd}{{	\dot\psi	}}
\nc{\pbd}{{	\dot{\bar\psi}	}}

\nc{\pos}{{	\tx{\tiny{po}}	}}
\nc{\mom}{{	\tx{\tiny{mo}}	}}

\nc{\vac}{{	\tx{\tiny{vac}}	}}

\nc{\mb}[1]{{	\mbox{$#1$}	}}

\nc{\GGG}{{	\bbb \GG	}}

\nc{\cdG}{{	\bb{G}\cd	}}
\nc{\cdGs}{{	\bb{G^*}\cd	}}

\rc{\l}{{	\bbb l		}}
\nc{\pms}{{	\{\pm\}	}}

\nc{\yh}{{	\hatt\yy	}}
\nc{\hy}{{	\hatt\yy	}}

\nc{\bpl}{{\boxplus}}
\nc{\bcd}{{\boxdot}}

\nc{\gau}{{	e^{-x^2/2}		 }}
\nc{\istp}{{	\fra{ 1 }{ \sqrt{2\pi} } }}
\nc{\stp}{{	\sqrt{2\pi}		 }}

\nc{\lrb}[2]{	\lb #1,#2\rb	}
\nc{\lrbb}[2]{	\lb #1|#2\rb	}

\rc{\sq}{{	\sqcup		}}

\nc{\Qg}{{	\Q_{\ge 0}	}}
\nc{\Rg}{{	\R_{\ge 0}	}}
\nc{\Qgg}{{\Q_{> 0}	}}
\nc{\Rgg}{{\R_{> 0}	}}

\nc{\dagg}{{\dagger}}

\nc{\raa}[1]{\xrightarrow{#1}}
\rc{\laa}[1]{\xleftarrow{#1} }

\begin{document}





\title[]{
Enriched Sets and Higher Categories
}

\author{	Bradley M. Willocks		}


\maketitle

\begin{abstract}                
We introduce the notion of an enriched set, as an abstraction of enriched categories, and a category of enriched sets. The set of enriched sets is itself described as a set enriched over the category of enriched sets. We introduce a method for the construction of sets enriched over the set of enriched sets from a given enriched set with some addition data, and for ``functors" from such enriched sets as should thereby arise to the enriched set of enriched sets.
\end{abstract}

\toc


\se{Introduction}


The present work is grown of a desire for a systematic description of methods by which one might reconfigure spaces of one type into spaces of another. To this end, we introduce in the present work "enriched sets," abstractions of categories, and a formalism by which they may be reconfigured into related enriched sets (``constellations") whose arrows are ``diagrams" in the original enriched set.  We furthermore construct (\ref{Lens}) ``functors" from sub-enriched sets of the reconfigured sets to the enriched set of enriched sets, whereby, in loose terms, ``an arrow in the reconfigured enriched set is sent to a functor from an arrow category over the domain to an arrow category over the codomain." 

In particular, we've constructed in other work a category of pointed categories and pointed correspondences (\cite{Wi}, the ``purely categorical" part of which forms the core of this work), and a formalism for the description of a category of ``locally affine sheaved spaces" as a subcategory thereof. This category is intended as a domain in which such categories (``geometries") might be extended and compared. With such geometries contained within a single category, and a plurality of arrows between them (of ``non-classical origin") one might construct various arrow categories, whose objects were arrows between such categories. The intention is that each object $x$ within a geometry should be attached to such an inter-geometric arrow category, and an arrow $x \rightarrow y$ (a morphism within a particular geometry) might be reconsidered as situated within a larger diagram (a ``constellation") in which the arrow categories of $x$ and $y$ might be mingled, depicting, for example, fibred products $u \times_{y} x$, with $u$ originating from the ``other" geometry (e.g. logarithmic structure as in \cite{Ab}, or divided power structure \cite{BO}; generally, alternate algebraic structure). Such an association would suggest associating to each $x$ a limit or colimit of the objects $u$ (by the proposition (\ref{FunctSys}) below). At the same time, one might associate to each $x$ the automorphisms $Aut(F)$ of a forgetful functor $((u \rightarrow x) \mapsto u)$, thinking of the Tannakian formalism of \cite{Del}.  It is hoped that such constructions might be useful in sytematically understanding the relationships between classical algebraic geometry of \cite{Hart}, the various $\mathbb{F}_{1}$ geometries described in \cite{LPL} (see the paths and bridges section), Berkovich spaces/non-archimedean geometry of \cite{Ta} or \cite{Con}, and spaces with modified structure sheaves as in \cite{Ab} or \cite{CFKap}.

Our intention is, that this work should constitute the categorical foundation for processes by which such geometries might be attached to, or subsumed within, enriched categories of diagrams (constellations) constituted possibly of objects and arrows from different geometries, from which they might inherit higher categorical structures (see (\ref{remNewEnrich}) or (\ref{Lens})) and homotopy invariants (from the Tannakian inspiration). Having described, in the other work, \cite{Wi}, a common category for a  somewhat general notion of geometry, we would inquire into the ``possibilities" regarding homotopy and cohomology theories, hoping, in particular, to extend such notions to alternate geometries, and to compare their different manifestations (we would hope for something like GAGA, \cite{Se}). 

Thinking of (co)homology, one meets with their plurality, and we are perhaps therefore inclined toward some adaptation of motivic cohomology, or some other formalism by which categories of ``schemes" or sheaves of some type (rings classically) might be enveloped by Abelian categories with translations, employing the $\kappa$-twists of \cite{Wi} to replace distinguished objects in $Sh(X,\mathfrak{Ring})$ by distinguished objects in some derived category. We hope in the future, based upon the present work (\ref{remNewEnrich}), to relate the higher categorical/homotopy data attached to a geometry to the Abelian data attached thereto (to adapt motivic cohomology of \cite{Vo} or \cite{Lev} to alternate geometries, using \cite{LurSt} to study the result).

$\bold{Notation}$. For that this work is mostly concerned with the consideration of composition laws and their variations, we denote by ``$f\cdot g$" the composition of functions $f$ with $g$, so that $f \cdot g : x \mapsto f(g(x))$. Brackets ``$\ulcorner$" and ``$\urcorner$" separate logical statements, where they are much manipulated. Otherwise, notation and general concepts are those of standard category theory (\cite{ML}, \cite{KS}).

\se{Enriched Sets}

We introduce the notion of sets enriched over a tensor category $(A,\otimes)$, which is an abstraction of that of a category, consisting essentially of compositon laws, which assign to each triple $(a,b,c) \in S$ of elements in a set an arrow in $A$, $\circ(a,b,c) : h(a,b) \otimes h(b,c) \rightarrow h(a,c)$. Out approach differs from that of \cite{Dub} primarily in the use of an extra datum, a ``skeleton functor" $sk : A \longrightarrow B$, to replace equality with ``equivalence."

\sus{A Variation on Limits ($(sk,e)$-limits)}

We define a notion of a limit of a functor $F : I \longrightarrow A$, with respect to a ``skeleton" functor $sk : A \longrightarrow B$, as the colimit of a domain functor, from a certain arrow category in a fibre product of a pair of categories of functors to $A$. The construction essentially takes the terminal object in the category of objects over the functor $F$ which are natural after the application of $sk$. 

\sss{The Use of $dob\downarrow_{(-)}$} Recall that $\Delta_{(J,A)} : A \longrightarrow Hom_{U-\mathfrak{Cat}^{2}(0)}(J,A)$ by sending an object $c$ to the $c$-valued constant functor, and $ob_{(Hom_{U-\mathfrak{Cat}^{2}(0)}(J,A))}(F \cdot e) : \star \longrightarrow Hom_{U-\mathfrak{Cat}^{2}(0)}(J,A)$ by sending the one arrow in $\star$ to $id_{F \cdot e}$. Recall also that the category $\downarrow_{(Hom_{U-\mathfrak{Cat}^{2}(0)}(J,A))} (\Delta_{(J,A)},ob_{Hom_{U-\mathfrak{Cat}^{2}(0)}(J,A)}(F\cdot e))$ of arrows is defined so that its objects are triples $(a,\alpha,\emptyset )$, where $a \in Ob(A)$, $\alpha : \Delta_{(J,A)}(a) \rightarrow F \cdot e$ is a natural transformation, and $\emptyset$ is the object in the category $\star$ (the category with one arrow). An arrow between $(a_{1},\alpha_{1},\emptyset ) \xrightarrow{(\phi,id_{\emptyset})} (a_{2},\alpha_{2},\emptyset )$ is a pair of arrows $(( a_{1} \xrightarrow{\phi} a_{2}) , id_{\emptyset} ) \in Arr(A) \times Arr(\star)$ for which $\alpha_{2} \cdot \Delta_{(J,A)(1)}(\phi) = \alpha_{1}$. 

An isomorphic category is given by forgetting both the $\emptyset$ symbol, and the $a$ term (since for any $j \in Ob(J)$, $a = dom(\alpha(j))$, so that $a$ is determined by $\alpha$), so that its objects are natural transformations $\alpha$, where $a \in Ob(A)$ and $\alpha : \Delta_{(J,A)}(a) \rightarrow F \cdot e$. If $\alpha_{1} : \Delta_{(J,A)}(a_{1}) \rightarrow F \cdot e$ and $\alpha_{2} : \Delta_{(J,A)}(a_{2}) \rightarrow F \cdot e$, then an arrow $\alpha_{1} \xrightarrow{\phi} \alpha_{2}$ is an arrow $( a_{1} \xrightarrow{\phi} a_{2}) \in Arr(A)$ for which $\alpha_{2} \cdot \Delta_{(J,A)(1)}(\phi) = \alpha_{1}$

\sss{$\bold{Definition}$ of an $(sk,e)$-Limit} \label{SkLim}

Consider functors
$J \xrightarrow{e} I \xrightarrow{F} A \xrightarrow{sk} B$. 

Consider the set of maps $\alpha : Ob(I) \rightarrow Arr(A)$ such that $sk \cdot \alpha$ defines a natural transformation from a diagonal functor to $sk\cdot F$. Let
$$
\mathcal{C} :_{t}= \downarrow_{(Hom_{U-\mathfrak{Cat}^{2}(0)}^{(1)}(I,B))} ( \Delta_{(I,B)} , ob_{(Hom_{U-\mathfrak{Cat}^{2}}^{(1)}( I,B ))}(sk\cdot F) )
$$

and
$$
\mathcal{D} :_{t}= \downarrow_{(Hom_{U-\mathfrak{Cat}^{2}(0)}^{(1)}(J,A))} ( \Delta_{(J,A)} , ob_{(Hom_{U-\mathfrak{Cat}^{2}}^{(1)}( J,A ))}(F\cdot e) )
$$

and
$$
\mathcal{E} :_{t}= \downarrow_{(Hom_{U-\mathfrak{Cat}^{2}(0)}^{(1)}(J,B))} ( \Delta_{(J,B)} , ob_{(Hom_{U-\mathfrak{Cat}^{2}}^{(1)}( J,B ))}(sk\cdot F\cdot e) )
$$

Let $\varepsilon : P :_{t}= \mathcal{C} \times_{\mathcal{E}} \mathcal{D} \longrightarrow \mathcal{D}$ be one of the arrows of a fibred product, an arrow in $U'-\mathfrak{Cat}$. If $For$ is the functor which takes the object $a$ from an arrow $\Delta_{(J,A)}(a) \longrightarrow F \cdot e$, an $(sk,e)$-limit is a colimit of $For \cdot \varepsilon
$. 

This is explained in the following sections.

\ref{SkLim}.1. Let $P$ be the full sub-category of the category 
$$
\downarrow_{(Hom_{U-\mathfrak{Cat}^{2}(0)}(J,A))} (\Delta_{(J,A)},ob_{Hom_{U-\mathfrak{Cat}^{2}(0)}(J,A)}(F\cdot e)) \subseteq Hom_{U-\mathfrak{Cat}^{2}(0)}(J,A)_{/ F \cdot e}
$$

whose objects are natural transformations $\alpha$, such that for some $a \in Ob(A)$ we have $\alpha : \Delta_{(J,A)}(a) \rightarrow F \cdot e$, such that there exists a natural transformation $\tilde{\alpha} : \Delta_{(I,B)}(sk(a))\rightarrow sk\cdot F$ such that the natural transformation $Hom_{U-\mathfrak{Cat}^{2}(1)}(id_{J},sk)(\alpha):  \Delta_{(J,B)}(sk(a)) \rightarrow sk \cdot F \cdot e$ given by sending $j \in Ob(J)$ to $sk(\alpha (j)) : sk(a) \rightarrow sk(F(e(j)))$ is equal to the natural transformation $Hom_{U-\mathfrak{Cat}^{2}(1)}(e,id_{B})(\tilde{\alpha}) : \Delta_{(J,B)}(sk(a)) \rightarrow sk \cdot F \cdot e$ given by sending $j \in Ob(J)$ to $\tilde{\alpha}(e(j)) : sk(a) \rightarrow sk(F(e(j)))$.

\ref{SkLim}.2. Denote by $\varepsilon : P \longrightarrow \downarrow_{(Hom_{U-\mathfrak{Cat}^{2}(0)}(J,A))}(\Delta_{(J,A)},ob_{(Hom_{U-\mathfrak{Cat}^{2}(0)}(J,A))}(F\cdot e))$ the inclusion, given by $\alpha \mapsto (a,\alpha,\emptyset )$. Denote also by $p$ the functor, 
$$
p :_{t}= dob\downarrow_{(Hom_{U-\mathfrak{Cat}^{2}(0)}(J,A))}(\Delta_{(J,A)},ob_{(Hom_{U-\mathfrak{Cat}^{2}(0)}(J,A))}(F \cdot e) \cdot \varepsilon : P \longrightarrow A.
$$

Thus, $p$ is given by sending $\alpha \mapsto a$, and the fibre of $p$ over any given $a \in Ob(A)$ is the set of natural transformations $\Delta_{(J,A)}(a) \xrightarrow{\alpha} F\cdot e $ such that for some natural transformation $\tilde{\alpha} : sk\cdot \Delta_{(I,A)}(a) = \Delta_{(I,B)}(sk(a)) \longrightarrow sk \cdot F$, one has
$$
Hom_{U-\mathfrak{Cat}^{2}(1)}(id_{B},e)(\tilde{\alpha}) = Hom_{U-\mathfrak{Cat}^{2}(1)}(id_{J},sk)(\alpha)
.$$

In other words, 
$$
P \subseteq \downarrow_{(Hom_{U-\mathfrak{Cat}^{2}(0}(J,A))}(\Delta_{(J,A)},ob_{(Hom_{U-\mathfrak{Cat}^{2}(0)}(J,A))}(F \cdot e)
$$ 

is the full subcategory which contains all objects $\alpha$ such that the image of $\alpha$ in \newline $Hom_{U-\mathfrak{Cat}^{2}(0)}(J,B)$ under the functor $Hom_{U-\mathfrak{Cat}^{2}(1)}(id_{J},sk)$ has a lift to \newline $Hom_{U-\mathfrak{Cat}^{2}(0)}(I,B)$ by the functor $Hom_{U-\mathfrak{Cat}^{2}(1)}(e,id_{B})$ (we denote this lift by $\tilde{\alpha}$).


\ref{SkLim}.3.
Then the $(sk,e)$-limit of $F$ is the colimit of $p$, i.e., for any pair $(l,\lambda) \in Ob(A) \times Arr(Hom_{U-\mathfrak{Cat}^{2}(0)}(P,A))$ for which $\lambda : p \rightarrow \Delta_{(P,A)}(l)$, we say that $(l,\lambda)$ is an $(sk)-limit(F)$ iff $(l,\lambda)$ is a $colimit(p)$ (in the sense in which $(\lambda,l)$ is a universal arrow in $Hom_{U-\mathfrak{Cat}^{2}(0)}(P,A)$, from $p$ to the constant functor $\Delta_{(P,A)} : A \longrightarrow Hom_{U-\mathfrak{Cat}^{2}(0)}(P,A)$).







 





\begin{ex}
In the above, if either $e$ or $sk$ is an identity functor, then the $(sk)$-limit of $F$ is the limit $(l,\lambda)$ of $F$, if the latter exists. 

If $e = id_{I}$, then $P$ consists of all natural transformations $\alpha : \Delta_{(I,A)}(a) \rightarrow F \cdot e = F$ for which there exists some lift $\tilde{\alpha} : \Delta_{(I,B)}(sk(a)) \rightarrow sk \cdot F$. But $\tilde{\alpha} = Hom_{U-\mathfrak{Cat}^{2}(1)}(sk,id_{I})(\alpha)$ would be such a lift. Therefore any $\alpha : \Delta_{(I,A)}(a) \rightarrow F$ has a lift. Furthermore, for any $i \in Ob(I)$, $\alpha(i) : a \rightarrow F(i)$, and for any $\beta : \Delta_{(I,A)}(b) \rightarrow F$, and any arrow $\phi : \alpha \rightarrow \beta$ in $P$, by definition of $P$, we have that $\beta(i) \cdot \phi = \alpha(i)$. Therefore, by the definition of a colimit, there exists a unique $\alpha_{l}(i) : l \rightarrow F(i)$ such that for any $(\Delta_{(I,A)}(a) \xrightarrow{\alpha} F) \in Ob(P)$, $\alpha(i) = \alpha_{l}(i) \cdot \lambda(\alpha)$. Since each colimit arrow $\alpha_{l}(i)$ is determined by the arrows $\alpha (i)$ which come from natural transformations $\alpha$, the assignment $\alpha_{l} = (i \mapsto \alpha_{l}(i))_{i \in Ob(I)}$ determines a natural transformation $\Delta_{(I,A)}(l) \rightarrow F$. Therefore $\alpha_{l} \in Ob(P)$. If the limit of $F$ exists, then it is isomorphic to a terminal object in $P$, $\alpha_{t} \in Ob(P)$. But by the above argument, this terminal object $\alpha_{t}$ determines a colimit arrow $\lambda (\alpha_{t}) : a_{t} \rightarrow l$, and being a terminal object in $P$ there is a unique arrow $e_{l} : l \rightarrow a_{t} = dom(\alpha_{t}(i))$ in $P$. By the definition of terminal objects, $e_{l} \cdot \lambda(\alpha_{t}) = id_{a_{t}}$ 
\end{ex}

\begin{lem}
(Inclusion, via right exactness) Given $sk, F, e : J \rightarrow I$, $\varepsilon : P \rightarrow$

$\downarrow_{(Hom_{(U-\mathfrak{Cat}^{2})(0)} (J, A) )}(\Delta_{(J,A)}, ob_{(Hom_{(U-\mathfrak{Cat}^{2})(0)}(J,A) )} (F \cdot e) )$, $l$, and $\lambda$ as above, suppose further that $sk$ is right exact, and $Ob(I) = Ob(J)$. For each $i \in Ob(I) = Ob(J)$, consider the arrow induced from the colimit $l$ to $F(i)$ by $\alpha \mapsto \alpha (i)$, where $(p,\alpha,\emptyset) \in Ob(P)$ is an object in $P$. Then this assignment determines an object $(l,\alpha_{l} , \emptyset ) \in Ob(P)$.

I.e. the $(sk)$-limit determines an object,
$$
(l, (j \mapsto \lambda ( ( (b,f,\emptyset) \mapsto f(j) )_{(b,f,\emptyset) \in Ob(P)} ) )_{j \in Ob(J)}, \emptyset ) \in Ob(P)
$$

in $P$.
\end{lem}
\begin{proof}
If the colimit $(l,\lambda)$ is sent to the colimit of the forward composition by $sk$ of the $dob\downarrow$ diagram on $P$, then arrows from $l$ to it are yet determined by their pullbacks to the components of the forward composition of the colimit diagram, which commute after forward composition.
\end{proof}

\begin{lem}
(Uniqueness, via monic) For any $sk : A \rightarrow B, F : I \rightarrow A \in Arr(U-\mathfrak{Cat})$, for any $(l, \lambda) \in Ob(A) \times Arr( Hom_{U-\mathfrak{Cat}^{2}(0)} (P,A) )$, ($P$ being as above) if the arrow from the $(sk)$-limit$(F)$ to the product $\prod_{j\in J} F(j)$ induced by the arrows from the previous lemma, i.e. $\lambda_{\prod}( (j \mapsto \lambda ( ((b,f,\emptyset) \mapsto f(j) )_{(b,f,\emptyset) \in Ob(P)}) ) )_{j \in Ob(J)} ) ) : l \rightarrow \prod_{j \in Ob(J)} F_{(0)}\cdot e_{(0)} (j)$, is monic, then for each $(b,f,\emptyset) \in Ob(P)$, the coproduct arrow $b \rightarrow l$ is the unique arrow $\lambda'$ for which $\lambda ( ( (b,f,\emptyset) \mapsto f(j) )_{(b,f,\emptyset) \in Ob(P)} ) \cdot \lambda' = f(j)$.
\end{lem}
\begin{proof}
Trivial.
\end{proof}

\begin{rem}
This is the uniqueness of factorization usually associated to limits. 
\end{rem}

\sss{$\bold{Definition}$ of the Skeleton Functor}\label{SCat}
Define the category $U-\mathfrak{SCat} \in Ob(U'-\mathfrak{Cat})$ so that
$$
Ob(U-\mathfrak{SCat}) = Ob(U-\mathfrak{Cat})
$$ 

and for any $x,y \in Ob(U-\mathfrak{SCat})$, $Hom_{(U-\mathfrak{SCat})} (x,y)$ is the the set
$$
Hom_{(U-\mathfrak{SCat})} (x,y) := \{ [F] \subseteq Ob(Hom_{U-\mathfrak{Cat}^{2}(0)}(x,y)); F \in Hom_{U-\mathfrak{Cat}}(x,y) \}
$$ 

of isomorphism classes of functors $x \xrightarrow{F} y$, where $[F] = [G]$ iff $F \cong G$, i.e. iff there exists an isomorphism $F \xrightarrow{\alpha} G$ of functors.

Define the functor 
$$
Skel : U-\mathfrak{Cat} \rightarrow U-\mathfrak{SCat}
$$

so that $Skel$ is the identity map on the objects and the quotient map $F \mapsto [F]$ on the arrows.

\begin{ex}
Consider $\phi,\psi \in Arr(U-\mathfrak{Cat})$, with the same codomain. The ($Skel$)-limit of the diagram is the subcategory $L$ of $dom(\phi) \times_{U-\mathfrak{Cat}} dom(\psi)$ such that $Arr(L) = \{ f \in Arr(dom(\phi)\times_{U-\mathfrak{Cat}}dom(\psi) ; \exists u,v \in Arr(codom(\phi)), u,v \text{ are isomorphisms and } u \cdot \pi_{\phi}(f) = \pi_{\psi}(f) \cdot v \}$. Any category with such functors into the two domain categories that the composition of functors on one side is isomorphic to the composition of functors on the other side factors through $L$ via the compositions of the projections with the embedding into the product. By the monic lemma the factorization is unique. However the conclusion of the inclusion lemma might not apply to it, i.e. the two compositions $L \rightarrow dom(\phi) \rightarrow codom(\phi)$ and $L \rightarrow dom(\psi) \rightarrow codom(\psi) = codom(\phi)$ might not be isomorphic, since I might imagine having two different pairs of arrows $(f_{1},g_{1})$, and $(f_{2},g_{2})$, such that the isomorphisms $u_{1},v_{1} \in Arr(codom(\phi))$ which form the commuting square $u_{1} \cdot f_{1} = g_{1} \cdot v_{1}$ differ from the isomorphisms $u_{2}, v_{2} \in Arr(codom(\phi))$ which form the commuting square $u_{2}\cdot f_{2} = g_{2} \cdot v_{2}$. 
\end{ex}

\sss{Lemma, for Reduction to the Standard Limit}
If 
$$
(l, (j \mapsto \lambda ( ( (b,f,\emptyset) \mapsto f(j) )_{(b,f,\emptyset) \in Ob(P)} ) )_{j \in Ob(J)}, \emptyset ) \in Ob(P)
$$

and the limit arrows are unique then this is the usual limit.

\begin{proof}
Trivial.
\end{proof}

\sss{Functoriality}\label{LimFun} An arrow of functors $F \cdot e \rightarrow G \cdot e$ which lifts to an arrow of functors $sk \cdot F \rightarrow sk \cdot G$ (i.e. an arrow in the fibred product of the two functors $Hom_{U-\mathfrak{Cat}^{2}(1)}(e,id_{B})$ and $Hom_{U-\mathfrak{Cat}^{2}(1)}(id_{J},sk)$ ) induces a map from the $(sk,e)$-limit of $F$ to that of $G$, using the colimit map. I.e. $\alpha : F \rightarrow G$ implies that $\alpha (dom(\phi)) \cdot F(\phi) = G(\phi) \cdot \alpha(codom(\phi))$, so that for any arrow $\beta : \Delta_{(J,C)(0)}(c) \rightarrow F\cdot e$ associated to $(a,\beta,\emptyset) \in Ob(P)$ (notation as in the first definition), $Hom_{U-\mathfrak{Cat}^{2}(1)}(e,id_{A})(\alpha) \cdot \beta$ also commutes after applying $sk$ (i.e. comes from an arrow in $Hom_{U-\mathfrak{Cat}^{2}(0)}(I,B)$). Therefore each such $a$ has an arrow into the $sk$-limit of $G$ from the colimit diagram of the definition, which induces a map from the colimit diagram which determines the $sk$-limit of $F$.

\ref{LimFun}.1. Given a diagram $F : I' \longrightarrow Hom_{U-\mathfrak{Cat}^{2}(0)}(I,A)$, and a choice of an $(sk,e)$-limit $(l(i),\lambda(i))$ for any object $i \in Ob(I')$, the construction of (\ref{LimFun}) determines a function $Arr(I') \longrightarrow Arr(A)$

\ref{LimFun}.2. If for any $i \in Ob(I')$, the $(sk,e)$-limit $(l(i),\lambda(i))$ is included in $P(i)$ ($P(i)$ being as in the definition of the $(sk,e)$-limit for $F(i)$) then (\ref{LimFun}.1) determines a functor $I' \longrightarrow A$.

\begin{rem}
Roughly speaking, one takes the colimit of the domains of all limit diagrams on the trivial category which, when forwards composed with $sk$, are the backwards composition by $e$ of an actual limit diagram of $sk \circ F$. Definition (2.3) following this remark is dual to Definition (2.1).
\end{rem}

\sss{$\bold{Definition}$ of the $(sk)$-Colimit}\label{SkColim} Consider functors $J \xrightarrow{e} I \xrightarrow{F} A \xrightarrow{sk} B$.

\ref{SkColim}.1. Let $P$ be the full sub-category of the category 
$$
\downarrow_{(Hom_{U-\mathfrak{Cat}^{2}}(J,A))} (ob_{Hom_{U-\mathfrak{Cat}^{2}}(J,A)}(F\cdot e),\Delta_{(J,A)}) \subseteq Hom_{U-\mathfrak{Cat}^{2}(0)}(J,A)_{ \backslash F \circ e}
$$ 

of arrows, whose objects are given by natural transformations from functor $F \cdot e$
to functor $\Delta_{(J,A)}(a)$, i.e. triples $(\emptyset,\alpha,a)$ for varying $a \in Ob(A)$, such that there exists a natural transformation $\tilde{\alpha}$ from functor $sk\cdot F$ to functor $\Delta_{(I,B)}(a)$ such that the natural transformation from $sk \cdot F$ to $\Delta_{(J,B)}(sk(a))$ is equal to the natural transformation given by sending $j \in Ob(J)$ to $\tilde{\alpha}(e(j))$, i.e. by the set
$$
\{ \alpha : \Delta_{(J,A)}(p) \xrightarrow{\alpha} F\cdot e;
$$
$$
\exists \tilde{\alpha} : sk\cdot \Delta_{(I,A)}(p) = \Delta_{(I,B)}(sk(p)) \longrightarrow sk \cdot F, Hom_{U-\mathfrak{Cat}^{2}(1)}^{(1)}(id_{B},e)(\tilde{\alpha}) = \alpha \}
,$$ 


so as to be given by the category of arrows from the diagonal functor to the object functor of $F \cdot e$ in the category of functors from $J$ to $A$.

\ref{SkColim}.2. Suppose that $\varepsilon : P \longrightarrow \downarrow_{(Hom_{U-\mathfrak{Cat}^{2}}(J,A))}(\Delta_{(J,A)},ob_{(Hom_{U-\mathfrak{Cat}^{2}}(J,A))}(F\cdot e))$ is the inclusion.

\ref{SkColim}.3.
For any $sk : A \rightarrow B, F : I \rightarrow A \in Arr(U-\mathfrak{Cat})$, $codom(F) = dom(sk)$ implies that any pair $(l,\lambda ) \in Ob(A) \times Arr(Hom_{(U-\mathfrak{Cat}^{2})(0)} (A,U-\mathfrak{Set}) )$, $(l,\lambda)$ is a $(sk,e)-colimit(F)$ iff $(l,\lambda)$ is a limit $(cob\downarrow_{(Hom_{(U-\mathfrak{Cat}^{2})(0)} (J,A) )}$

$( ob_{(Hom_{(U-\mathfrak{Cat}^{2})(0)} (J,A) )} (F \cdot e) , \Delta_{(J,A)} ) \cdot \varepsilon_{c} )$.

\begin{lem}
(Inclusion via exactness) Dual to the above.
\end{lem}

\begin{lem}
(Uniqueness via epic) Dual to the above.
\end{lem}

\begin{ex}
Consider $\phi$,$\psi \in Arr(U-\mathfrak{Cat})$, with the same domain. The ($Skel$)-colimit of the diagram is the category $L$ such that its set of objects is the disjoint union of the objects of the codomain categories and the arrows are the formal compositions of the disjoint union of arrows in $Arr(codom(\phi))$, $Arr(codom(\psi))$, and arrows $e_{a} : \phi_{(0)}(a) \rightarrow \psi_{(0)}(a)$, $e_{a}^{-1} : \psi_{(0)}(a) \rightarrow \phi_{(0)}(a)$ formally added for each $a \in Ob(dom(\phi)) = Ob(dom(\psi))$, with the relation generated by requiring that $\forall f \in Arr(dom(\phi)), \phi_{(1)}(f) \cdot e_{dom(f)} = e_{codom(f)} \cdot \psi_{(1)}(f)$. If $l_{\phi} : codom(\phi) \rightarrow L$ and $l_{\psi} : codom(\psi) \rightarrow L$ are given by the $U-\mathfrak{Set}$ coproduct maps then for any $l'_{\phi}, l'_{\psi} \in Arr(U-\mathfrak{Cat})$ such that $l'_{\phi}\cdot \phi \cong \l'_{\psi} \cdot \psi$, there is an arrow $q : L \rightarrow codom(l'_{\phi}) = codom(l'_{\psi})$ such that $l'_{\phi} = q\cdot l_{\phi}$ and $l'_{\psi} = q \cdot l_{\psi}$. If an isomorphism $\alpha : l'_{\phi} \cdot \phi \rightarrow l'_{\psi} \cdot \psi$ is specified (or vice versa), then there is a unique $q : L \rightarrow codom(l'_{\phi})$ such that  $Hom_{(U-\mathfrak{Cat}^{2})(1)}((id_{dom(\phi)},q))_{(1)}( (a \mapsto e_{a})_{a \in Ob(dom(\phi))} ) = \alpha$ (and vice versa).
\end{ex}

\begin{lem}
(Reduction) Dual to the above.
\end{lem}

\begin{lem}
(Functoriality) An arrow of functors $F \rightarrow G$ induces a map from the ($sk$)-colimit of $F$ to that of $G$, using the limit map.
\end{lem}

\begin{rem}
Products and coproducts are not affected by $sk$.
\end{rem}

\sus{Definitions regarding Enrichments}
We will define weak enrichment of sets and categories. 
Sets will be enriched over tensor categories $(A,\otimes)$
and categories over triples $(A,\otimes,F)$
where tensor category $(A,\otimes)$ comes with a tensor 
functor $F: (A,\ten) \longrightarrow (\mathfrak{Set},\times)$.

A weak  enrichment of a set $s$ over $(A,\ten)$
adds to $s$ a category-like structure, a version of $\Hom$ which has values in $A$
(rather than in sets) but without any associativity or unital requirements. We later introduce, for each functor $sk : A \longrightarrow B$, a category of weakly enriched sets, ``associative up to $sk$," in that the associativity diagrams are commutative after the functor $sk$ is applied to them. A weak enrichment of a category $C$ over $(A,\ten)$ with respect to a tensor functor $(F,\rho) : (A,\otimes) \rightarrow (\mathfrak{Set},\times_{\mathfrak{Set}})$ is
a weak enrichment of the set $Ob(C)$ over $(A,\otimes)$ which is compatible with the
$Hom_C$, this compatibility being formulated in terms of the tensor functor $(F,\rho)$.

\sss{$\bold{Definition}$ of a Weakly Enriched Set}
A 
weak enrichment of a set $s\in Ob(U-\mathfrak{Set})$ over a tensor category $(A, \otimes) \in Ob(U-\mathfrak{TCat})$ (a pair consisting of a category $A \in Ob(U-\mathfrak{Cat}$ and a functor $\otimes \in A \times_{U-\mathfrak{Cat}} A \longrightarrow A$) 
is a 
pair
consisting of a map $h : s^{2} \rightarrow Ob(A)$
and a ``composition map'' 
$ \circ : s^{3} \rightarrow Arr(A) )$ such that for any $a,b,c \in s$, 
$$
\circ (a,b,c) :h(a,b) \otimes h(b,c) \longrightarrow h(a,c).
$$

\let\tensor\ten
\newcommand{\USet}{{ U-\mathfrak{Set} }}
\newcommand{\Cat}{{ \mathfrak{Cat} }}
\newcommand{\UCat}{{ U-\mathfrak{Cat} }}
\newcommand{\UTCat}{{ U-\mathfrak{TCat} }}

\sss{$\bold{Definition}$ of the Category of Weak Enrichments}\label{CatWE}
 For any $(A, \otimes ) \in Ob(U-\mathfrak{TCat})$, 
the category of $(A,\otimes)$-enriched sets $WE( A , \otimes ) \in Ob(U-\mathfrak{Cat})$ has as objects weak enrichements of sets  $S=(s,h_S,\circ_S)$,
and for two weak enrichments $S$ and $T$ an arrow 
$f: S = (s, h_{S}, \circ_{S}) \rightarrow T = (t, h_{T}, \circ_{T} )$ is a pair of functions 
$f=( f_{1} : s \rightarrow t , f_{2} : s^{2} \rightarrow Arr(A) )$
such that the following hold.

\ref{CatWE}.1. $\forall a,b \in s,\ \  f_{2}(a,b) : h_{S}(a,b) \longrightarrow h_{T}(f_{1}(a),f_{1}(b))$, and 

\ref{CatWE}.2. $\forall a,b,c \in s$,  
$$
\circ_{T} ( f_{1}(a), f_{1}(b), f_{1}(c) ) \cdot ( f_{2}(a,b) \otimes f_{2}(b,c) ) = f_{2}(a,c) \cdot \circ_{S} ( a,b,c ),
$$ 

i.e. the compositions commute with the arrows defining a ``functor from $S$ to $T$". 

\begin{lem}
\label{TCatFun}

The above construction, of $WE(A,\otimes)$, extends to a functor $WE : U-\mathfrak{TCat} \longrightarrow U'-\mathfrak{Cat}$, from the category of tensor categories to the category of categories.

For any functor of tensor categories $(F,\rho) : (A,\otimes_{A}) \rightarrow (B,\otimes_{B})$ define a functor $WE(F,\rho) : WE(A,\otimes_{A}) \rightarrow WE(B,\otimes_{B})$ from the category of weak enrichments over $(A,\otimes_{A})$ to that of $(B,\otimes_{B})$ as follows.

\ref{TCatFun}.1. 
It sends an object $S=(s,h,\ci)$
of 
$WE(A,\otimes_{A})$ 
to the triple 
$F(S)=(s,h',\circ')$ 
where for 
$a,b,c\in s$,
$$
h'(a,b)= F(h(a,b)) \text{ and } \circ'(a,b,c) = F(\circ(a,b,c)) \cdot \rho(h(b,c),h(a,b)).
$$
 

\ref{TCatFun}.2.
It sends an arrow 
$\phi: S = (s,h_{s},\circ_{s}) \rightarrow (t,h_{t},\circ_{t}) = T$ in $WE(A,\otimes_{A})$ 
(here $s^2\ni (a,b)\mm \phi(a,b)\in Arr(A)$)
to the arrow $F(\phi):F(S)\to F(T)$ that sends 
$
(a,b)\in s^2$ to  $F(\phi(a,b))) \in Arr(B)$.



\end{lem}

\sss{$\bold{Definition}$ of an Weakly Enriched Category}\label{WECat}
A weak enrichment of a category $C$ with respect to a tensor functor $(A, \otimes) \xrightarrow{(F,\rho)} (U-\mathfrak{Set},\times_{U-\mathfrak{Set}})$ is a quadruple $(C,h,\circ,\phi)$ such that $C \in Ob(U-\mathfrak{Cat})$ is a category, $h$ and $\circ$ define a weak enrichment of the set $Ob(C)$, and $\phi : Ob(C)^{2} \rightarrow Arr(U-\mathfrak{Set})$ is a function, such that

\ref{WECat}.1. For any $a,b \in Ob(C)$, $\phi(a,b) : F(h(a,b)) \rightarrow Hom_{C}(a,b)$ is an isomorphism;

\ref{WECat}.2. For any $a,b,c \in Ob(C)$, the composition 
$$
\circ_{C}(a,b,c) : Hom_{C}(b,c) \times_{U-\mathfrak{Set}} Hom_{C}(a,b) \rightarrow Hom_{C}(a,c)
$$
 
of hom sets in $C$ is given by the weak enrichment, i.e.
$$
\circ_{C}(a,b,c) = \phi(a,c)^{-1} \cdot F(\circ(a,b,c)) \cdot \rho(h(b,c),h(a,b)) \cdot (\phi(b,c) \times_{U-\mathfrak{Set}} \phi(a,b))
$$

\sss{$\bold{Definition}$ of the Category of Weakly Enriched Categories}\label{CatWECat}

The category \newline
$WE_{\mathfrak{Cat}}(F,\rho)$ of categories weakly enriched over a tensor category 
$(A,\otimes)$ with respect to a tensor functor
$(F,\rho): (A,\otimes) \to (\USet , \times_{U-\mathfrak{Set}})$, has objects which are
categories $(C,h,\circ,\phi)$ weakly enriched over $(A, \otimes)$.
An arrow 
$f : (C,h_{C},\circ_{C},\phi) \rightarrow (D,h_{D},\circ_{D},\psi)$
consists  
of a functor
$(f_{0},f_{1}) : C \longrightarrow D$ 
and a function $f_{2} : Ob(C)^{2} \longrightarrow Arr(A)$, 
such that 

\ref{CatWECat}.1. $(f_{0},f_{2}) : (Ob(C),h_{C},\circ_{C}) \rightarrow (Ob(D),h_{D},\circ_{D})$ is an arrow of weak enrichments of sets; 

\ref{CatWECat}.2. For any $a,b \in Ob(C)$,
$$
F_{1}(f_{2}(a,b)) = \psi(f_{0}(a),f_{0}(b)) \cdot F(f_{2}(a,b)) \cdot \phi(a,b)^{-1}
$$ 

i.e. the functor agrees with that implied by the enrichment.

\sss{}\label{TCatFunC}
One can construct a functor from the category of tensor categories over the tensor category of sets $U-\mathfrak{TCat}_{/(U-\mathfrak{Set},\times_{U-\mathfrak{Set}})}$ to the category of categories, i.e. 
$$
WE_{\mathfrak{Cat}}( ) := (WE_{\mathfrak{Cat}0}( ), WE_{\mathfrak{Cat}1}( )) :U-\mathfrak{TCat}_{/(U-\mathfrak{Set},\times_{U-\mathfrak{Set}})} \longrightarrow U'-\mathfrak{Cat} 
$$

in analogue to the construction of Lemma \ref{TCatFun}, as follows. For any arrow $(\Phi,\rho) : (F,\rho_{F}) \longrightarrow (G,\rho_{G})$ of tensor categories $(F,\rho_{F}) : (A,\otimes_{A}) \longrightarrow (U-\mathfrak{Set},\times_{U-\mathfrak{Set}})$ and $(G,\rho_{G}) : (B,\otimes_{B}) \longrightarrow (U-\mathfrak{Set},\times_{U-\mathfrak{Set}})$ over $(\mathfrak{Set},\times_{U-\mathfrak{Set}} )$ define a functor \newline $WE_{\mathfrak{Cat}0}(F,\rho_{F}) \longrightarrow WE_{\mathfrak{Cat}0}(G,\rho_{G})$.

\ref{TCatFunC}.1. It is defined on an object $(C,h,\circ,\phi) \in Ob(WE_{\mathfrak{Cat}0}(F,\rho))$ by
$$
(C, h, \circ ,\phi ) \longmapsto (C, \Phi_{(0)} \circ h, ((a,b,c) \mapsto
\Phi_{(1)}(\circ (a,b,c)) \circ \rho (h(b,c),h(a,b)) )_{a,b,c \in Ob(\mathcal{C})},\phi ).
$$

\ref{TCatFunC}.2. It is deifined on arrows $(F,F_{2}) : (C,h_{C},\circ_{C},\phi) \rightarrow (D,h_{D},\circ_{D},\psi)$ by
$$
(F,F_{2}) \mapsto WE_{\mathfrak{Cat}1}(\Phi,\rho)(F,F_{2}) := (F,\Phi_{(1)} \circ F_{2}).
$$

\sss{$\bold{Definition}$ of Two Forgetful Functors}\label{DefForWE}

Define the following two functors.

\ref{DefForWE}.1. For any tensor functor $(F,\rho) : (A,\otimes) \longrightarrow (U-\mathfrak{Set},\times_{U-\mathfrak{Set}})$, the forgetful functor $For^{WE(F,\rho)}_{WE( dom(F, \rho) )} : WE_{\mathfrak{Cat}}(A, \otimes, F) \longrightarrow WE_{\mathfrak{Cat}}(A, \otimes )$ from the category of weakly enriched categories with respect to $(F,\rho)$ to weakly enriched sets with respect to $(A,\otimes)$ is the functor given by passing from a category $C$ to its set of objects $Ob(C)$. More precisely, it is defined on an object $(C,h,\circ,\phi) \in Ob(WE_{\mathfrak{Cat}}(F,\rho) )$ by 
$$
(C,h,\circ,\phi) \mapsto (Ob(C),h,\circ )
$$

and on an arrow $(f,f_{2}) \in Arr(WE_{\mathfrak{Cat}}(F,\rho))$ by
$$
(f,f_{2}) \mapsto (f_{(0)},f_{2})
$$

\ref{DefForWE}.2. The forgetful functor from the category of weakly enriched categories to the category of categories $For^{WE( F,\rho)}_{\mathfrak{Cat}} : WE_{\mathfrak{Cat}}(F,\rho) \longrightarrow U-\mathfrak{Cat}$ is the functor which forgets the enrichment structure, returning the underlying category. I.e. it sends a weakly enriched category $(C,h,\circ,\phi)$ to $C$.

\sss{$\bold{Definition}$ of the Category $WE_{(sk)}(A,\otimes)$} \label{DefWESk}
For any $sk: A \longrightarrow B \in Arr(\mathfrak{Cat})$, define the category $WE_{(sk)}(A,\otimes) \in Ob(\mathfrak{Cat})$ of $((A,\otimes),sk)$-enriched sets. 

\ref{DefWESk}.1. Its objects are sets enriched over $A$. 

\ref{DefWESk}.2. The hom sets
$$
Hom_{WE_{(sk)}(A,\otimes)}((S,h_{S},\circ),(T,h_{T},\circ_{T})) =
$$
are the pairs of maps of sets $(F_{0},F_{1}) \in Arr(\mathfrak{Set})^{2}$ such that $F_{0} : S \rightarrow T$ and $F_{1} : S^{2} \rightarrow Arr(A)$ and 

\ref{DefWESk}.2.1.
For any $a,b \in S, F_{1}(a,b) \in Hom_{A}(h_{S}(a,b),h_{T}(F_{0}(a),F_{0}(b))$.

\ref{DefWESk}.2.2. $F_{1}$ respects composition after applying $sk$, i.e. for any $a,b,c \in S$,
$$
sk(F_{1}(a,c) \cdot \circ_{S}(a,b,c)) = sk(\circ_{T}(F_{0}(a),F_{0}(b),F_{0}(c)) \cdot (F_{1}(a,b) \otimes F_{1}(b,c)))
$$

holds.

\ref{DefWESk}.3. For any $(sk)$-associator $\alpha$, we define the subcategory $WE_{Assoc(sk,\alpha)}(A,\otimes) \subseteq WE_{(sk)}(A,\otimes)$, informally $WE_{Assoc(sk)}(A,\otimes)$, to be the full subcategory whose objects are $(sk)$-associative enriched sets.

\begin{rem}
Roughly speaking, $F_{0}$ is the map between objects of enriched sets, and $F_{1} : h_{S} \rightarrow h_{T}\circ F$ is the ``natural transformation of hom functors," (there are no non-trivial arrows in $S$). This means that applying the ``functor," $(F_{0},F_{1})$, then composing in $T$, versus composing in $S$ and  then applying the functor, gives two arrows in $A$, such that $sk$ of one arrow is equal to $sk$ of the other.
\end{rem}

\begin{lem}
$WE_{(sk)}(A,\otimes)$ is a category.
\end{lem}
\begin{proof}
The issue is composition. Given composible arrows
$$
((S,h_{S},\circ_{S}) \xrightarrow{(F_{0},F_{1})} (T,h_{T},\circ_{T})), ((T,h_{T},\circ_{T}) \xrightarrow{(G_{0},G_{1})} (U,h_{U},\circ_{U})) \in Arr(WE_{(sk)}(A,\otimes))
$$

Starting from the result of application of the functor $sk$ to the arrow which uses the composition $\circ_{U}$,
$$
sk(
$$
$$
\circ_{U}(G_{0}\cdot F_{0}(a),G_{0}\cdot F_{0}(b), G_{0}
\ \cdot\ 
F_{0}(c))
$$
$$
\cdot ((G_{1}(F_{0}(a),F_{0}(b))\cdot F_{1}(a,b)) \otimes (G_{1}(F_{0}(b),F_{0}(c))\cdot F_{1}(b,c)))
$$
$$
) =
$$
by functoriality of $\otimes$ 
$$
sk(
$$
$$
\circ_{U}(G_{0}\cdot F_{0}(a),G_{0}\cdot F_{0}(b)
, G_{0}\cdot F_{0}(c)) \cdot
$$
$$
(G_{1}(F_{0}(a),F_{0}(b)) 
\otimes G_{1}(F_{0}(b),F_{0}(c))) \cdot
$$
$$
(F_{1}(a,b) \otimes F_{1}(b,c))
$$
$$
) =
$$
by functoriality of $sk$
$$
sk(\circ_{U}(G_{0}\cdot F_{0}(a),G_{0}\cdot F_{0}(b), G_{0}\cdot F_{0}(c))) \cdot 
$$
$$
sk((G_{1}(F_{0}(a),F_{0}(b)) \otimes G_{1}(F_{0}(b),F_{0}(c)))) \cdot 
$$
$$
sk((F_{1}(a,b) \otimes F_{1}(b,c))) =
$$
by $(G_{0},G_{1}) \in Arr(WE_{(sk)}(A,\otimes))$,
$$
sk(G_{1}(F_{0}(a),F_{0}(c))) \cdot sk(\circ_{T}(F_{0}(a),F_{0}(b),F_{0}(c))) \cdot sk((F_{1}(a,b) \otimes F_{1}(b,c))) = 
$$
by $(F_{0},F_{1}) \in Arr(WE_{(sk)}(A,\otimes))$,
$$
sk(G_{1}(F_{0}(a),F_{0}(c))) \cdot sk(F_{1}(a,c))\cdot sk(\circ_{S}(a,b,c))
$$
\end{proof}


\begin{lem}
If $(A,\otimes)$ has products, then so does $WE_{(sk)}(A,\otimes)$. The product is functorial.
\end{lem}

\sss{$\bold{Definition}$ of $(sk)$-Associativity}
Consider a tensor category $(A, \otimes)\in Ob(U-\mathfrak{ATCat})$ with a functor $sk:A\to B$ and an $(sk)$-associator $\alpha : Ob(A)^{3} \rightarrow Arr(A)$. An $(A,\ten)$-enriched set $(S, h, \circ) \in Ob(WE (A, \otimes) )$ is said to be {\em ($sk,\alpha$)-associative} if for any $a,b,c,d \in S$, 
$$
sk_{(1)} (\circ(a,b,d) \cdot ( id_{h(a,b)} \otimes \circ(b,c,d) ) \cdot \alpha(h(a,b), h(b,c), h(c,d)) ) =
$$
$$
sk_{(1)} (\circ(a,c,d) \cdot (\circ (a,b,c) \otimes id_{h(c,d)} ) )
,$$ i.e. the standard self-consistency  diagram (pentagram) for the enriched composition $\circ$ is required to commute  after applying the functor $sk$
$$
\begin{CD}
h(c,d) \otimes ( h(b,c) \otimes h(a,b))
@>\alpha(h(c,d),h(b,c),h(a,b)) >>
(h(c,d) \otimes h(b,c)) \otimes h(a,b)
\\
@V{id_{h(c,d)} \otimes \circ(a,b,c)}VV
@V{\circ(b,c,d) \otimes id_{h(a,b)}}VV
\\
h(c,d) \otimes h(a,c)
@.
h(b,d) \otimes h(a,b)
\\
@V{\circ(a,c,d)}VV
@V{\circ(a,b,d)}VV
\\
h(a,d)
@>=>>
h(a,d).
\end{CD}
$$

If the associator $\alpha$ is understood, then we will write ``$(sk)$-associative".

\sss{$\bold{Definition}$ of $WE_{Ass(A,\otimes)(sk,\alpha)}$} Suppose that $(A,\otimes)$ has an associator $\alpha$ (see \ref{ATCat}). Define $WE_{Ass(A,\otimes)(sk,\alpha)}$ to be the full subcategory of $WE_{(sk)}(A,\otimes)$ generated by enriched sets $(S,h_{S},\circ_{S}) \in Ob(WE_{(sk)}(A,\otimes))$ which are $(sk,\alpha)$-associative. If the associator is understood, then we will denote this by ``$WE_{Ass(A,\otimes)(sk)}$".

\sus{Enrichment of $Hom_{WE_{(sk)}(A,\otimes)}(I,C)$}

Consider a tuple of functors $\{ p_{i} : I_{i} \longrightarrow A \}_{i=1}^{n}$. Suppose that for each $i \in \{1,...,n \}$, the colimit $colim \ p_{i} \in Ob(A)$ exists, with universal arrows $e_{i(x_{i})} : p_{i}(x_{i}) \rightarrow colim \ p_{i}$. Suppose that the colimit of the functor $\otimes_{i=1}^{n} p_{i} : \prod_{i=1}^{n} I_{i} \longrightarrow A$ defined by $(x_{i})_{i=1}^{n} \mapsto \otimes_{i=1}^{n}p_{i}(x_{i})$ is also an object in $A$. Consider the arrow $(colim \otimes_{i=1}^{n} p_{i} \rightarrow \otimes_{i=1}^{n} colim \ p_{i}) \in Arr(A)$ induced by $(x_{i})_{i=1}^{n} \mapsto \otimes_{i=1}^{n} e_{i(x_{i})}$; i.e. by tensoring the universal arrows together. The following lemma states that under certain conditions on the $p_{i}$, the above defines a natural transformation with respect to arrows of functors $\phi_{i} : p_{i} \rightarrow q_{i}$.

The $(A,\otimes)$-enrichment of the hom-sets in $WE_{(sk)}(A,\otimes)$ involves such colimits, and the definition of the composition requires that the above arrows should be isomorphisms. This means that the ``forward and backward composition functors" to be introduced in lemma \ref{PushPull} below are determined by the arrows between products $\prod h_{S}(...) \rightarrow \prod h_{T}(...)$.

\sss{Lemma on the Naturality of $\tau$}\label{LemNatTau}
Suppose that $\{ F_{i},G_{i} : I_{i} \longrightarrow A \}_{i=1}^{n}$ are functors, and $\{ (F_{i} \xrightarrow{\phi_{i}} G_{i}) \}_{i=1}^{n}$ are arrows of functors. Suppose that for each $i \in \{ 1,...,n\}$, $P_{F_{i}} \subseteq \downarrow_{(Hom^{(1)}_{U-\mathfrak{Cat}^{2}}(J_{i},A))}(\Delta_{(J_{i},A)},F_{i}\circ \varepsilon_{i} )$ and $P_{G_{i}} \subseteq \downarrow_{(Hom^{(1)}_{U-\mathfrak{Cat}^{2}}(J_{i},A))}(\Delta_{(J_{i},A)},G_{i}\circ \varepsilon_{i} )$ are subcategories, where $\varepsilon_{i} : J_{i} \subseteq I_{i}$ is the subcategory with only identity arrows.

\ref{LemNatTau}.1. Suppose that the functors $p_{F_{i}} : P_{F_{i}} \longrightarrow A$ and $p_{G_{i}} : P_{G_{i}} \longrightarrow A$ are as in the conditions of the limit inclusion lemma (i.e. $colim ( p_{F_{i}} )$ determines an object in $P_{F_{i}}$, with the analogue holding for $G_{i}$) 

\ref{LemNatTau}.2. Define an arrow of sets $\tau_{( \ )} : \prod_{i=1}^{n} Ob(Hom^{(1)}_{U-\mathfrak{Cat}^{2}}(I_{i},A)) \rightarrow Arr(A)$ so that for any $(H_{i})_{i=1}^{n} \in \prod_{i=1}^{n} Ob(Hom^{(1)}_{U-\mathfrak{Cat}^{2}}(I_{i},A))$, $\tau_{((H_{i})_{i=1}^{n}))} : colim (\bigotimes_{i=1}^{n} p_{H_{i}}) \rightarrow \bigotimes_{i=1}^{n} colim(p_{H_{i}})$ is the universal arrow for the colimit induced by the assignment (where $\lambda_{(i)}$ is the natural transformation defining the colimit of $p_{H_{i}}$)
$$
((a_{(i)},f_{(i)})_{i=1}^{n} \mapsto \otimes_{i=1}^{n} \lambda_{(i)}((a_{(i)},f_{(i)})) \ )_{(a_{(i)},f_{(i)})_{i=1}^{n} \in \prod_{i=1}^{n} P_{H_{i}}}
$$ 

\ref{LemNatTau}.3. If $u=u_{(p)} : colim (\bigotimes_{i=1}^{n} p_{F_{i}}) \rightarrow colim (\bigotimes_{i=1}^{n} p_{G_{i}})$ is the universal arrow for the colimit induced by the assignment $$
((a_{( \ )} , f_{( \ )}) \mapsto \lambda'_{G}((\otimes_{i=1}^{n}a_{(i)},\otimes_{i=1}^{n}\phi_{i}\cdot f_{(i)} )) \ )_{(a_{( \ )} , f_{( \ )}) \in Ob(\prod_{i=1}^{n}P_{F_{i}})}
$$

then
$$
\otimes_{i=1}^{n} \lambda_{G_{j}}((a,\phi_{j}\cdot f )) \cdot \tau_{F_{( \ )}} = \tau_{G_{( \ )}} \cdot u
$$

I.e., $\tau : colim (\bigotimes_{i=1}^{n} p_{i}) \rightarrow \bigotimes_{i=1}^{n} colim(p_{i})$ is ``natural at $(\phi_{i})_{i=1}^{n}$."

\begin{proof}
By the monic arrow condition, the arrows involved are situated above the products, $\prod_{a \in Ob(I_{i})} F_{i}(a)$, so that the arrows $\otimes_{i=1}^{n}a_{(i)} \rightarrow \otimes_{i=1}^{n} l_{G_{i}}$ are pure tensors respecting the arrows $\phi_{i}$.
\end{proof}

The following lemma defines a weak enrichment of the set $Hom_{WE_{(sk)}(A,\otimes)}(C,D)$. To any ``$(A,\otimes)$-functors" $\Phi,\Psi : C \rightarrow D$, one attaches a category $P$, and defines the hom object between $\Phi$ and $\Psi$ to be a colimit of a certain functor $P \longrightarrow A$. Roughly speaking, $P$ keeps track of all arrows into to the product $\prod_{x \in Ob(C)}h_{D}(\Phi(x),\Psi(x))$  which respect the composition with any arrows ``coming from some $h_{C}(x,y)$," after one applies $sk$. $P$ is a full sub-category of the category of arrows over $\prod_{x \in Ob(C)}h_{D}(\Phi(x),\Psi(x))$. The objects of $P$ are all arrows $(a \xrightarrow{\pi} \prod_{x \in Ob(C)}h_{D}(\Phi(x),\Psi(x))$, such that for any $x,y \in Ob(C)$, for any $(t_{0} \xrightarrow{t} h_{C}(x,y)) \in Arr(A)$, tensoring $\pi$ with $t$, projecting to the $y$-component $\prod_{x \in Ob(C)}h_{D}(\Phi(x),\Psi(x)) \rightarrow h_{D}(\Phi(y),\Psi(y))$, and composing in $D$ is $(sk)$-equal to tensoring $t$ with $\pi$, projecting to the $x$-component, and composing in $D$. One defines $p : P \longrightarrow A$ to be the functor which remembers the domain of a given arrow. One associates to $\Phi$ and $\Psi$ the object $colim(p) \in Ob(A)$ (assuming that the colimit exists).

One composes, i.e. defines, for all $(A,\otimes)$-functors $\Phi,\Psi,\Xi$, an arrow 
$$
(h(\Phi,\Psi) \otimes h(\Psi,\Xi) \xrightarrow{\circ} h(\Phi,\Xi)) \in Arr(A)
$$ 

by taking the inverse of the arrow $colim \ (P_{\Phi,\Psi} \otimes P_{\Psi,\Xi}) \rightarrow (colim \ P_{\Phi,\Psi}) \otimes (colim \ P_{\Psi,\Xi})$ (that this is an isomorphism is assumed), and recognizing $colim \ (P_{\Phi,\Psi} \otimes P_{\Psi,\Xi})$ as an object in $P_{\Phi,\Xi}$ by using the composition in $D$ and the projection for the products to define arrows $P_{\Phi,\Psi}(x) \otimes P_{\Psi,\Xi}(y) \rightarrow \prod_{x \in Ob(C)}h_{D}(\Phi(x),\Xi(x))$. As an object in $P_{\Phi,\Xi}$, $colim \ (P_{\Phi,\Psi} \otimes P_{\Psi,\Xi})$ has assigned to it an arrow into $colim \ P_{\Phi,\Xi}$, which is defined to be the hom object assigned to $\Phi$ and $\Xi$. One composes this colimit arrow with the inverse of the first arrow to define the composition arrow.

\sss{Lemma on the Enrichment of $Hom_{WE_{(sk)}(A,\otimes)}(C,D)$} \label{HomEnr}

Suppose that $(A,\otimes)$ has a symmetrizer and associator for the tensor. 

\ref{HomEnr}.1. For any $\Phi,\Psi \in Hom_{WE(A,\otimes)(sk)}(C,D)$, define 
$$
P \subseteq \downarrow_{( A )}(Id_{A} , ob_{(A)}(\prod_{x \in Ob(C)} h_{D}(\Phi(x),\Psi(x))) )
$$

to be the full subcategory generated by objects (i.e. arrows $a \xrightarrow{\pi} \prod_{x \in Ob(C)} h_{D}(\Phi(x),\Psi(x))$ in $A$) such that for any $x,y \in Ob(C), (t_{0} \xrightarrow{t} h_{C}(x,y)) \in Arr(A)$, 
$$
sk(\circ_{D} \cdot (id_{h_{D}(\Phi(y),\Psi(y))} \otimes \Psi(x,y)) \cdot ((\pi \cdot \pi_{y}) \otimes t)) =
$$
$$ 
sk(\circ_{D} \cdot (\Phi (x,y) \otimes id_{h_{D}(\Phi(x),\Psi(x))}) \cdot \sigma \cdot ((\pi \cdot \pi_{x}) \otimes t))
$$

Then define $\bar{h}_{WE(sk)(A,\otimes)1}(C,D)(\Phi,\Psi)$ to be the colimit of the domain object functor $p : P \longrightarrow A$ defined by $(a,f) \mapsto a$.

\ref{HomEnr}.2. Suppose that the compostion on $D$ is $(sk)$-associative, and the arrows $u = u_{(p)} : colim \otimes_{i=1}^{n} \ p_{i} \rightarrow \otimes_{i=1}^{n} colim \ p_{i}$ are isomorphisms, defined as in the previous lemma, and $p = \{ (1,p_{\Phi,\Psi}) \} \cup \{ (2,p_{\Psi,X} ) \} : \{1,2\} \rightarrow Arr(\mathfrak{Cat})$. Then define the composition $\circ_{WE(A,\otimes)(sk)}(C,D)(\Phi,\Psi,X) \in Arr(A)$ by taking it to be the composition of the colimit arrow $e : colim (-_{1} \otimes -_{2}) \rightarrow h_{WE(A,\otimes)(sk)}(C,D)(\Phi,X)$ associated to the object $ (colim (-_{1} \otimes -_{2}),\phi ) \in Ob(P_{\Phi,X})$ determined by the arrow
$$
\phi : colim (-_{1} \otimes -_{2}) \rightarrow \prod_{a \in Ob(C)} h_{D}(\Phi(a), X(a))
$$

induced by sending any given $((a,f),(b,g)) \in Ob(P_{\Phi,\Psi}\times_{\mathfrak{Cat}} P_{\Psi,X})$ to the product arrow given to the assignment
$$
a \mapsto \circ_{D}(\Phi(a),\Psi(a),X(a))\cdot (\pi_{(\Phi,\Psi)a}\cdot f \otimes \pi_{(\Psi,X)a}\cdot g)
$$

with $u^{-1}$, I.e.
$$
\circ_{WE(A,\otimes)(sk)}(C,D)(\Phi,\Psi,X) := e \cdot u^{-1}
$$

\ref{HomEnr}.3. Define $\bar{h}_{WE(A,\otimes)(sk)}(C,D) := $
$$
(Hom_{WE_{(sk)}(A,\otimes)}(C,D),
$$
$$
\bar{h}_{WE(A,\otimes)(sk)1}(C,D)( \ , \ ) , \circ_{WE(A,\otimes)(sk)}(C,D)( \ , \ , \ ) ) \in Ob(WE(A,\otimes))
$$

i.e. part i. gives the hom objects and part ii. gives the composition.

\ref{HomEnr}.4.1. The enriched set $\bar{h}_{WE(A,\otimes)(sk)}(C,D)$ is $(sk)$-associative. 

\ref{HomEnr}.4.2. If $(A,\otimes)$ has a unit $I$ such that $\circ_{D}$ is $(Yo^{opp}_{(0)}(I))$-associative, then so does $\bar{h}_{WE(A,\otimes)(sk)}(C,D)$.

\begin{rem}
The enrichment on $Hom_{WE_{(sk)}(A,\otimes)}(C,D)$, i.e. the objects $h(\Phi,\Psi)$ defined in the previous lemma for $(A,\otimes)$-functors $\Phi$ and $\Psi$, were initially constructed as $(sk)$-equalizers. I believe that the present construction can also be realized as an $(sk)$-equalizer, but by use of a diagram containing arrows of the form $[ ( Homfun(A)\circ (( - \otimes J) \times id_{A}), \circ \circ (\pi \otimes id_{J})) ] \in Arr(\Omega)$, and with restrictions on $A$.
\end{rem}

\sss{$\bold{Definition}$ of the Enriched Arrows Functor}
If $(A,\otimes) \in Ob(\mathfrak{TCat})$ has coproducts, then define
$$
\bar{Arr}_{(A,\otimes)} : WE_{(A,\otimes)} \longrightarrow A
$$

by $(S,h,\circ) \mapsto \coprod_{s,t \in S} h(s,t)$ and $(F_{0},F_{1}) \mapsto \coprod_{s,t \in S} F_{1}(s,t)$.

\begin{rem}
The functor $sk$ is not referred to in this definition. $\bar{Arr}_{(A,\otimes)}$ is the ``enriched arrow functor."
\end{rem}

\begin{lem}
$\bar{Arr}_{(A,\otimes)}$ is faithful.
\end{lem}

The following lemma concerns the self-enrichment of the category $WE_{Assoc(sk)}(A,\otimes)$. The enriched hom set defined in the previous lemma is denoted by ``$\bar{h}_{WE(A,\otimes)(sk)}(B,C)$." Part (i) of the following lemma defines the ``forward composition/pushforward functor," \newline $\bar{h}_{WE(A,\otimes)(sk)}(B,C) \rightarrow \bar{h}_{WE(A,\otimes)(sk)}(B,D)$. Part (ii) defines the ``backward composition/ pullback functor," $\bar{h}_{WE(A,\otimes)(sk)}(C,D) \rightarrow \bar{h}_{WE(A,\otimes)(sk)}(B,D)$. Part (iii) states that one can use these to define an arrow $(\bar{h}_{WE(A,\otimes)(sk)}(B,C) \times \bar{h}_{WE(A,\otimes)(sk)}(C,D)$ $\rightarrow$ $\bar{h}_{WE(A,\otimes)(sk)}(B,D))$ $\in Arr(WE_{Assoc(sk)}(A,\otimes))$ which gives the enriched composition in $WE_{Assoc(sk)}(A,\otimes)$.

\sss{Lemma on Composition Functors} \label{PushPull}
Given $(sk) \in Arr(Cat)$, for any $(F : C \rightarrow D)$, $(G : B \rightarrow C ) \in Arr(WE_{Assoc(sk)}(A,\otimes))$,

\ref{PushPull}.1.
$$
(F_{*} : \bar{h}_{WE(A,\otimes)(sk)}(B,C) \rightarrow \bar{h}_{WE(A,\otimes)(sk)}(B,D)) \in Arr(WE_{Assoc(sk)}(A,\otimes))
$$

is induced by $\prod_{a \in Ob(B)} h_{C}(\Psi_{1}(a),\Psi_{2}(a)) \rightarrow \prod_{a\in Ob(B)} h_{D}(F\cdot \Psi_{1}(a),F\cdot\Psi_{2}(a))$, which induces a functor $P_{\bar{h}_{WE(A,\otimes)(sk)}(B,C)(\Psi_{1},\Psi_{2})} \longrightarrow P_{\bar{h}_{WE(A,\otimes)(sk)}(B,D)(F\cdot\Psi_{1},F\cdot\Psi)}$, so that an arrow is induced from the colimit of the first diagram ($p_{\bar{h}_{WE(A,\otimes)(sk)}(B,C)(\Psi,\Psi_{2})}$ to the colimit of the second $p_{\bar{h}_{WE(A,\otimes)(sk)}(B,D)(\F\cdot\Psi_{1},F\cdot\Psi)}$.

\ref{PushPull}.2.
$$
(G^{*} : \bar{h}_{WE(A,\otimes)(sk)}(C,D) \rightarrow \bar{h}_{WE(A,\otimes)(sk)}(B,D)) \in Arr(WE_{Set(sk)}(A,\otimes))
$$

is induced by $\prod_{a \in Ob(C)} h_{D}(\Phi_{1}(a),\Phi_{2}(a)) \rightarrow \prod_{a \in Ob(B)} h_{D}(\Phi_{1}\cdot G(a), \Psi_{2} \cdot G(a))$, which is the product map induced by the assignment $(a \mapsto \pi_{G(a)})$.

These are analogues to the usual forward and backward functors associated to composition on either end of a functor category $Hom(B,C)$.

\ref{PushPull}.3. From an arrow of functors $\alpha : \times_{A} \rightarrow \otimes$, the previous two constructions, and the product structure, construct an arrow in $WE_{Set(A,\otimes)}$

$$
\bar{h}_{WE(A,\otimes)(sk)}(B,C) \times_{WE_{Assoc(sk)}(A,\otimes)} \bar{h}_{WE(A,\otimes)(sk)}(C,D) \longrightarrow \bar{h}_{WE(A,\otimes)(sk)}(B,D)
$$

(Not unique. The choice corresponds with the choice, of the path $F_{1}\cdot G_{1} \rightarrow F_{1} \cdot G_{2} \rightarrow F_{2} \times G_{2}$, versus the path $F_{1}\cdot G_{1} \rightarrow F_{2} \cdot G_{1} \rightarrow F_{2} \times G_{2}$).

\ref{PushPull}.4. Defining $\bar{\circ} : Ob(WE_{Ass(A,\otimes)(sk)})^{3} \mapsto Arr(WE_{Ass(A,\otimes)(sk)})$ by sending $(B,C,D) \in Ob(WE_{Assoc(sk)}(A,\otimes)$ to the arrow in (iii), where $WE_{Assoc(sk)}(A,\otimes) \subseteq WE_{Assoc(sk)}(A,\otimes)$ is defined to the 
$$
(Ob(WE_{Assoc(sk)}(A,\otimes)),\bar{h}_{WE_{Assoc(sk)}(A,\otimes)},\bar{\circ})
$$

is an $(WE_{Ass(sk)}(A,\otimes),\times_{WE_{Assoc(sk)}}(A,\otimes)$-enriched set, whose composition is $(Ob)$-\newline associative and $(sk\cdot\bar{Arr}_{(A,\otimes)})$-associative.

\begin{proof}
Parts i. and ii. consist only in checking for $(sk)$-commutativity so that the constructions can be made. Part iii., states that for any $C,D,E \in Ob(WE(A,\otimes))$, for any $\Phi_{1},\Phi_{2},\Phi_{3} \in Ob(\bar{h}_{WE(A,\otimes)(sk)}(C,D))$, for any $\Psi_{1},\Psi_{2},\Psi_{3} \in Ob(\bar{h}_{WE(A,\otimes)(sk)}(D,E)$,
$$
\circ_{\bar{h}_{(C,E)}}(\Psi_{1}\cdot \Phi_{1} , \Psi_{2} \cdot \Phi_{2}, \Psi_{3} \cdot \Phi_{3}) \cdot
$$
$$
( \circ_{\bar{h}_{(C,E)}}(\Psi_{1}\cdot \Phi_{1} , \Psi_{1} \cdot \Phi_{2}, \Psi_{2}\cdot \Phi_{2}) \otimes \circ_{\bar{h}_{(C,E)}}(\Psi_{2}\cdot\Phi_{2} ,\Psi_{2}\cdot \Phi_{3}, \Psi_{3}\cdot\Phi_{3})) \cdot
$$
$$
((\Phi_{2}^{*}\otimes\Psi_{1*}) \otimes (\Phi_{3}^{*} \otimes \Psi_{2*})) \cdot \sigma_{1*} =_{(sk)}
$$
$$
\circ_{\bar{h}_{(C,E)}}(\Psi_{1}\cdot\Phi_{1} , \Psi_{1}\cdot\Phi_{3} , \Psi_{3}\cdot\Phi_{3}) \cdot(\Phi_{3}^{*} \otimes \Psi_{1*}) \cdot (\circ_{\bar{h}_{(D,E)}}(\Psi_{1},\Psi_{2},\Psi_{3})\otimes \circ_{\bar{h}_{(C,D)}}(\Phi_{1},\Phi_{2},\Phi_{3}) ))
$$

given that $colim(p) \in P$ with a monic arrow into the relevant product, and that $\forall f,g : x \rightarrow \prod_{i \in I}y_{i}$, $\forall i \in I, sk(\pi_{i}\cdot f) = sk(\pi\cdot g) \Longrightarrow sk(f) = sk(g)$.

All arrows between the objects $\bar{h}_{WE(A,\otimes)(sk)}(C,D) \rightarrow \bar{h}_{WE(A,\otimes)(sk)}(C',D')$ commute with monic arrows $\bar{h}_{WE(A,\otimes)(sk)}(C,D) \rightarrow \prod_{c \in Ob(C)} h_{D}(F(c),G(c))$.

After taking the inverse of the isomorphism $\otimes colim \ p_{i} \leftarrow colim \otimes \ p_{i}$ (that this is an isomorphism is assumed), these maps are determined by the arrows $\Psi_{i*}$ and $\Phi_{i}^{*}$. On the components of the product $\Phi_{i}^{*}$ come from identity arrows and $\Psi_{i*}$ from $\Psi(a,b)$.
$$
\text{Diagram with two arrows,}
$$
$$
\prod_{a \in Ob(D)} h_{E}(\Psi_{1}(a),\Psi_{2}(a)) \otimes \prod_{a \in Ob(D)} h_{E}(\Psi_{2}(a),\Psi_{3}(a)) \otimes
$$
$$
\prod_{a \in Ob(C)} h_{D}(\Phi_{1}(a),\Psi_{2}(a)) \otimes \prod_{a \in Ob(C)} h_{D}(\Phi_{2}(a),\Phi_{3}(a))
$$
$$
\rightarrow h_{E}(\Psi_{1}\circ\Phi_{1}(a),\Psi_{3}\circ\Phi_{3}(a))
$$

 (one side is $\Phi^{*}_{3} \otimes \Phi^{*}_{3} \otimes \Psi_{1*} \otimes \Psi_{1*}$ and the other is $\Phi_{2}^{*} \otimes \Phi_{3}^{*} \otimes \Psi_{1*} \otimes \Psi_{2*}$). The $\Phi^{*}_{3} \otimes \Phi^{*}_{3} \otimes \Psi_{1*} \otimes \Psi_{1*}$ side is
$$
\Pi \rightarrow
$$
$$
h_{E}(\Psi_{1}\cdot\Phi_{3} (a) , \Psi_{2} \cdot \Phi_{3}(a)) \otimes h_{E}(\Psi_{2},\Phi_{3}(a),\Psi_{3}\cdot\Phi_{3}(a)) \otimes h_{D}(\Phi_{1}(a),\Phi_{2}(a))\otimes h_{D}(\Phi_{2}(a),\Phi_{3}(a))
$$
$$
\xrightarrow{id \otimes id \otimes \Psi_{1}(\Phi_{1}(a),\Phi_{2}(a)) \otimes \Psi_{1}(\Phi_{2}(a),\Phi_{3}(a))}
$$
$$
h_{E}(\Psi_{1}\cdot\Phi_{3}(a),\Psi_{2}\cdot\Phi_{3}(a))\otimes h_{E}(\Psi_{2}\cdot \Phi_{3}(a),\Psi_{3}\cdot\Phi_{3}(a)) \otimes
$$
$$
h_{E}(\Psi_{1}\cdot\Phi_{1}(a),\Psi_{1}\cdot\Phi_{2}(a))\otimes h_{E}(\Psi_{1}\cdot\Phi_{2}(a) , \Psi_{1}\cdot\Phi_{3}(a))
$$
$$
\xrightarrow{\circ_{E}} h_{E}(\Psi_{1}\cdot\Phi_{1}(a),\Psi_{3}\cdot\Phi_{3}(a))
$$

The $\Phi_{2}^{*} \otimes \Phi_{3}^{*} \otimes \Psi_{1*} \otimes \Psi_{2*}$ side is
$$
\Pi \rightarrow
$$
$$
h_{E}(\Psi_{1}\cdot\Phi_{2}(a),\Psi_{2}\cdot\Phi_{2}(a)) \otimes h_{E}(\Psi_{2},\Psi_{3}(a),\Psi_{3}\cdot\Phi_{3}(a)) \otimes h_{D}(\Phi_{1}(a),\Phi_{2}(a)) \otimes h_{D}(\Phi_{2}(a),\Phi_{3}(a))
$$
$$
\xrightarrow{id \otimes id \otimes \Psi_{1}(\Phi_{1}(a),\Phi_{2}(a)) \otimes \Psi_{2}(\Phi_{2}(a),\Phi_{3}(a)) \cdot \sigma}
$$
$$
h_{E}(\Psi_{2}\cdot\Phi_{3}(a),\Psi_{3}\cdot\Phi_{3}(a))\otimes h_{E}(\Psi_{1}\cdot\Phi_{2}(a),\Psi_{2}\cdot\Phi_{2}(a))\otimes 
$$
$$
h_{E}(\Psi_{2}\cdot\Phi_{2}(a),\Psi_{2}\cdot\Phi_{3}(a))\otimes h_{E}(\Psi_{1}\cdot\Phi_{1}(a),\Psi_{1}\cdot\Phi_{1}(a))
$$
$$
\xrightarrow{\circ_{E}}  h_{E}(\Psi_{1}\cdot\Phi_{1}(a),\Psi_{3}\cdot\Phi_{3}(a))
$$

By definition of $P_{\bar{h}_{(C,D)}(\Psi_{1},\Psi_{2})}$, in particular, ``commutativity" of the composition with any arrow going through a hom object of $C$, the two arrows are $(sk)$-equal.

\end{proof}
\begin{rem}
On underlying ``objects" this is the usual composition (e.g. 1-composition, of functors).
\end{rem}

\sss{Lemma} For any arrow of functors $\rho : \times_{A} \rightarrow \otimes_{A}$, for any $\bar{S},\bar{T} \in$ \newline $Ob(WE_{Assoc(sk)}(A,\otimes))$, such that the diagrams determining the enrichment on $\bar{h}(\bar{S},\bar{T})$ satisfy the $P$-colimit inclusion condition, we construct two arrows of enriched sets
$$
\eta_{l} = (\eta_{l0},\eta_{l1}) : \bar{S} \times_{WE_{Assoc(sk)}(A,\otimes)} \bar{h}(\bar{S},\bar{T}) \rightarrow \bar{T}
$$
$$
\eta_{r} = (\eta_{r0},\eta_{r1}) : \bar{h}(\bar{S},\bar{T}) \times_{WE_{Assoc(sk)}(A,\otimes)} \bar{S} \rightarrow \bar{T}
$$

by the following. We explicitly describe $\eta_{l}$. For any $(s,\phi) \in S \times Hom_{WE_{Assoc(sk)}(A,\otimes)}(\bar{S},\bar{T})$ $\cong Ob(\bar{S} \times \bar{h}(\bar{S},\bar{T}))$,
$$
\eta_{l0}(s,\phi) := \phi_{0}(s)
$$

and for any $(s,\phi),(t,\psi) \in S \times Hom_{WE_{Assoc(sk)}(A,\otimes)}(\bar{S},\bar{T}) \cong Ob(\bar{S} \times \bar{h}(\bar{S},\bar{T}))$, if $q : h_{\bar{h}(\bar{S},\bar{T})}(\phi,\psi) \rightarrow \prod_{s_{0} \in S} h_{T}(\phi(s_{0}),\psi(s_{0}))$ is the arrow in $A$ given by the colimit universality, determining an object in $P$ of the enrichment,
$$
\eta_{l1}((s,\phi),(t,\psi)) := \circ_{T}(\phi(s),\phi(t),\psi(t)) \cdot (\phi(s,t) \otimes (\pi_{t} \cdot q)) \cdot \rho((h_{S}(s,t),h_{\bar{h}(\bar{S},\bar{T})}(\phi,\psi))).
$$

Define $\eta_{r}$ so that
$$
\eta_{r1}((s,\phi),(t,\psi)) := \circ_{T}(\phi(s),\psi(s),\psi(t)) \cdot ((\pi_{s}\cdot q) \otimes \psi(s,t)) \cdot \rho((h_{\bar{h}(\bar{S},\bar{T})}(\phi,\psi)),h_{S}(s,t)).
$$

\sss{$\bold{Definition}$}\label{DefWEIn} $WE$ with initial object. For any $(A,\otimes)$, for any initial object $\emptyset_{A} \in Ob(A)$, we make the following definitions.

\ref{DefWEIn}.1. For any two $(sk)$-commutative arrows of functors $\lambda_{A} : \emptyset_{A} \otimes Id_{A} \rightarrow Id_{A} \leftarrow Id_{A} \otimes \emptyset_{A} : \rho_{A}$ we define $WE_{(sk,\lambda_{A},\rho_{A})}(A,\otimes) \subseteq WE_{Assoc(sk)}(A,\otimes)$ to be the full subcategory generated by enriched sets $\bar{S} = (S,h_{S},\circ_{S}) \in Ob(WE_{Assoc(sk)}(A,\otimes)$ such that for any $a,b,c \in S$, 
$$
sk(\circ_{S}(a,b,c) \cdot (e_{h_{S}(a,b)} \otimes id_{h_{S}(b,c)})) = sk(\alpha_{l}(h_{S}(b,c))).
$$

\ref{DefWEIn}.2. For any associator $\alpha$ so that $(A,\otimes,\alpha) \in U-\mathfrak{ATCat}$, for any two $(sk)$-arrows of functors $\lambda_{A} : \emptyset_{A} \otimes Id_{A} \rightarrow Id_{A} \leftarrow Id_{A} \otimes \emptyset_{A} : \rho_{A}$ such that
$$
Hom^{(1)}_{U-\mathfrak{Cat}^{2}}(Id_{A_{0}},sk)(\lambda_{A} \cdot Hom^{(1)}_{U-\mathfrak{Cat}^{2}}(Id_{A_{0}},\emptyset_{A}\otimes Id_{A})(\rho_{A}) ) = 
$$
$$
Hom^{(1)}_{U-\mathfrak{Cat}^{2}}(Id_{A_{0}},sk)(\rho_{A} \cdot Hom^{(1)}_{U-\mathfrak{Cat}^{2}}(Id_{A_{0}},Id_{A} \otimes \emptyset_{A})(\lambda_{A}) \cdot Hom^{(1)}_{U-\mathfrak{Cat}^{2}}(\varepsilon_{\emptyset_{A}\times A_{0} \times \emptyset_{A}}^{A_{0}\times A_{0} \times A_{0}},Id_{A})(\alpha) )
$$

and
$$
Hom^{(1)}_{U-\mathfrak{Cat}^{2}}(Id_{A_{0}} \times Id_{A_{0}},sk) (Hom^{(1)}_{U-\mathfrak{Cat}^{2}}(Id_{A_{0}}\times Id_{A_{0}}, \otimes)((id_{\varepsilon_{A_{0}}^{A}},\lambda_{A}))) =
$$
$$
Hom^{(1)}_{U-\mathfrak{Cat}^{2}}(Id_{A_{0}} \times Id_{A_{0}},sk) (Hom^{(1)}_{U-\mathfrak{Cat}^{2}}(Id_{A_{0}}\times Id_{A_{0}}, \otimes)((\rho_{A},id_{\varepsilon_{A_{0}}^{A}})) \cdot
$$
$$
Hom^{(1)}_{U-\mathfrak{Cat}^{2}}(\varepsilon_{A_{0} \times \emptyset_{A} \times A_{0}}^{A_{0} \times A_{0} \times A_{0}},Id_{A})(\alpha))
$$

i.e., which are $sk$-associative, we define $WE_{(sk,\alpha,\lambda_{A},\rho_{A})}(A,\otimes) \subseteq WE_{(sk,\alpha)}(A,\otimes)$ to be the full subcategory generated by enriched sets $\bar{S} = (S,h_{S},\circ_{S}) \in Ob(WE_{Assoc(sk)}(A,\otimes)$ such that for any $a,b,c \in S$, 
$$
sk(\circ_{S}(a,b,c) \cdot (e_{h_{S}(a,b)} \otimes id_{h_{S}(b,c)})) = sk(\lambda_{A}(h_{S}(b,c)))
$$

and
$$
sk(\circ_{S}(a,b,c) \cdot (id_{h_{S}(a,b)} \otimes e_{h_{S}(b,c)} )) = sk(\rho_{A}(h_{S}(a,b))).
$$

\sss{Example} $U-\mathfrak{Set}$, taking the initial object to be the empty set. The product of the empty set with any set is empty, and so the equality is trivial.

\sss{Example} $n-\mathfrak{Cat}$, taking the initial object to be the empty category.  As above.

\sss{Example} Pointed $\mathfrak{Set}$. The initial object is the set with one element, $\{ \emptyset \}$, the one element distinguished. In this case one would require that composition of any element with the distinguished element should return the original element, as an identity arrow.

\sss{Lemma} The category $WE_{(sk,\alpha,\lambda_{A},\rho_{A})}(A,\otimes)$ of $(A,\otimes)$-enriched sets with an initial object has coproducts.

\sss{$\bold{Definition}$} For any tensor category $(A,\otimes)$ with an action of an initial object $\lambda_{A} : \emptyset_{A} \otimes Id_{A} \rightarrow Id_{A} \leftarrow Id_{A} \otimes \emptyset_{A} : \rho_{A}$, define a functor $Bar_{(\lambda_{A},\rho_{A})} : A \longrightarrow WE_{(sk)}(A,\otimes)$, on objects by
$$
Bar_{(\lambda_{A},\rho_{A})(0)} : a \mapsto (\{ 1,2 \} , h_{Bar_{(\lambda_{A},\rho_{A})}(a)} := \{  ((1,1),\emptyset_{A}),((1,2),a),((2,1),\emptyset_{A}),((2,2),\emptyset_{A}) \} ,
$$
$$
\{((1,1,1),\lambda_{A})(\emptyset_{A}),((1,1,2),\lambda_{A}(a)),((1,2,2),\rho_{A}(a)),
$$
$$
((2,2,2),\rho_{A}(\emptyset_{A})),((2,2,1),\lambda_{A}(a)),((2,1,1),\rho_{A}(a))  \} )
$$

and on arrows by
$$
Bar_{(\lambda_{A},\rho_{A})(1)} : \phi \mapsto \{ ... , ((1,2),\phi),... \}
$$

so that it takes the arrow $\phi$ itself for the non-trivial (non-initial) part of the enrichment.

\sss{$\bold{Definition}$}\label{DefUnit} For any $(A,\otimes,\alpha) \in Ob(\mathfrak{ATCat})$, for any functor $sk : A \longrightarrow B$, for any $(sk)$-action $(\lambda_{A},\rho_{A})$ of an initial object $\emptyset_{A} \in Ob(A)$, we define the category $\mathfrak{Unit}_{(sk)}(A,\otimes)$ and the functor 
$$
\Delta-Unit_{(\lambda_{A},\rho_{A})} : \Delta_{inj} \times_{U-\mathfrak{Cat}} \mathfrak{Unit}(A,\otimes) \longrightarrow WE_{sk}(A,\otimes)
$$ 

by the following.

\ref{DefUnit}.1. The category $\mathfrak{Unit}_{(sk)}(A,\otimes)$ has for objects a sort of class of monads, arrows $\mu : a \otimes a \rightarrow a$ in $A$, and for arrows arrows $f : a \rightarrow b$ which ``respect the multiplication."
$$
\mathfrak{Unit}_{(sk)}(A,\otimes) := ( O :_{t}= \{ (a,\mu) \in Ob(A) \times Arr(A) ; \mu \in Hom_{(A)}(a\otimes a,a) \} ,
$$
$$
\coprod Hom_{0}((a,\mu),(b,\nu)) , Hom_{0} :_{t}= \{ (((a,\mu),(b,\nu)),
$$
$$
\{ f \in Hom_{(A)}(a,b) ; sk(\nu \cdot (f \otimes f)) = sk(f \cdot \mu) \}) \in (O \times O) \times U ; taut \} , ... )
$$

\ref{DefUnit}.2. The functor $\Delta-Unit_{\lambda_{A},\rho_{A}}) := (\Delta-Unit_{0},\Delta-Unit_{1})$ is defined on objects, so that for any $n \in \mathbb{N}$, for any $(a,\mu) \in Ob(\mathfrak{Unit}(A,\otimes))$
$$
\Delta-Unit_{0} : ((0,...,n),(a,\mu)) \mapsto (\{0,...,n\},
$$
$$
h :_{t}=  \{ ((i,j) \} \cup \{ ((i,j) \},
$$
$$
\circ :_{t} = \{ ((i,j,k),\lambda_{A}(\emptyset_{A}) ) \in ; \ulcorner i = j = k\urcorner \text{ and } \ulcorner i < n\urcorner \} \cup
$$
$$
\{ ((i,j,k),\lambda_{A}(a)) \in ; i = j < k\} \cup \{ ((i,j,k),\mu) \in ; i < j < k \} \cup 
$$
$$
\{ ((i,j,k),\rho_{A}(a)) \in ; i < j = k\} \cup \{ ((n,n,n),\rho_{A}(\emptyset_{A})) \} ),
$$

and on arrows, so that for any $(e,f) \in Arr(\Delta_{inj} \times\mathfrak{Unit}(A,\otimes))$,

$$
\Delta-Unit_{1} : (e,f) \mapsto (e ,
$$
$$
\{ ((i,j),id_{\emptyset_{A}}) \in Ob(dom((e,f)))^{2} \times Arr(A) ; i = j \} \cup
$$
$$
\{ ((i,j),f) \in Ob((dom((e,f))^{2} \times Arr(A) ; i < j \} ) \in Arr(WE_{(sk)}(A,\otimes))
$$

\sss{Lemma} If the action $(\lambda_{A},\rho_{A})$ is $(\alpha)$-associative then $\Delta-Unit_{(\lambda_{A},\rho_{A})}$ (and \newline $Bar_{(\lambda_{A},\rho_{A})}$) factors through $WE_{(sk,\alpha)}(A,\otimes)$, i.e. produces $(sk)$-associative enriched sets.

\sss{Lemma}\label{WEUnit} If the product structure $(A,\times_{A})$ has a unit $(I_{0},\lambda_{0},\rho_{0})$, then for any arrow $\mu_{0}$, that $(I_{0},\mu_{0}) \in Ob(\mathfrak{Unit}_{(sk)}(A,\otimes))$, i.e. that $\mu$ gave a sort of monad structure to $I_{0}$, would imply that $(WE_{(sk)}(A,\otimes),\times_{WE_{(sk)}(A,\otimes)})$ had a unit, given by the enriched set $( \{ \emptyset \} , h_{0} \circ_{0})$ with one element, single hom object $h_{0} = I_{0}$, and composition $\circ_{0}(\emptyset,\emptyset,\emptyset) = \mu)$.

\sss{Lemma} If $A$ has an initial object with an action $\lambda_{A},\rho_{A}$, and $F : A \longrightarrow U-\mathfrak{Set}$ is representable, and preserves coproducts, then the functor 
$$
F \circ \bar{Arr}_{(A,\otimes,)} \cong Arr \circ For^{WE_{(sk,F,\rho)}(A,\otimes)}_{U-\mathfrak{Cat}} : WE_{(sk,F,\rho)}(A,\otimes) \longrightarrow U-\mathfrak{Set}
$$

is representable.

\sss{$\bold{Definition}$. of Enriched Sets with Units} For any tensor category $(A,\otimes)$ for any functor $sk : A \longrightarrow B$, for any $(A,\otimes)$-unit$_{(sk)}$  $(I,\lambda_{A},\rho_{A})$ we make the following definitions.

(i). Define a category $WE_{(sk,I,\lambda_{A},\rho_{A})}(A,\otimes) \in Ob(U'-\mathfrak{Cat})$ so that its objects are enriched sets with unit data,
$$
Ob(WE_{(sk,I,\lambda_{A},\rho_{A})}(A,\otimes)) =
$$
$$
\{ (S,h_{S},\otimes_{S},i) \in U ; i \in Hom_{U-\mathfrak{Set}}(S,Arr(A)) \text{ and }
$$
$$
\forall s,t \in S, \ulcorner\ulcorner i(s) \in Hom_{A}(I,h_{S}(s,s) \urcorner \text{ and } 
$$
$$
\ulcorner \forall \lambda^{inv},\rho^{inv} \in Arr(A), 
$$
$$
\ulcorner\ulcorner\ulcorner sk(\lambda_{A}(h_{S}(s,t)) \cdot \lambda^{inv}) = sk(id) \urcorner \text{ and } \ulcorner sk(\lambda^{inv} \cdot \lambda_{A}(h_{S}(s,t))) = sk(id) \urcorner \text{ and }
$$
$$
\ulcorner sk(\rho_{A}(h_{S}(s,t)) \cdot \rho^{inv}) = sk(id) \urcorner \text{ and } \ulcorner sk( \rho^{inv} \cdot \rho_{A}(h_{S}(s,t))) = sk(id) \urcorner\urcorner \Longrightarrow
$$
$$
\ulcorner sk(\circ(s,s,t) \cdot (i(s) \otimes id_{h_{S}(s,t)}) \cdot \lambda_{A}(h_{S}(s,t))^{-1} = sk(id_{h_{S}(s,t)}) =
$$
$$
sk(\circ(s,t,t) \cdot (id_{h_{S}(s,t)} \otimes i(t)) \cdot \rho_{A}(h_{S}(s,t))^{-1} \urcorner\urcorner\urcorner \urcorner \} 
$$

and its arrows are those of $WE_{(sk)}(A,\otimes)$ which preserve the unit, so that for all $\bar{S} = (S,h_{S},\circ_{S},i_{S}), \bar{T} = (T,h_{T},\circ_{T},i_{T}) \in Ob(WE_{(sk,I,\lambda_{A},\rho_{A})}(A,\otimes))$,
$$
Hom_{WE_{(sk,I,\lambda_{A},\rho_{A})}(A,\otimes)}(\bar{S},\bar{T}) = 
$$
$$
\{ (f_{0},f_{1}) \in Hom_{WE_{(sk)}(A,\otimes)}((S,h_{S},\circ_{S},(T,h_{T},\circ_{T})) ; \forall s \in S, sk(f_{1}(s,s)\cdot i_{S}(s)) = sk(i_{T}(s)) \}
$$

(ii). Define the forgetful functor $For^{WE_{(sk,I,\lambda_{A},\rho_{A})}(A,\otimes)}_{WE_{(sk)}(A,\otimes)} : WE_{(sk,I,\lambda_{A},\rho_{A})}(A,\otimes) \longrightarrow$ \newline $WE_{Assoc(sk)}(A,\otimes)$ by forgetting the unit data, $(S,h_{S},\circ_{S},i_{S}) \mapsto (S,h_{S},\circ_{S})$.

\sss{Lemma. Pre-Curry}\label{PreCurry} For any $(A,\otimes)$ with $(sk)$-unit$(A,\otimes)$ data $(I,\lambda_{A},\rho_{A})$, and $(sk)$-unit$(A,\times_{A})$ data $(I_{0},\lambda_{0A},\lambda_{0A})$ for any arrow $c \in Hom_{A}(I_{0},I)$ in $A$ we have the following.

\ref{PreCurry}.1. Define an arrow $Cur_{0}$ of sets by the following. For any arrow of $(A,\otimes)$-enriched sets
$$
F = (F_{0},F_{1}) \in Hom_{WE_{Assoc(sk)}(A,\otimes)}(For^{WE_{(sk,I,\lambda_{A},\rho_{A})}(A,\otimes)}_{WE_{Assoc(sk)}(A,\otimes)}(\bar{S}) \times For^{WE_{(sk,I,\lambda_{A},\rho_{A})}(A,\otimes)}_{WE_{Assoc(sk)}(A,\otimes)}(\bar{T}),\bar{U})
$$

there is an arrow of $(A,\otimes)$-enriched sets
$$
Cur_{0}(F) = F' = (F'_{0},F'_{1}) \in 
$$
$$
Hom_{(WE_{Assoc(sk)}(A,\otimes))}(For^{WE_{(sk,I,\lambda_{A},\rho_{A})}(A,\otimes)}_{WE_{Assoc(sk)}(A,\otimes)}(\bar{S}) , \bar{h}(For^{WE_{(sk,I,\lambda_{A},\rho_{A})}(A,\otimes)}_{WE_{Assoc(sk)}(A,\otimes)}(\bar{T}),\bar{U}))
$$

defined by the following. For all $s \in S$ the object 
$$
F'_{0}(s) = (F'_{00},F'_{01}) \in Ob(\bar{h}(For^{WE_{(sk,I,\lambda_{A},\rho_{A})}(A,\otimes)}_{WE_{Assoc(sk)}(A,\otimes)}(\bar{T}),\bar{U}))
$$

is an arrow in $WE_{Assoc(sk)}(A,\otimes)$ for some maps of sets $F'_{0}$ and $F'_{1}$, such that for all $t,t' \in T$,
$$
F'_{00}(t) = F_{0}(s,t)
$$

and
$$
F'_{01}(t,t') = F((s,s),(t,t'))\cdot ((i_{S}(s) \cdot c) \times id) \cdot \lambda_{0A}(h_{T}(t,t')).
$$

For all $s_{1},s_{2} \in S$, the arrow $F'_{1}(s_{1},s_{2}) : h_{S}(s_{1},s_{2}) \rightarrow \bar{h}(F'_{0}(s_{1}),F'_{0}(s_{2}))$ in $A$ is the colimit arrow given by the enrichment lemma diagram to the object  $(h_{S}(s_{1},s_{2}),\pi,\emptyset)$ in the arrow category by the arrow 
$$
\pi : h_{S}(s_{1},s_{2}) \rightarrow \prod_{t \in T} h_{U}((F'_{0}(s_{1}),F'_{0}(s_{2}))) = \prod_{t\in T} h_{U}( F(s_{1},t),F(s_{2},t))
$$

determined by the map
$$
t \mapsto F((s_{1},s_{2}),(t,t))\cdot (id \times (i_{S}(t) \cdot c)) \cdot \rho_{0A}(h_{S}(s_{1},s_{2})).
$$

\ref{PreCurry}.2. If the $P$-colimits are objects of $P$, then for any two arrows (objects in the hom enriched set)
$$
F = (F_{0},F_{1}), G = (G_{0},G_{1}) \in 
$$
$$
Hom_{WE_{Assoc(sk)}(A,\otimes)}(For^{WE_{(sk,I,\lambda_{A},\rho_{A})}(A,\otimes)}_{WE_{Assoc(sk)}(A,\otimes)}(\bar{S}) \times For^{WE_{(sk,I,\lambda_{A},\rho_{A})}(A,\otimes)}_{WE_{Assoc(sk)}(A,\otimes)}(\bar{T}),\bar{U})
$$

define an arrow
$$
Cur_{1} : h_{\bar{h}(For^{WE_{(sk,I,\lambda_{A},\rho_{A})}(A,\otimes)}_{WE_{Assoc(sk)}(A,\otimes)}(\bar{S}) \times For^{WE_{(sk,I,\lambda_{A},\rho_{A})}(A,\otimes)}_{WE_{Assoc(sk)}(A,\otimes)}(\bar{T}),\bar{U})}(F,G) \rightarrow 
$$
$$
h_{\bar{h}(For^{WE_{(sk,I,\lambda_{A},\rho_{A})}(A,\otimes)}_{WE_{Assoc(sk)}(A,\otimes)}(\bar{S}), \bar{h}( For^{WE_{(sk,I,\lambda_{A},\rho_{A})}(A,\otimes)}_{WE_{Assoc(sk)}(A,\otimes)}(\bar{T}),\bar{U}))}(F',G')
$$

by assigning, to each $s \in S$ the colimit arrow
$$
h_{...}(F,G) \rightarrow h_{...}(F'(s),G'(s))
$$

assigned to the $P$-object given by the projection
$$
h_{...}(F,G) \rightarrow \prod_{(s_{0},t_{0}) \in S \times T} h_{U}(F(s_{0},t_{0}),G(s_{0},t_{0})) \xrightarrow{\pi} \prod_{t_{0} \in T} h_{U}(F(s,t_{0}),G(s_{0},t_{0})),
$$

whence we obtain a $P$-arrow
$$
h_{...}(F,G) \rightarrow \prod_{s \in S} h_{...}(F'(s),G'(s))
$$

whence the colimit arrow
$$
Cur_{1} : h_{...}(F,G) \rightarrow h_{...}(F',G').
$$

The pair 
$$
Cur := ( Cur_{0} , Cur_{1} ) \in
$$
$$
Hom_{WE_{(sk)}(A,\otimes}(\bar{h}( For^{...}_{...}(\bar{S}) \times For^{...}_{...}(\bar{T}) , \bar{U}) , \bar{h}( For^{...}_{...}(\bar{S}) , \bar{h}(For^{...}_{...}(\bar{T}) , \bar{U}) ))
$$

is an arrow of $(A,\otimes)$-enriched sets.

\ref{PreCurry}.3. If the $P$-colimits are objects of $P$, with monic arrows into the relevant products, then this is a natural transformation of functors
$$
(WE_{(sk,I,\lambda_{A},\rho_{A})}(A,\otimes))^{opp} \times (WE_{(sk,I,\lambda_{A},\rho_{A})}(A,\otimes))^{opp} \times WE_{Assoc(sk)}(A,\otimes) \longrightarrow U'-\mathfrak{Set}
$$

\sss{Remark} One would expect, that if the tensor structure were the product structure, then the arrow $c$ of the previous should be the identity.

\sss{Lemma}\label{EnrichProd} Let $(A,\otimes)$ be a tensor category with a functorial product $ \times_{A} : A \times A \longrightarrow A$, and $F : I \longrightarrow A$, $G : J \longrightarrow A$ any functors with limits $(l_{F},\lambda_{F})$ and $(l_{G},\lambda_{G})$. Temporarily define the arrow 
$$
\tau : colim(\times_{A} \cdot (F \times_{U-\mathfrak{Cat}} G)) \rightarrow colim(F) \times_{A} colim(G)
$$

to be that which is naturally induced by the assignment to each $(i,j) \in Ob(I \times J)$ of the arrow $t : \times(F(i),F(j)) \rightarrow colim(F) \times_{A} colim(G)$ such that
$$
\pi_{1} \cdot t = \lambda_{F}(id_{l_{F}})(i) \text{ and } \pi_{2} \cdot t = \lambda_{G}(id_{l_{G}})(j).
$$

If $\tau$ is an isomorphism, then for any $\bar{S},\bar{T},\bar{U} \in Ob(WE_{(sk)}(A,\otimes)$, 
$$
\bar{h}(\bar{S},\bar{T} \times \bar{U}) \cong \bar{h}(\bar{S},\bar{T}) \times \bar{h}(\bar{S},\bar{U}).
$$

\sss{$\bold{Definition}$. of the Skeleton Quotient}\label{SkelQuot} For any $(A,\otimes,\alpha) \in Ob(\mathfrak{ATCat})$, for any functor $sk : A \longrightarrow B$, for any $(sk)$-unit $(I,\lambda_{A},\rho_{A}) \in Ob(A) \times Arr(Hom_{U-\mathfrak{Cat}^{2}}^{(1)}(A,A)) \times Arr(Hom_{U-\mathfrak{Cat}^{2}}^{(1)}(A,A))$, such that the arrows of functors
$$
\lambda_{A} : Id_{A} \otimes - \rightarrow Id_{A}
$$
$$
\rho_{A} : - \otimes Id_{A} \rightarrow Id_{A}
$$ 

are isomorphisms, we define the functor
$$
SK_{(A,\otimes,sk,\alpha,I,\lambda_{A},\rho_{A})} : WE_{Assoc(sk)}(A,\otimes) \longrightarrow Q
$$

by the following. Temporarily define $Q \in Ob(U'-\mathfrak{Cat})$ so that $$
Ob(Q) = Ob(WE_{Assoc(sk)}(A,\otimes)),
$$ 

and for any $\bar{S} = (S,h_{S},\circ_{S}), \bar{T} = (T,h_{T},\circ_{T}) \in Ob(WE_{Assoc(sk)}(A,\otimes)$,
$$
Hom_{Q}(\bar{S},\bar{T}) = \{ [F]_{\sim(R)} \in \bold{2}^{Hom_{WE_{Assoc(sk)}(A,\otimes)}(\bar{S},\bar{T})} ; taut \},
$$

hold, where the equivalence relation $R$ is temporarily defined so that for any $\bar{S},\bar{T} \in Ob(WE_{Assoc(sk)}(A,\otimes))$, for any $F,G \in Hom_{WE_{Assoc(sk)}(A,\otimes)}(\bar{S},\bar{T})$, $F \sim_{R} G$ iff there exist
$$
\phi \in Hom_{A}(I,\bar{h}_{WE_{Assoc(sk)}(A,\otimes)}(\bar{S},\bar{T})(F,G))
$$
$$
\psi \in Hom_{A}(I,\bar{h}_{WE_{Assoc(sk)}(A,\otimes)}(\bar{S},\bar{T})(G,F))
$$

such that
$$
sk(\bar{\circ} \cdot (\psi \otimes id_{\bar{h}_{...}(F,G)}) \cdot \lambda_{A}(\bar{h}_{...}(F,G)) \cdot \bar{\circ} \cdot (\phi \otimes id_{\bar{h}_{...}(F,F)}) \cdot \lambda_{A}(\bar{h}_{...}(F,F)) = sk(id_{\bar{h}_{...}(F,F)})
$$

and
$$
sk(\bar{\circ} \cdot (\psi \otimes id_{\bar{h}_{...}(G,F)}) \cdot \lambda_{A}(\bar{h}_{...}(G,F)) \cdot \bar{\circ} \cdot (\phi \otimes id_{\bar{h}_{...}(G,G)}) \cdot \lambda_{A}(\bar{h}_{...}(G,G)) = sk(id_{\bar{h}_{...}(G,G)})
$$

are both true. We define $SK_{(A,\otimes,sk,\alpha,I,\lambda_{A},\rho_{A})}$ to be the quotient functor, so that
$$
SK_{(A,\otimes,sk,\alpha,I,\lambda_{A},\rho_{A})(0)} = id_{Ob(WE_{Assoc(sk)}(A,\otimes))}
$$

and
$$
SK_{(A,\otimes,sk,\alpha,I,\lambda_{A},\rho_{A})(1)} : F \mapsto [F]_{\sim(R)}.
$$

Generally, when the accompanying data are understood, we will informally write $\bar{sk} = SK_{(A,\otimes,sk,\alpha,I,\lambda_{A},\rho_{A})}$.

\se{Higher Categories and Constellations}

According to the literature (\cite{Lein},\cite{LurTop}) one desires that the higher categorical structures ought not to satisfy equalities, but equivalences of some sort. We use the ``skeleton functors" and $(sk)$-limits throughout to implement this. The first section defines $n$-categories, and the enriched set of $n$-categories, by the usual inductive intuition. The second section introduces a few book-keeping notions. The enriched set of $n$-categories is not associative in a satisfactory sense.

The third section introduces constellations, certain constructions of enriched sets, which are under certain conditions associative. We give, in the ``Lens" theorem, a formalism for the construction of arrows of $(WE_{(sk)}(A,\otimes),\times)$-enriched sets from individual constellations to the enriched set of $(A,\otimes)$-enriched sets.

\sus{n-Categories}

An n-Category is defined inductively as an object in the category of (n-1)-enriched categories. $n$-Categories with their basic structures are inductively defined, referring to each other (and therefore inseparable).

\sss{
The inductive  construction of $n$-categories
}
We define, inductively and simultaneously, the 
\ben
``forgetful functors" (``objects functors'') $F(n)$, 
\i
natural transformations $\rho (n)$, 
\i
the ``associators" $\alpha(n)$, 
\i
the ``product functors" $\times (n)$, 
\i
``symmetrizers" $\sigma (n)$, 
\i
``unit objects" $I(n)$, 
\i
right and left unit arrows $\rho_{u}(n)$, $\lambda_{u}(n)$, 
\i
$(n)-equivalence$ of
$(n)$-categories,
\i
$(n)$-equivalence of $(n)$-functors, 
\i
the $(n)$-skeleton functor $sk(n)$, and 
\i
the $U'$-category $U(n)-\mathfrak{Cat}$. 
\een
Here, for any $n\in\mathbb{N}$ the category of $(n)$-categories $U(n)-\mathfrak{Cat}$ 
is the category of sets that are weakly enriched over the category 
$U(n-1)-\mathfrak{Cat}$ of $(n-1)$-categories. 

\sss{$\bold{Definition}$ of n-Category}\label{nCat}
Assuming that we have defined these objects for all integers $\le n$ we define them for $n+1$.

\ref{nCat}.1.
The ``forgetful", or ``objects" functor
is defined on $n$-categories and it takes an $n$-category (an enriched set) to the underlying set 
$$
F(n+1) := Yo^{opp}_{((U,n+1)-\mathfrak{Cat}) (0)}(I(n+1)) : (U,n+1)-\mathfrak{Cat} \longrightarrow U-\mathfrak{Set}
$$

\ref{nCat}.2.
Define a natural transformation
$$
\rho(n) : \times_{U'-\mathfrak{Set}} \cdot (F(n) \times_{U'-\mathfrak{Cat}} F(n)) \rightarrow F(n) \cdot \times (n)
$$

by the identity maps. 

\ref{nCat}.3.
Define the $(n)$-associator
$$
\alpha (n+1) : \times(n+1) \cdot (\times(n+1) \times_{U'-\mathfrak{Cat}}  id_{U'-\mathfrak{Cat}} ) \rightarrow \times(n+1) \cdot (id_{(U,n+1)-\mathfrak{Cat}} \times_{U'-\mathfrak{Cat}} \times(n+1))
$$

as arrow of functors $((U,n+1)-\mathfrak{Cat})^{3} \longrightarrow (U,n+1)-\mathfrak{Cat}$, defined on objects by the associator and on hom objects by $\alpha (n)$. 

\ref{nCat}.4.
Define the $(n)$-product functor
$$
\times(n+1) : (U,n+1)-\mathfrak{Cat} \times_{U'-\mathfrak{Cat}} (U,n+1)-\mathfrak{Cat} \longrightarrow (U,n+1)-\mathfrak{Cat}
$$

on objects by the usual product functor (arrow in $U'-\mathfrak{Cat}$), defined on objects by the usual product functor and on hom objects by $\times (n)$, $\sigma (n)$, and $\alpha(n)$. 

\ref{nCat}.5.
Define the symmetrizing transformation
$$
\sigma (n+1) : \times(n+1) \rightarrow \times \sigma_{U'-\mathfrak{Cat}}
$$

on underlying objects by the usual symmetrizer and on hom objects by $\times(n)$ and $\sigma(n)$. 

\ref{nCat}.6.
Define the unit object
$$
I(n+1) \in Ob((U,n+1)-\mathfrak{Cat})
$$

having $\{ \emptyset \}$ as its underlying set, and $I(n)$ for the hom object. 

\ref{nCat}.7.
Define the left and right unit arrows
$$
\rho_{u}(n+1), \lambda_{u}(n+1) \in Arr(Hom_{U'-\mathfrak{Cat}^{2}}^{(1)}((U,n+1)-\mathfrak{Cat},(U,n+1)-\mathfrak{Cat}))
$$
$$
\rho_{u}(n+1) : - \times I(n+1) \rightarrow Id
$$
$$
\lambda_{u}(n+1) : I(n+1)\times - \rightarrow Id
$$

By the usual units on objects and $\rho_{u}(n)$ and $\lambda_{u}(n)$ on the hom objects. 

\ref{nCat}.8.
Define, for any $C,D \in Ob((U,n+1)-\mathfrak{Cat})$, the statement

$(C,D) \text{ are }(n+1)equivalent_{0} \Longleftrightarrow$

There exist $F : C \rightarrow D$, $G : D \rightarrow C$, such that for any $(c_{1},c_{2}) \in Ob(C)$, $(d_{1},d_{2}) \in Ob(D)$, $F_{(1)}(c_{1},c_{2})$ and $G_{(1)}(d_{1},d_{2})$ are $(n)$-equivalences, and  $(G\cdot F,id_{C}), (F\cdot G, id_{D})$ are $(n+1)$-equivalent$_{1}$.

\ref{nCat}.9.
Define, for any $C,D \in Ob((U,n+1)-\mathfrak{Cat})$, for any $F,G \in Hom_{(U,n+1)-\mathfrak{Cat})}(C,D)$, the statement 

$(F,G)$ are $(n+1)-equivalent_{1} \Longleftrightarrow$ 

There exist 
$$
\phi \in F(n)(\bar{h}_{WE((U,n)-\mathfrak{Cat},\times(n))(sk(n))}(C,D)(F,G))
$$
$$
\psi \in F(n)(\bar{h}_{WE((U,n)-\mathfrak{Cat},\times(n))(sk(n))}(C,D)(G,F))
$$, 

such that the various arrows
$$
\bar{h}_{WE((U,n)-\mathfrak{Cat},\times(n))(sk(n))}(C,D)(F,F) \rightarrow \bar{h}_{WE((U,n)-\mathfrak{Cat},\times(n))(sk(n))}(C,D)(F,G)
$$
$$
\bar{h}_{WE((U,n)-\mathfrak{Cat},\times(n))(sk(n))}(C,D)(G,G) \rightarrow \bar{h}_{WE((U,n)-\mathfrak{Cat},\times(n))(sk(n))}(C,D)(G,F)
$$

given by the composition of $\bar{\circ}$, the arrow $\bar{I(n)} \rightarrow \bar{h}_{WE((U,n)-\mathfrak{Cat},\times(n))(sk(n))}(C,D)(F,G)$ associated to $\phi$ or $\psi$ (see \ref{HomEnr} part(iv).) and a unit arrow ($\lambda_{u}(n)$ or $\rho_{u}(n)$), are $(n)$-equivalences$_{0}$ (Slightly loose usage. Adapt part (8).) (i.e. they $(sk(n))$-invert one another).

Roughly speaking there are $(sk(n))$-natural transformations between $F$ and $G$, which induce forward and backward compostion functors by the unit and enrichment lemma, which are $(n)$-equivalences, and such that $\phi \circ \psi$ and $\psi \circ \phi$ induce $(n)$-equivalent functors to the identities for the respective hom objects. 

\ref{nCat}.10.
Define the $(n)$-skeleton functor $sk(n)$ as a quotient functor
$$
sk(n) : (U,n)-\mathfrak{Cat} \longrightarrow Q 
$$

where $Q$ is the category defined by
$$
Ob(Q) = Ob((U,n)-\mathfrak{Cat})
$$
$$
Hom_{(Q)}(C,D) :_{t}= \{ [ F ]_{(n)eq} \in \bold{2}^{(Hom_{((U,n)-\mathfrak{Cat})}(C,D))} ; F \in Hom_{((U,n)-\mathfrak{Cat})}(C,D) \}
$$

where $[F]_{(n)eq} = [G]_{(n)eq}$ iff $(F,G)$ are $(n)$-equivalent. 

\ref{nCat}.11.
Define the category of $(n+1)$-categories
$$
(U,n+1)-\mathfrak{Cat} := WE_{Ass((U,n)-\mathfrak{Cat},\times(n)(sk(n),\alpha(n))}
$$

to be the category of sets $(sk(n))$-associatively enriched over over the category of $(n)$-categories

\sss{} Parts (ii) and (iii) of the following lemma give construction for limits and colimits in $WE(A,\otimes)$, to be applied to the (co)limits appearing in the construction of the enriched hom sets.

\sss{Lemma on Limits and Colimits in $WE(A,\otimes)$} \label{LemLimWE}
For any $(A,\otimes) \in Ob(\mathfrak{TCat})$,

\ref{LemLimWE}.1. $For : WE_{(sk)}(A,\otimes) \longrightarrow WE_{(term \circ sk)}(A,\otimes)$ is faithful, where $term :$ $codom(sk)$ $\longrightarrow$ $\star$ is the functor whose codomain is the terminal category. I.e. one forgets that one had had a composition requirement.

\ref{LemLimWE}.2. The limit of $F : I \longrightarrow WE(A,\otimes)$ can be constructed by the limit of the underlying sets and $(a_{i},b_{i})_{i\in I} \mapsto lim F'_{1}$, where $F_{1}' : I \longrightarrow A$ is defined on objects by
$$
F_{1(0)}' : j \mapsto h_{F_{(0)}(j)}(a_{j},b_{j}) )
$$

\ref{LemLimWE}.3. For any $F : I \longrightarrow WE_{(sk)}(A,\otimes)$, if $\tau : colim \circ \otimes \rightarrow \otimes \circ (colim \times _{\mathfrak{Cat}} colim )$ is an isomorphism where hom objects $h_{F_{(0)}(i)}(x,y)$ are concerned, then the colimit can be similarly constructed, by 
$$
([(a,i)],[(b,j)]) \mapsto colim F'_{1} \cdot cob\downarrow_{(Hom^{(1)}_{\mathfrak{Cat}^{2}}((\{ 1,2 \} ,...), I)} (ob_{(I)}(i) \cup ob_{(I)}(j),\Delta_{ (\{ 1,2 \})} )
$$

i.e. taking the colimit of all hom objects below both $i$ and $j$. Define composition by the arrow induced by tensoring the colimit arrows assigned to $([(a,i)],[(b,j)])$ and $([(b,j)],[(c,k)])$, composed with the inverse of $\tau$.

\begin{rem} 
The explicit description of limits and (co)limits is applied to verify in the following lemma the isomorphism required for part (ii) of \ref{HomEnr}.
\end{rem}

\begin{lem}
$\forall n \in \mathbb{N}$, $colim \cdot \times(n) \rightarrow \times(n) \cdot (colim \times _{\mathfrak{Cat}} colim )$ is an isomorphism.
\end{lem}
\begin{proof}
On the level of sets, this is the isomorphism given by $[(a_{i},b_{j})] \mapsto ([a_{i}],[b_{j}])$. By the previous lemma the product of enriched sets is given by taking the products of their hom objects, so that $\tau_{n+1} : colim \circ \times(n+1) \rightarrow \times(n+1) \circ (colim \times _{\mathfrak{Cat}} colim )$ is determined by $\tau_{0}$ on underlying set and $\tau_{n}$ on hom objects. By induction, $\tau_{n}$ is for any $n$ an isomorphism.
\end{proof}

The ``meaning" of the following theorem consists in the special cases of parts (iii) and (iv) of \ref{PushPull}.

\sss{$\bold{Theorem}$ on $(U,n)-\mathfrak{Cat}$}
The category $(U,n)-\mathfrak{Cat}$ is weakly enriched over itself. I.e. $((U,n+1)-\mathfrak{Cat},\bar{h}_{WE((U,n)-\mathfrak{Cat},\times (n))(sk(n)},\bar{\circ} (n) ) \in Ob(WE((U,n+1)-\mathfrak{Cat},\times (n+1))$. The hom set agrees with that given by applying the objects functor $Ob = F(n)$ to the hom $n$-category, i.e.  $Ob\circ\bar{Hom}_{(U,n)-\mathfrak{Cat}} \cong Hom_{(U,n)-\mathfrak{Cat}}$.

\begin{proof}
One must check that the constructions of \ref{HomEnr} (see part(ii)) and \ref{PushPull} can be applied at each step.

$sk(n)$-associativity is part of the definition of $(U,n)-\mathfrak{Cat}$. The isomorphism of the previous lemma is the only other requirement.
\end{proof}

\begin{rem}
The restriction of $WE((U,n)-\mathfrak{Cat},\times (n))$ to the subcategory of $(sk(n))$-associative enrichments is necessary for the construction of the hom set enrichment, which is necessary for the definition of the next skeleton functor, $sk(n+1)$).
\end{rem}

\begin{rem}
That $(U,n+1)-\mathfrak{Cat}$ as an enriched set is $sk(n+1)$-associative (and therefore properly an $(n+2)$-category) was expected, but not yet clear to me. By part (iv) of \ref{PushPull} it is associative with respect to the objects functor and $sk(n)\circ \bar{Arr}_{((U.n)-\mathfrak{Cat},\times (n))}$, i.e. it is $sk(n)$-associative with respect to each hom object ($n$-category).  The difficulty seems to be in inferring, from the arrows giving the equivalences within the hom objects, arrows giving equivalences from without. I suspect that this should be easier to do for particular types of $n$-categories.
\end{rem}

\begin{ex}
$(2)-\mathfrak{Cat} \in Ob(WE((2)-\mathfrak{Cat},\times(2)))$. The skeleton is used at the level of the hom objects, so that only the usual skeleton, $sk(1)$, is seen in this case. The objects are enriched sets.
$$
O = Ob((2)-\mathfrak{Cat}) = \{ \bar{C} = (C,h,\circ ) \}
$$

where the composition is $(sk)$-associative, where $sk = sk(1) : \mathfrak{Cat} \longrightarrow Q$ is the quotient functor determined by identifying isomorphic arrows (functors). The arrows are arrows of enriched sets
$$
\Phi = (\Phi_{0},\Phi_{1}) : ( C,h_{C} ,\circ_{C}) \rightarrow (D,h_{D},\circ_{D})
$$

respecting composition after the application of $(sk)$.

By the Hom-enrichment construction one associates to any $C,D \in Ob((2)-\mathfrak{Cat})$, $\Phi, \Psi \in Hom_{((2)-\mathfrak{Cat})}(C,D)$, the category $P_{\Phi,\Psi}$ of all arrows $( x \xrightarrow{f} \prod_{c \in Ob(C)} h_{D}(\Phi(c),\Psi(c))$ satisfying the $(sk)$-commutativity requirement. $p : P_{\Phi,\Psi} \longrightarrow \mathfrak{Cat}$ is the functor defined by $(( x, f ) \mapsto x)_{(x,f) \in Ob(P_{\Phi,\Psi})}$. By definition $\bar{h}_{2-\mathfrak{Cat}}(a,b)(\Phi,\Psi) := colim P_{\Phi,\Psi}$

The description of the enrichment on $(2)-\mathfrak{Cat}$ requires, for any $(C,D,E) \in O$, an arrow
$$
(\bar{h}_{2-\mathfrak{Cat}} (C,D) \times \bar{h}_{2-\mathfrak{Cat}}(D,E) \xrightarrow{\circ} \bar{h}_{2-\mathfrak{Cat}}(C,E)) \in Arr((2)-\mathfrak{Cat})
$$

representing composition. That the above is an arrow in $(2)-\mathfrak{Cat}$, interpreted, means that for any $\Phi_{1},\Phi_{2},\Phi_{3} \in Hom_{((2)-\mathfrak{Cat})}(C,D), \Psi_{1},\Psi_{2},\Psi_{3} \in Hom_{((2)-\mathfrak{Cat})}(D,E)$, 
$$
F \cong G \in Hom_{\mathfrak{Cat}}(
$$
$$
\bar{h}_{(2)-\mathfrak{Cat}}(C,D)(\Psi_{1},\Psi_{2}) \times \bar{h}_{(2)-\mathfrak{Cat}}(C,D)(\Psi_{2},\Psi_{3})) \times (\bar{h}_{2-\mathfrak{Cat}}(D,E)(\Phi_{1},\Phi_{2}) \times \bar{h}_{2-\mathfrak{Cat}}(\Phi_{2},\Phi_{3})),
$$
$$
\bar{h}_{2-\mathfrak{Cat}}(C,E)(\Phi_{1}\circ\Psi_{1},\Phi_{3}\circ\Psi_{3}))
$$

where
$$
F :_{t}= \bar{\circ}(C,E) \circ (\Phi_{1*} \times \Psi_{3}^{*}) \circ (\bar{\circ}(C,D) \times \bar{\circ}(D,E))
$$
$$
G :_{t}= \bar{\circ}(C,E) \circ (\bar{\circ}(C,E) \times \bar{\circ}(C,E)) \circ ((\Phi_{1*} \times \Psi_{2}^{*}) \times (\Phi_{2*} \times \Psi_{3}^{*})) \circ \sigma 
$$

where $\bar{\circ}(C,D)$ denotes the enriched composition in $\bar{h}_{2-\mathfrak{Cat}}(C,D)$. I.e., there is a function (arrow of sets) $\alpha : Ob(dom(F)) = Ob(dom(G)) \rightarrow Arr(\mathfrak{Cat})$ defining a natural isomorphism between the functors $F$ and $G$.

\end{ex}

\begin{pro}
If the $P$-colimit inclusion condition is satisfied for $(U,n)-\mathfrak{Cat}$, regarding the construction of the hom enrichment, then it is satisfied for $(U,n+1)-\mathfrak{Cat}$ as well. I.e., the two arrows $colim \ p \otimes e_{0} \rightarrow \prod_{c \in Ob(C)} h_{D}(\Phi(c),\Psi(c))$, one from right composition and the other from left composition, are $(n+1)$-equivalent.
\end{pro}
\begin{proof}
The forgetful functor is at each step given by the objects functor. In this case, $P$ is given by all arrows $(a\xrightarrow{\pi} \prod_{c \in Ob(C)} h_{D}(\Phi(c),\Psi(c))) \in Arr((U,n)-\mathfrak{Cat})$, such that for any arrow $(e_{0} \xrightarrow{e} h_{C}(x,y)) \in Arr((U,n)-\mathfrak{Cat})$ into a hom object in $C$, the two arrows (if $\otimes= \times(n)$) 
$$
r_{(a)},l_{(a)} :a \otimes e_{0} \rightarrow \bar{h}_{D}(\Phi(x),\Psi(y))
$$

one given by composition with $e_{0}$ on one side and the other by composition on the other, are $sk(n)$-equivalent. Therefore a choice of an $(n+1)$-equivalence of $(n+1)$-functors is still a choice of
$$
\phi \in  F(n)(\bar{h}_{WE((U,n)-\mathfrak{Cat},\times(n))(sk(n))}(r,l))
$$
$$
\psi \in  F(n)(\bar{h}_{WE((U,n)-\mathfrak{Cat},\times(n))(sk(n))}(l,r))
$$

where $\bar{h}_{WE((U,n)-\mathfrak{Cat},\times(n))(sk(n))}(r,l)$ is itself by construction a colimit of the domain object functor
$$
p = dob\downarrow_{((U,n)-\mathfrak{Cat})} (id_{(U,n)-\mathfrak{Cat}},ob_{((U,n)-\mathfrak{Cat})}(\prod_{x \in Ob(t_{0})} h_{D}(r_{(0)}(x),l_{(0)}(x)))) \circ \varepsilon : 
$$
$$
P \longrightarrow (U,n)-\mathfrak{Cat}.
$$

By the inclusion condition for the $n$ case the hom object assigned to $r$ and $l$ has a monic arrow into the product of hom objects $h_{D}(r_{(0)}(x),l_{(0)}(x))$. By the isomorphism of the previous lemma and the construction of the colimit in $WE_{(sk)}(A,\otimes)$ in the lemma before that, an arrow of functors $\phi \in Ob(\bar{h}_{WE((U,n)-\mathfrak{Cat},\times(n))(sk(n))}(l_{colim \ p},r_{colim \ p})))$ is a map of sets
$$
\phi : Ob(colim \ p \times (n) e_{0}) \cong Ob(colim \ p) \times Ob(e_{0}) = 
$$
$$
\{ (a, \pi) ; a \in dom(\pi) \text{ and } (dom(\pi),\pi) \in Ob(P) \} \times Ob(e_{0})
$$
$$
\rightarrow \bigcup Ob(h_{D}(l_{colim \ p} ,r_{colim \ p} ))
$$

Claim - That a choice argument implies the existence of a natural isomorphism $\phi$ from the natural isomorphism $\phi_{i}$.
\end{proof}

\sus{Addresses}


We introduce the notion of an address, which is sequence of hom objects, each nested within the previous by the n-categorical enrichment. It is essentially a book-keeping tool, meant to record the ``location of a $k$-arrow within an $n$-category."

\sss{$\bold{Definition}$ of the Empty $n$-Category}
$\emptyset_{U(1)-\mathfrak{Cat}} := (\emptyset , \emptyset, \emptyset, \emptyset, \emptyset ) \in Ob(U-\mathfrak{Cat}) = Ob(U(1)-\mathfrak{Cat})$ is the empty category, and $\forall n \in \mathbb{N}, \emptyset_{U(n+2)-\mathfrak{Cat}} := (\emptyset_{U(1)-\mathfrak{Cat}}, \emptyset, \emptyset) \in Ob(U(n+2)-\mathfrak{Cat})$ is the empty (n+2)-category.

\sss{$\bold{Definition}$ of Addresses}\label{DefAdd}

We define two address functions, one for objects in $(U,n)-\mathfrak{Cat}$ and one for arrows.

\ref{DefAdd}.1. For any $n \in \mathbb{N}$, $fAdd_{U(n)0} : Ob(U(n)-\mathfrak{Cat}) \rightarrow U'$ is defined to be the function which sends an $n$-category $x \in Ob(U(n))-\mathfrak{Cat})$ to the set of functions $\alpha : \{ 1,...,j \} \rightarrow U'$ such that for any $k \in \{ 1,...,j \}$, where $j \in \{0,...,n\}$, 
$$
\alpha (k) = ( a(k),b(k),C(k),h(k),\circ (k)
$$
$$
h(k)(a(k),b(k)) = (C(k+1),h(k+1),\circ (k+1) )
$$
$$
a(k),b(k) \in Ob(C(k))
$$
$$
x = (C(0),h(0),\circ(0))
$$
For any $n \in \mathbb{N}$, $Add_{U(n)0} : Ob(U(n)-\mathfrak{Cat}) \rightarrow U'$ is the function which sends an $n$-category $x$ as above to the set of functions $\alpha : \{ 0,...,j \} \rightarrow U'$ such that there exist $a,b,C,h,\circ$ for which $\alpha = (a(k),b(k))_{k \in \{ 0,...,j \}}$ and $(a(k),b(k),C(k),h(k),\circ (k) ) \in fAdd_{U(n)0}(x)$.

These assign to each $n$-category its set of ``(full) addresses," being sequences \newline $(a(i),b(i),\mathcal{C}(i),h(i),\circ (i))$ such that $(a(i+1),b(i+1))$ is a pair of objects in the base category $\mathcal{C}(i)$ of the $(n-i-1)$-category associated to the previous pair $(a(i),b(i))$ by the enrichment. $fAdd$ refers to the former list and $Add$ to the truncated latter.

The ``length," $|\alpha | = | (a,b)|$, will denote its order as a set.

\ref{DefAdd}.2. For any $n \in Ob(\mathbb{N})$, $Add_{U(n)1} : Arr(U(n)-\mathfrak{Cat}) \rightarrow U'$ is defined to be the function which sends $\phi \in Arr(U(n)-\mathfrak{Cat})$ to a function
$$
S : Add_{U(n)0}(dom(\phi)) \rightarrow \bigcup_{k \in \mathbb{N}} Arr(U(k)-\mathfrak{Cat})
$$
defined inductively, by requiring that
$$
S : \emptyset \mapsto \phi
$$
and that for any $(a,b) \in Add_{U(n)0}(dom(\phi))$, for any $\bar{\phi} \in Arr(U(n-|(a,b)|)-\mathfrak{Cat})$, $S(a,b) := \bar{\phi}$ iff there exists $(a_{0},b_{0}) \in Add_{U(n)0}(dom(\phi))$ such that
$$
|(a_{0},b_{0})|+1 = |(a,b)| \text{ and } (a,b)|_{0,...,|(a_{0},b_{0})|-1} = (a_{0},b_{0})
$$

and there exists $\psi = ((f_{0},f_{1}),f_{2}) \in Arr(U(n-|(a,b)|+1)-\mathfrak{Cat})$, such that 
$$
\psi = S(a_{0},b_{0}) \text{ and } f_{2}(a(|(a,b)|),b(|(a,b)|)) = \bar{\phi}
$$

This associates to every arrow of $n$-categories a function which sends an address for the domain category to the arrow of ($n-k$)-categories assigned to it by the original arrow.

\begin{rem}
That the above definition consists of two maps, one for $n$-categories and the other for arrows of $n$-categories, suggests some functor giving an alternate description of $n$-categories.
\end{rem}

\sss{$\bold{Definition}$ of the Functors $Inc^{U(m)-\mathfrak{Cat}}_{U(n)-\mathfrak{Cat}}$ and $For^{U(m)-\mathfrak{Cat}}_{U(n)-\mathfrak{Cat}}$} \label{DefIncFor}
For any $n,m \in \mathbb{N} \backslash \{ 0 \}$ such that $n < m$, define functors $Inc^{U(m)-\mathfrak{Cat}}_{U(n)-\mathfrak{Cat}} : U(n)-\mathfrak{Cat} \longrightarrow U(m)-\mathfrak{Cat}$ and $For^{U(m)-\mathfrak{Cat}}_{U(n)-\mathfrak{Cat}} : U(m)-\mathfrak{Cat} \longrightarrow U(n)-\mathfrak{Cat}$ inductively, by the following.

\ref{DefIncFor}.1. For any $x = (C,h,\circ) \in Ob(U(n+1)-\mathfrak{Cat})$,
$$
Inc_{0U(n+1)-\mathfrak{Cat}}(x) :_{t}= (C,Inc_{0U(n)-\mathfrak{Cat}} \circ h, Inc_{1U(n)-\mathfrak{Cat}} \circ \circ )
$$
and for any $\phi = (\phi_{0},\phi_{2}) \in Arr(U(n+1)-\mathfrak{Cat})$,
$$
Inc_{1U(n+1)-\mathfrak{Cat}}(\phi) :_{t}= (\phi_{0},Inc_{U(n)-\mathfrak{Cat}}(\phi_{2}))
$$
so that $Inc_{U(n+1)-\mathfrak{Cat}} := (Inc_{0U(n+1)-\mathfrak{Cat}},Inc_{1U(n+1)-\mathfrak{Cat}}) : U(n+1)-\mathfrak{Cat} \longrightarrow U(n+2)-\mathfrak{Cat}$.

Now temporarily define $Inc_{U(1)-\mathfrak{Cat}} : U-\mathfrak{Cat} \longrightarrow U(2)-\mathfrak{Cat}$ to be the functor which sends a category $C$ to the 2-category with enrichment $h_{C}(a,b) := (Hom_{C}(a,b), \{ id_{f} ; f \in Hom_{C}(a,b) \} ,... )$ given by attaching only identity arrows. Define
$$
Inc^{U(m+1)-\mathfrak{Cat}}_{U(n)-\mathfrak{Cat}} := Inc_{U(m)-\mathfrak{Cat}} \circ Inc^{U(m)-\mathfrak{Cat}}_{U(n)-\mathfrak{Cat}}, \text{ and }
$$
$$
Inc^{U(2)-\mathfrak{Cat}}_{U(1)-\mathfrak{Cat}} := Inc_{U(1)-\mathfrak{Cat}}
$$

\ref{DefIncFor}.2. Similarly, for any $x = (C,h,\circ ) \in Ob(U(m+1)-\mathfrak{Cat})$,
$$
For_{0}^{U(m+1)-\mathfrak{Cat}}(x) :_{t}= (C,For_{0}^{U(m)-\mathfrak{Cat}} \circ h ,For_{1}^{U(m)-\mathfrak{Cat}} \circ \circ )
$$

and for any $\phi = (\phi_{0},\phi_{2}) \in Arr(U(m+1)-\mathfrak{Cat})$,
$$
For^{U(m+1)-\mathfrak{Cat}}_{1}(\phi) :_{t}= (\phi_{0},For^{U(m)-\mathfrak{Cat}}_{1}(\phi_{2}))
$$

so that $For^{U(m+1)-\mathfrak{Cat}} :_{t}= (For^{U(m+1)-\mathfrak{Cat}}_{0},For^{U(m+1)-\mathfrak{Cat}})_{1}) : U(m+1)-\mathfrak{Cat} \longrightarrow U(m)-\mathfrak{Cat}$.

Now temporarily define $For^{U(2)-\mathfrak{Cat}} : U(2)-\mathfrak{Cat} \longrightarrow U-\mathfrak{Cat}$ to be the functor which forgets the enrichment. Define
$$
For^{U(m+1)-\mathfrak{Cat}}_{U(n)-\mathfrak{Cat}} := For^{U(m)-\mathfrak{Cat}}_{U(n)-\mathfrak{Cat}} \circ For^{U(m+1)-\mathfrak{Cat}} \text{ and }
$$
$$
For^{U(n)-\mathfrak{Cat}}_{U(n)-\mathfrak{Cat}} := id_{U(n)-\mathfrak{Cat}}
$$

\begin{lem}
$\forall n \in \mathbb{N}$, $U(n+1)-\mathfrak{Cat}$ has products and coproducts.
\end{lem}
\begin{proof}
For products, by induction on $n$. At the base take the usual product category. For any tuple $(x_{i})_{i \in S}$, $(y_{i})_{i \in S}$, use the inductive step to take the product $\prod_{i \in S} h_{\mathcal{C}_{i}}(x_{i},y_{i})$.

For coproducts, at the base take the usual coproduct category (objects are the disjoint union. $Hom_{\coprod_{i \in S}}((a,j),(b,k))$ is $\empty$ for $j \neq k$, and $Hom_{\mathcal{C}_{j}}(a,b)$ for $j = k$). If $n \geq 1$, then for the enrichment, $h_{\coprod_{i \in S} \mathcal{C}_{i}}((a,j),(b,k))$  is $\emptyset_{U(n)-\mathfrak{Cat}}$ for $j \neq k$, and $h_{\mathcal{C}_{j}}(a,b)$ for $j = k$.
\end{proof}

\sss{$\bold{Definition}$ of Products and Coproducts}
$\prod _{U(n)-\mathfrak{Cat}}$ and $\coprod _{U(n)-\mathfrak{Cat}}$ will be functions \newline $\bigcup_{S \in U} Hom_{U'-\mathfrak{Set}}(S,Ob(U(n)-\mathfrak{Cat})) \longrightarrow Ob(U(n)-\mathfrak{Cat})$, the canonical constructions described in the previous lemma's proof.

\sss{$\bold{Definition}$ of the Restricted Simplicial Sets}
Define $\Delta \in Ob(U-\mathfrak{Set})$ to to be the simplicial category, i.e. its objects are finite ordered sets and its arrows are order-preserving functions.

For any $n \in Ob(U-\mathfrak{Set})$, define the category $\Delta_{(n)} := \Delta_{\backslash (\{ j \in \mathbb{N} ; j \leq n-1 \} , \leq_{\mathbb{N}} )} = \downarrow_{(\Delta)}(ob_{(\Delta)}((\{ j \in \mathbb{N} ; j \leq n-1 \} , \leq_{\mathbb{N}} )),id_{(U-\mathfrak{Cat})(\Delta)} )$. This is the arrow category under the set with $n$ elements.

\sss{$\bold{Definition}$ of Primitive Arrows}
$\forall n \in \mathbb{N}$, $\forall f \in Arr(\Delta)$, $f$ is primitive iff $||dom(f)|-|codom(f)|| = 1$. $\forall \phi = (f, e, id_{\circ}) \in Arr(\Delta_{(n)})$, $\phi$ is primitive iff $f$ is primitive.

\begin{lem}
Any arrow in $\Delta$ or $\Delta_{(n)}$ is a composition of primitive arrows.
\end{lem}

\sss{Lemma on a Pseudo-Simplicial Structure on $(U.n)-\mathfrak{Cat}$}\label{LemSimp}
For any $n \in \mathbb{N}$, there exists a unique 
$$
\rho \in Hom_{(U''-\mathfrak{Cat})}(\Delta_{(n)}, \downarrow_{(U'-\mathfrak{Cat})}(ob_{(U'-\mathfrak{Cat})}(U(n)-\mathfrak{Cat}),id_{U'-\mathfrak{Cat}}))
$$ 

such that for any $\phi = (f,id_{(\{1,...,n\}, \leq )}) \in Arr(\Delta_{(n)})$, $f$ is primitive implies the following.

\ref{LemSimp}.1. If $f$ injective, then $\rho_{(1)}(\phi) : U(|dom(f)|)-\mathfrak{Cat} \longrightarrow U(|codom(f)|)-\mathfrak{Cat}$ is defined on objects by
$$
\rho_{(1)}(\phi)_{(0)} : (C,h,\circ) \mapsto (D,\bar{h},\bar{\circ})
$$

iff
$$
For^{U(|codom(f)|)-\mathfrak{Cat}}_{U(|dom(f)|-1)-\mathfrak{Cat}}(D) = For^{U(|dom(f)|)-\mathfrak{Cat}}_{U(|dom(f)|-1)-\mathfrak{Cat}}(C)
$$

and for any full address $\alpha = (a,b,C_{\alpha},h_{\alpha},\circ_{\alpha}) \in fAdd_{U(|dom(f)|)}((C,h,\circ ))$, for any $k \in \{ 1,...,|dom(f)| \}$, $f(k+1) = f(k) + 2$ implies
$$
Ob(h_{\alpha}(k)(a(k),b(k))) = \{ \emptyset \} \text{ and }
$$
$$
\alpha \in fAdd_{U(|codom(f)|)}((D,\bar{h},\bar{\circ})) \text{ and }
$$
$$
\forall \bar{\alpha} = (\bar{a},\bar{b},C_{\bar{\alpha}},h_{\bar{\alpha}},\circ_{\bar{\alpha}} ) \in fAdd_{U(|codom(f)|)}((D,\bar{h},\bar{\circ})),
$$
$$
\ulcorner |\bar{\alpha}| = |\alpha |+1 \text{ and } \bar{\alpha}_{\{0,...,k\} } = \alpha \urcorner  \Longrightarrow h_{\bar{\alpha}}(|\bar{\alpha}|)(\emptyset ,\emptyset ) = h_{\alpha}(|\alpha|)(a(k),b(k))
$$

The functor $\rho_{1}(\phi)$ is defined on arrows by
$$
\rho_{1}(\phi)_{(1)} : F = ((F_{0},F_{1}),F_{2}) \mapsto ((G_{0},G_{1}),G_{2}) = G
$$

iff
$$
For^{U(|codom(f)|)-\mathfrak{Cat}}_{U(|dom(f)|-1)-\mathfrak{Cat}}(G) = For^{U(|dom(f)|)-\mathfrak{Cat}}_{U(|dom(f)|-1)-\mathfrak{Cat}}(F)
$$

and for any address $\alpha = (a,b) \in Add_{U(|codom(f)|)}(dom(G))$, for any $k \in \{ 1,...,|dom(f)| \}$,
$$
\ulcorner f(k+1) = f(k)+2 \text{ and } |\alpha | = k+1 \urcorner \Longrightarrow 
$$
$$
Add_{U(|codom(f)|)1}(G)(\alpha) = Add_{U(|dom(f)|)1}(F)(\alpha|_{ \{ 0,...,k \} } )
$$

\ref{LemSimp}.2. If $f$ is surjective, then $\rho_{1}(\phi) : U(|dom(f)|)-\mathfrak{Cat} \longrightarrow U(|codom(f)|)-\mathfrak{Cat}$ is defined on objects by
$$
\rho_{1}(\phi)_{(0)} :(C,h,\circ) \mapsto (D,\bar{h},\bar{\circ})
$$

iff
$$
For^{U(|dom(f)|)-\mathfrak{Cat}}_{U(|codom(f)|-1)-\mathfrak{Cat}}(C) = For^{U(|codom(f)|)-\mathfrak{Cat}}_{U(|codom(f)|-1)-\mathfrak{Cat}}(D)
$$

and for any full address $\alpha = (a,b,C_{\alpha},h_{\alpha},\circ_{\alpha}) \in fAdd_{U(|codom(f)|)0}(D)$, for any $k \in \mathbb{N}$, $f(k+1) = f(k)$ implies
$$
(C_{\alpha}(k), h_{\alpha}(k), \circ_{\alpha}(k) ) = \coprod_{\bar{\alpha} \in S } (C_{\bar{\alpha}}(k+1),h_{\bar{\alpha}}(k+1), \circ_{\bar{\alpha}}(k+1))
$$

where 
$$
S = \{ \bar{\alpha} = (\bar{a},\bar{b},C_{\bar{\alpha}},h_{\bar{\alpha}},\circ_{\bar{\alpha}} ) \in fAdd_{U(|dom(f)|)0}(C) ;
$$
$$
(\bar{a},\bar{b})_{\{0,...,k-1\}} = (a,b)_{ \{ 0,...,k-1\} } \text{ and } |\bar{\alpha}| = k+1 \}.
$$

The functor $\rho_{1}(\phi)$ is defined on arrows by
$$
\rho_{1}(\phi) : F = ((F_{0},F_{1}),F_{2}) \mapsto ((G_{0},G_{1}),G_{2}) = G
$$

iff
$$
For^{U(|dom(f)|)-\mathfrak{Cat}}_{U(|codom(f)|-1)-\mathfrak{Cat}}(F) = For^{U(|codom(f)|)-\mathfrak{Cat}}_{U(|codom(f)|-1)-\mathfrak{Cat}}(G)
$$

and for any full address $\alpha = (a,b,C_{\alpha},h_{\alpha},\circ_{\alpha}) \in fAdd_{U(|codom(f)|)0}(C)$, for any $k \in \mathbb{N}$, $f(k+1) = f(k)$ and $|\alpha | = k+1$ imply
$$
Add_{U(|codom(f)|)1}(G)(\alpha) = \coprod_{\bar{\alpha} \in S} Add_{U(|dom(f)|)1}(F)(\bar{\alpha})
$$

I.e. if $f$ is injective, delete the k-th step, replacing it with the coproduct of all n-k-1 categories appearing in the enriched homs there. If surjective, add a step, a base category with only one object, leaving its enriched hom as that which had preceded it.

\sss{Lemma on Representing the k-Arrows Functor}
Adopt the notation of (\ref{LemSimp}). Then for any arrow $f \in Arr(\Delta_{(n)})$ if the functor $R : (U,n)-\mathfrak{Cat} \longrightarrow (U,|f|)-\mathfrak{Cat}$ given by requiring that $\rho(f) = ( \cdot , R , (U,|codom(f)|)-\mathfrak{Cat})$, then the functor
$$
F_{(U,n)-\mathfrak{Cat}} \cdot \rho(f) : (U,n)-\mathfrak{Cat} \longrightarrow U'-\mathfrak{Set}
$$

is representable. 

\sss{Remark} I expect there to be some enriched version of this.

\begin{lem}
Adjunction of functors given to opposite pairs of primitive arrows by $\rho$.
\end{lem}

\sss{Conjecture on $(sk)$-associativity for a Subcategory of $n-Cat$}\label{assncat}

For any $n \in \mathbb{N}$, for any $\bar{B} = (B,h_{B},\bar{\circ}) \in Ob((WE(U,n)-\mathfrak{Cat},\times(n)))$, if there exists $C \in Ob(U-\mathfrak{Cat})$, and 
$$
(Add(\bar{B}) \xrightarrow{\Phi} Arr(U-\mathfrak{Cat})), (Add(\bar{B}) \xrightarrow{\varepsilon} Arr(U-\mathfrak{Cat}) \in Arr(U'-\mathfrak{Set})
$$

satisfying the following, properties, then $\bar{B}$ is $sk(n)$-associative.


\ref{assncat}.1. $C$ has colimits.

\ref{assncat}.2. For any address $\beta = (B_{i},h_{i},\bar{\circ}_{i},a_{i},b_{i})_{i \in \{ 1,...,k \} } \in Add(\bar{B})$, for some $c,d \in Ob(C)$, the functor
$$
\varepsilon(\beta) : E \longrightarrow C_{/c} = \downarrow_{(C)}(id_{(C)},ob_{(C)}(c))
$$

is faithful, and
$$
\Phi(\beta) : For^{(U,n-|\beta|)-\mathfrak{Cat}}_{U-\mathfrak{Cat}} (B_{|\beta|}) \longrightarrow Hom^{(1)}_{U-\mathfrak{Cat}^{2}}(E,C_{/d})
$$

is an equivalence of categories, where ``$|\beta|$" denotes the order of $\beta$ as a set of pairs, i.e. the number of categories or pairs of objects appearing in the sequence.

\ref{assncat}.3. The functors $\Phi(\beta)$ agree with the composition given by the Hom enrichment lemma, (\ref{PushPull}), up to natural isomorphism. Explanation follows.

\ref{assncat}.3.1. Let there be three addresses $\beta,\beta_{1},\beta_{2} \in Add(\bar{B})$, such that 
$$
|\beta_{1}| = |\beta_{2}| = |\beta_{3}| = |\beta|+1 \text{ and } \beta = \beta_{1} \cap \beta_{2} \cap \beta_{3}
$$
and 

$$
b_{1(|\beta|+1)} = a_{2(|\beta|+1)}
\text{ and }
a_{3(|\beta|+1)} = a_{1(|\beta|+1)}
\text{ and }
b_{3(|\beta|+1)} = b_{2(|\beta|+1)}
$$

i.e. the addresses $\beta_{1}$ and $\beta_{2}$ correspond to a triple $a_{1(k+1)},b_{1(k+1)}=a_{2(k+1)},b_{2(k+1)} \in Ob(C_{k})$ in the underlying category for one of the hom objects, composed to yield $\beta_{3}$.

\ref{assncat}.3.2. Then there is a natural isomorphism of functors
$$
\bar{\circ}_{\mathfrak{Cat}} \circ (\Phi(\beta_{1}) \times \Phi(\beta_{2})) \cong \Phi(\beta_{3})\circ For^{(U,n-|\beta|)-\mathfrak{Cat}}_{U-\mathfrak{Cat}}(\bar{\circ}),
$$

where $\bar{\circ}_{\mathfrak{Cat}}$ is that of (Enrichment, \ref{PushPull}) for $(U,2)-\mathfrak{Cat}$.








\sus{Constellations}

We describe, in this section, a method by which $WE_{(sk)}(A,\otimes)$-enriched sets can be constructed from associative $(A,\otimes)$-enriched sets. Given an assignment, to each pair of objects in $\bar{S} \in Ob(WE_{Assoc(sk)}(A,\otimes))$, of an $(A,\otimes)$-enriched set, we define the hom object attached to a particular pair to be the enriched set of $(A,\otimes)$-functors which send distinguished elements to the pair. The composition is the composition of a restriction functor with a Kan extension functor which would extend the domains of the two component arrows. Such an enriched set is called a constellation of $\bar{S}$. Under certain conditions, constellations are associative.

We also construct ``Lens functors," $(A,\otimes)$-arrows $T : L \rightarrow WE_{(\bar{sk})}(WE_{(sk)}(A,\otimes),\times)$, from sub-enriched sets $L \subseteq Stell(\bar{S},...)$ of constellations to the enriched set of enriched sets. If a constellation is associative, then the sub-enriched sets generated by such functors would inherit the associativity.

\sss{Lemma. Enriched Yoneda} Let $(A,\otimes)$ be a tensor category with an $(sk)$-associator $\alpha$ and $(sk)$-units $\lambda_{u}$ and $\rho_{u}$. If an object $a \in S$ admits an $(sk)$-identity, $I \rightarrow h_{S}(a,a)$, then any ``$(sk)$-natural transformation of functors" $h_{S}(-,a) \rightarrow h_{S}(-,b)$ is given by an associated arrow $I \rightarrow h_{S}(a,b)$. In detail, for any tensor category $(A,\otimes)$ with $(sk)$-associators and an $(sk)$-unit,
$$
\ulcorner \forall (A,\otimes) \in Ob( U-\mathfrak{TCat}), \ulcorner \forall (sk : A \longrightarrow B) \in Arr(U-\mathfrak{Cat}),
$$
$$
\ulcorner \forall I \in Ob(A), \ulcorner \forall \alpha,\lambda_{u},\rho_{u}, \ulcorner\ulcorner (I,\lambda_{u},\rho_{u}) \text{ is an } (\otimes,sk)-unit \urcorner \Longrightarrow
$$

for any enriched set $(S,h_{S},\circ_{S})$, for any two objects, $a,b \in S$,
$$
\ulcorner \forall (S,h_{S},\circ_{S}) \in Ob(WE_{(sk)}(A,\otimes), \ulcorner \forall a,b \in S,
$$

if $\Phi$ is a map of sets assigning to each $c \in S$ an arrow $h_{S}(c,a) \rightarrow h_{S}(c,b)$ in $A$, satisfying a certain naturality condition, and if the arrow $id_{A} : I \rightarrow h_{S}(a,a)$ is an $(sk)$-identity of the object $A$ (in the sense that $sk(\circ(a,a,x) \cdot (id_{h_{S}(x,a)} \otimes id_{a}) \rho_{u}) = sk(id_{h_{S}(x,a)})$),
$$
\ulcorner \forall \Phi \in Hom_{(U'-\mathfrak{Set})}(S,Arr(A)), \ulcorner \forall id_{a} \in Hom_{A}(I,h_{S}(a,a)),
$$
$$
\ulcorner\ulcorner\ulcorner \forall c \in S, \ulcorner \Phi(c) \in Hom_{(A)}(h_{S}(c,a),h_{S}(c,b)) \urcorner\urcorner \text{ and }
$$
$$
\ulcorner \forall c,d \in S, \ulcorner sk(\circ_{S}(d,c,b) \cdot (id_{h_{S}(d,c)} \otimes \Phi(c)) = sk(\Phi(d) \cdot \circ_{S}(d,c,a)) \urcorner\urcorner \text{ and }
$$
$$
\ulcorner id_{a} \text{ is a } (\circ_{S},\rho_{u},sk)-identity_{WE}\urcorner\urcorner \Longrightarrow
$$

then for any $c \in S$, $\Phi(c)$ is indicated by $\Phi(a)$ and the identity.
$$
\ulcorner \forall c \in S, \ulcorner sk(\Phi(c)) = sk(\circ_{S}(c,a,b) \cdot (id_{h_{S}(c,a)} \otimes (\Phi(a) \cdot id_{a})) \cdot \rho_{u}(h_{S}(c,a))) \urcorner\urcorner\urcorner\urcorner\urcorner\urcorner\urcorner\urcorner\urcorner\urcorner\urcorner\urcorner
$$

\sss{$\bold{Definition}$ of Enriched Adjoints}\label{EnrichAdj} The first part defines one-sided adjoints. The second side requires an associator, $\alpha$, and defines two-sided adjoints. \newline

\ref{EnrichAdj}.1. We define, for a tensor category $(A,\otimes)$ an $ladjunction(F)$,
$$
\ulcorner \forall (A,\otimes) \in Ob(U-\mathfrak{TCat}), \ulcorner \forall (sk : A \longrightarrow B ) \in Arr(U-\mathfrak{Cat}),
$$
$$
\ulcorner \forall F,G \in Arr(WE_{(sk)}(A,\otimes)), \ulcorner\ulcorner\ulcorner dom(F) = codom(G)\urcorner \text{ and } \ulcorner dom(G) = codom(F) \urcorner\urcorner \Longrightarrow
$$

to be a map $\Phi$, from pairs of objects in the opposing domains to $Arr(A)$,
$$
\ulcorner \forall \Phi \in Hom_{U-\mathfrak{Set}}(Ob_{(A,\otimes)}(dom(F)) \times Ob_{(A,\otimes)}(dom(G)),Arr(A)),
$$
$$
\ulcorner\ulcorner \Phi \text{ is an ladjunction}(F,G)\urcorner \Longleftrightarrow
$$

which satisfies a naturality condition,
$$
\ulcorner \forall a,b \in Ob_{(A,\otimes)}(dom(F)), \ulcorner \forall c \in Ob_{(A,\otimes)}(dom(G)), \ulcorner\ulcorner \Phi(a,b) \text{ is an isomorphism} \urcorner
$$

$
\text{ and }$ $\ulcorner sk(\circ_{dom(G)}(F(a),F(b),c) \cdot (F_{(1)}(a,b) \otimes id_{h_{dom(G)}(F(b),c)}) \cdot (id_{h_{dom(F)}(a,b)} \otimes \Phi(b,c))) = sk(\Phi(a,c)\cdot \circ_{dom(F)}(a,b,G(c))) \urcorner \text{ and }$

$
\ulcorner sk(\Phi(a,c)^{-1} \cdot \circ_{dom(G)}(F(a),F(b),c) \cdot (F_{(1)}(a,b) \otimes id_{h_{dom(G)}(F(b),c)})) = sk (\circ_{dom(F)}(a,b,G(c)) \cdot (id_{h_{dom(F)}(a,b)} \otimes \Phi(b,c)^{-1})) \urcorner\urcorner\urcorner\urcorner\urcorner \text{ and } \ulcorner\ulcorner \Phi \text{ is an radjunction}(F,G)\urcorner \Longleftrightarrow$

$
\text{(symmetric condition, respecting composition on the other side)} \urcorner\urcorner\urcorner\urcorner\urcorner\urcorner\urcorner$ 

with a symmetric definition for an $radjunction(F,G)$ on the other side.

\ref{EnrichAdj}.2. For a tensor category $(A,\otimes)$ with an $(sk)$ associator $\alpha$,
$$
\ulcorner \forall (A,\otimes) \in Ob(U-\mathfrak{ATCat}), \ulcorner \alpha \in U, \ulcorner \forall (sk : A \longrightarrow B) \in Arr(U-\mathfrak{Cat}),
$$
$$
\ulcorner \forall F,G \in Arr(WE_{(sk)}(A,\otimes)), \ulcorner\ulcorner\ulcorner \alpha \text{ is an (sk)-associator}(A,\otimes) \urcorner
$$
$$
\text{ and } \ulcorner dom(F) = codom(G) \urcorner \text{ and } \ulcorner codom(F) = dom(G) \urcorner\urcorner \Longrightarrow
$$

we similarly define an $adjunction(F,G)$ to be a map of sets $\Phi$,
$$
\ulcorner \forall \Phi \in Hom_{U-\mathfrak{Set}}(Ob(dom(F)) \times Ob(dom(G)),Arr(A)),
$$
$$
\ulcorner\ulcorner \Phi \text{ is an adjunction}(F,G)\urcorner \Longleftrightarrow
$$
$$
\ulcorner \forall (a,b) \in dom(\Phi), \ulcorner\ulcorner \Phi(a,b,) \text{ is an isomorphism}\urcorner \text{ and } \ulcorner \forall (a',b') \in dom(\Phi),
$$
$$
\ulcorner\ulcorner sk(\Phi(a',b') \cdot \circ_{dom(F)}(a',a,G(b')) \cdot (id_{h_{dom(F)}(a',a)} \otimes \circ_{dom(F)}(a,G(b),G(b'))) \cdot
$$
$$
(id_{h_{dom(F)}(a',a)} \otimes (id_{h_{dom(F)}(a,G(b))} \otimes G_{(1)}(b,b')))) =
$$
$$
sk(\Phi(a',b') \cdot \circ_{dom(F)}(a',G(b),G(b')) \cdot ( \circ_{dom(F)}(a',a,G(b)) \otimes G_{(1)}(b,b')) \cdot
$$
$$
 \alpha (h_{dom(F)}(a',a),h_{dom(F)}(a,G(b)),h_{dom(G)}(b,b'))^{-1}) = 
$$
$$
sk(\circ_{dom(G)}(F(a'),F(a),b') \cdot (F_{(1)}(a',a) \otimes \Phi(a,b')) \cdot (id_{h_{dom(F)}(a',a)} \otimes \circ_{dom(F)}(a,G(b),G(b))) \cdot
$$
$$
(id_{h_{dom(F)}(a',a)} \otimes (id_{h_{dom(F)}(a,G(b))} \otimes G_{(1)}(b,b')))) =
$$
$$
sk(\circ_{dom(G)}(F(a'),b,b')) \cdot (\Phi(a',b) \otimes id_{h_{dom(G)}(b,b')}) \cdot
$$
$$
( \circ_{dom(F)}(a',a,G(b))\otimes id_{h_{dom(G)}(b,b')}) \cdot \alpha(h_{dom(F)}(a',a),h_{dom(F)}(a,G(b)),h_{dom(G)}(b,b'))^{-1}) =
$$
$$
sk( \circ_{dom(G)}(F(a'),F(a),b') \cdot (F_{(1)}(a',a) \otimes \circ_{dom(G)}(F(a),b,b')) \cdot
$$
$$
(id_{h_{dom(F)}(a',a)} \otimes (\Phi(a,b) \otimes id_{h_{dom(G)}(b,b')}))) = 
$$
$$
sk(\circ_{dom(G)}(F(a'),b,b') \cdot (\circ_{dom(G)}(F(a'),F(a),b) \otimes id_{h_{dom(G)}(b,b')}) \cdot
$$
$$
((F_{(1)}(a',a) \otimes \Phi(a,b)) \otimes id_{h_{dom(G)}(b,b')}) \cdot \alpha (h_{dom(F)}(a',a),h_{dom(F)}(a,G(b)),h_{dom(G)}(b,b'))^{-1}) \urcorner
$$
$$
\text{ and } (\text{An analogous set of equalities for } \Phi^{-1}) \urcorner\urcorner\urcorner\urcorner\urcorner\urcorner\urcorner\urcorner\urcorner\urcorner\urcorner
$$

\sss{$\bold{Definition}$ of Enriched Kan Extensions}\label{EnrichKan}. The first part defines functorial Kan extensions, as adjoints to restriction ``functors" $i^{*}$, where a restriction functor is understood to be as in the the lemma on composition functors (\ref{PushPull}). The second part defines pointwise Kan extensions.

\ref{EnrichKan}.1. For a tensor category with skeleton and associator,
$$
\ulcorner\forall (A,\otimes) \in Ob(U-\mathfrak{TCat}),\ulcorner \forall \alpha \in U, \ulcorner \forall (sk : A \longrightarrow B ) \in Arr(U-\mathfrak{Cat}),
$$

for any arrows $i,i^{*}$ of enriched sets for which $i^{*}$ is the ``restriction" of $i$,
$$
\ulcorner \forall i,i^{*} \in Arr(WE_{Assoc(sk)}(A,\otimes)), \ulcorner \forall S \in Ob(WE_{Assoc(sk)}(A,\otimes)),
$$
$$
\ulcorner \forall K \in Hom_{WE_{Assoc(sk)}(A,\otimes)}(
$$
$(Hom_{WE_{Assoc(sk)}(A,\otimes)})(dom(i),S),\bar{h}_{(A,\otimes)}(dom(i),S),\bar{\circ}_{(A,\otimes)}(dom(i),S)),$
$$
(Hom_{WE_{Assoc(sk)}(A,\otimes)}(codom(i),S),\bar{h}_{(A,\otimes)}(codom(i),S),\bar{\circ}_{(A,\otimes)}(codom(i),S))), 
$$
$$
(\text{That } i^{*} : \bar{Hom}_{WE_{Assoc(sk)}(A,\otimes)}(codom(i),S) \rightarrow \bar{Hom}_{WE_{Assoc(sk)}(A,\otimes)}, dom(i),S) \text{ is to } i \text{ as is }
$$
$$
F^{*} \text{ to } F \text{ in } (\ref{PushPull})  )
$$

any pair $(K,\Phi)$ of arrows of sets are said to be left or right Kan extensions (Lan or Ran) accordingly as they serve as adjunctions to the restriction arrow $i^{*}$.
$$
\Longrightarrow \ulcorner \forall \Phi \in Arr(U'-\mathfrak{Set}),
$$
$$
\ulcorner\ulcorner\ulcorner \Phi \text{ is an adjunction}(K,i^{*}) \urcorner \Longleftrightarrow \ulcorner (K,\Phi) \text{ is a } Lan_{(A,\otimes)(sk)}(i)\urcorner\urcorner \text{ and }
$$
$$
\ulcorner\ulcorner (K,\Phi) \text{ is a Ran}_{(A,\otimes)(sk)}(i) \urcorner \Longleftrightarrow \ulcorner \Phi \text{ is an adjunction}(i^{*},K)\urcorner\urcorner\urcorner\urcorner\urcorner\urcorner\urcorner\urcorner\urcorner\urcorner
$$

\ref{EnrichKan}.2. For any tensor category $(A,\otimes)$, for any $F \in Hom_{(WE_{Assoc(sk)}(A,\otimes))}(dom(i),$\newline $codom(F))$, $\bar{F} \in Hom_{WE_{Assoc(sk)}(A,\otimes)}(codom(i),codom(F)), \Phi \in Arr(U'-\mathfrak{Set})$, we say that $(\bar{F},\Phi) \text{ is a }$ $Lan_{(A,\otimes)(sk)}(i,F)$ if it satisfies above requirements, involving a single object on the right side. We say that $(\bar(F),\Phi) \text{ is a Ran}_{(A,\otimes)(sk)}(i,F)$ in the analogous case.

\sss{$\bold{Definition}$ of a Constellation} Suppose that $WE_{(sk)}(A,\otimes)$ has cofibres. Given, (i) an enriched set $\bar{S} = (S,h_{S},\circ_{S})$, and (ii) maps of sets $e_{1},e_{2},e_{3},i_{1},i_{2}$, so that (ii.i) the $e$-maps assign to any triple $r,s,t \in S$ a triple of $(A,\otimes)$-arrows $e_{i}(r,s,t)$ for which all three the same codomain , and (ii.ii) the $i$-maps assign to each $s,t \in S$ distinguished objects in the enriched sets $dom(e_{2}(r,s,t)) = dom(e_{1}(s,t,r)) = dom(e_{3}(s,r,t)$ for arbitrary $r \in S$ (requiring this we implicitly assign to each pair $(s,t)$ this enriched set $dom(e_{2}(r,s,t))$), we define an enriched set on $S$, over $WE_{Assoc(sk)}(A,\otimes)$ by associating to each $s,t \in S$ the enriched set of $(A,\otimes)$-functors $dom(e_{3}(s,r,t) \rightarrow \bar{S}$ which send distinguished elements of the domain to $a$ and $b$. Composition is given by Kan extensions.

In detail, for any tensor category $(A,\otimes)$ with skeleton functor $sk$ and associator $\alpha$, for any $(A,\otimes)$-enriched set $\bar{S}$,
$$
\ulcorner \forall (A,\otimes), \in Ob(U-\mathfrak{ATCat}),\ulcorner \forall \alpha : Ob(A)^{3} \rightarrow Arr(A), \ulcorner \forall (sk : A \longrightarrow B) \in Arr(U-\mathfrak{Cat}),
$$
$$
\ulcorner \forall \bar{S} = (S,h_{S},\circ_{S}),
$$

for any arrows of sets $e_{1},e_{2},e_{3}$ (with their domains and codomains marked by $\bar{I}$ and $\bar{J}$), which assign to tuples of $S$ arrows of $WE_{Assoc(sk)}(A,\otimes)$, and $i_{1},i_{2}$, which distinguish elements within the enriched sets $\bar{I}$, (in which case we say that ``$e_{1},e_{2},e_{3},i_{1},i_{2}$ are constellation data for $\bar{S}$")
$$
\ulcorner\forall \bar{I} = (I,h_{I},\circ_{I}) : S \times S \rightarrow Ob(WE_{Assoc(sk)}(A,\otimes)),
$$
$$
\ulcorner \forall \bar{J} = (J,h_{J},\circ_{J}) : S \times S \times S \rightarrow Ob(WE_{Assoc(sk)}(A,\otimes)), 
$$
$$
\ulcorner \forall e_{1},e_{2},e_{3} : S \times S \times S \rightarrow Arr(WE_{Assoc(sk)}(A,\otimes)),
\ulcorner \forall i_{1},i_{2} : S \times S \rightarrow U', 
$$
$$
\ulcorner\ulcorner \forall a,b,c \in S, \ulcorner\ulcorner dom(e_{2}(a,c,b)) = dom(e_{1}(b,a,c)) = dom(e_{3}(a,b,c) = \bar{I}(a,c) \urcorner
$$
$$
\text{ and } \ulcorner codom(e_{1}(a,b,c)) = codom(e_{2}(a,b,c)) = codom(e_{3}(a,b,c)) = \bar{J}(a,b,c) \urcorner
$$
$$
\text{ and } \ulcorner e_{1}(a,b,c)(i_{1}(a,b)) = e_{3}(a,b,c)(i_{1}(a,c)) \urcorner
$$
$$
\text{ and } \ulcorner e_{1}(a,b,c)(i_{2}(a,b)) = e_{2}(a,b,c)(i_{1}(b,c)) \urcorner
$$
$$
\text{ and } \ulcorner e_{2}(a,b,c)(i_{2}(b,c)) = e_{3}(a,b,c)(i_{2}(a,c))\urcorner\urcorner\urcorner
$$

we say that any enriched set $(S,\bar{h},\bar{\circ})$ is a $constellation(\bar{S},e_{1},e_{2},e_{3},i_{1},i_{2})$ iff

$
\Longrightarrow \ulcorner \forall \bar{h} \in Hom_{U-\mathfrak{Set}}(S \times_{U-\mathfrak{Set}} S, Ob(A)), \ulcorner \forall \bar{\circ} \in Hom_{U-\mathfrak{Set}}(S \times_{U-\mathfrak{Set}} S \times_{\mathfrak{Set}} S , Arr(A)),
$

$$
\ulcorner\ulcorner\ulcorner (S,\bar{h},\bar{\circ}) \text{ is an }l-constellation_{(A,\otimes,\alpha)(sk)}(\bar{S},e_{1},e_{2},e_{3},i_{1},i_{2}) \urcorner \Longleftrightarrow
$$

its hom objects are sets of $(A,\otimes)$-arrows respecting the distinguished objects,
$$
\ulcorner\ulcorner \forall a,b \in S, \ulcorner \bar{h}(a,b) = ( h_{0}(a,b) :_{t}= 
$$
$$
\{ \Sigma  \in Hom_{WE_{Assoc(sk)}(A,\otimes)}(\bar{I}(a,b),\bar{S}) ; \ulcorner\ulcorner \Sigma(i_{1}(a,b)) = a \urcorner \text{ and } \ulcorner \Sigma(i_{2}(a,b)) = b \urcorner\urcorner \}, 
$$
$$
\bar{h}_{\bar{h}_{WE_{Assoc(sk)}(A,\otimes)}(\bar{I}(a,b),\bar{S})} \cdot \varepsilon_{h_{0}(a,b)\times h_{0}(a,b)}^{Hom...\times Hom...}, \bar{\circ}_{\bar{h}_{WE_{Assoc(sk)}(A,\otimes)}(\bar{I}(a,b),bar{S})} \cdot \varepsilon_{h_{0}(a,b)\times h_{0}(a,b) \times h_{0}(a,b)}^{Hom...\times Hom... \times Hom...}) \urcorner\urcorner \text{ and }
$$

and the composition arrow $\bar{\circ}(a,b,c)$ is given by the restriction along $e_{3}(a,b,c)$ of a Kan extensions function. The arrow \newline $e_{3*} : \bar{h}_{WE_{Assoc(sk)}(A,\otimes)}(\bar{J}(a,b,c),\bar{S})$ $\rightarrow \bar{h}_{WE_{Assoc(sk)}(A,\otimes)}(\bar{I}(a,c),\bar{S})$ below is as in the \newline $Hom_{WE_{Assoc(sk)}(A,\otimes)}$ enrichment lemma.
$$
\ulcorner \forall a,b,c \in S, 
$$
$$
\ulcorner \exists K \in Hom_{WE_{Assoc(sk)}(A,\otimes)}(
$$
$$
\bar{h}_{WE_{Assoc(sk)}(A,\otimes)}(\bar{I}(a,b)\sqcup \bar{I}(b,c), \bar{S}),\bar{h}_{WE_{Assoc(sk)}(A,\otimes)}(\bar{J}(a,b,c),S)),
$$
$$
\ulcorner \exists \Phi \in Arr(U'-\mathfrak{Set}), 
$$
$$
\ulcorner\ulcorner (K,\Phi) \text{ is a } Lan_{(A,\otimes)(sk)}(\sqcup_{(\bar{I}(a,b),\bar{I}(b,c))(\bar{J}(a,b,c))(0)}(e_{1}(a,b,c),e_{2}(a,b,c))) \urcorner \text{ and }
$$
$$
\ulcorner \bar{\circ}(a,b,c) = e_{3*} \cdot K \cdot \sqcup_{(\bar{I}(a,b),\bar{I}(b,c))(\bar{J}(a,b,c))} \cdot \varepsilon_{h_{0}(a,b)\times h_{0}(a,b)}^{...\times ...} \urcorner\urcorner\urcorner\urcorner\urcorner\urcorner\urcorner
$$

and $\ulcorner\ulcorner (S,\bar{h},\bar{\circ}) \text{ is an r-constellation}_{(A,\otimes,\alpha)(sk)}(\bar{S},e_{1},e_{2},e_{3},i_{1},i_{2}) \urcorner \Longleftrightarrow$ (Analogous, but with Ran)$\urcorner\urcorner\urcorner\urcorner\urcorner\urcorner\urcorner\urcorner\urcorner\urcorner\urcorner\urcorner\urcorner\urcorner$

\sss{Lemma} If $A$ has products and an initial object $\emptyset_{A} \in Ob(A)$, then for any $S_{1},S_{2},S_{3} \in Ob(WE_{Assoc(sk)}(A,\otimes))$, the coproduct map 
$$
\sqcup_{0(S_{1},S_{2},S_{3}} : Hom_{WE_{(sk)}(A,\otimes)}(S_{1},S_{3}) \times Hom_{WE_{(sk)}(A,\otimes)}(S_{2},S_{3}) \rightarrow Hom_{WE_{(sk)}(A,\otimes)}(S_{1} \sqcup S_{2}, S_{3})
$$ 

associated to $S_{1} \sqcup S_{2}$ ``extends" to an arrow 
$$
(\sqcup_{(S_{1},S_{2})(S_{3})} : \bar{h}_{WE_{Assoc(sk)}(A,\otimes)}(S_{1},S_{3}) \times_{WE_{Assoc(sk)}(A,\otimes)} \bar{h}_{WE_{Assoc(sk)}(A,\otimes)}(S_{2},S_{3}) \rightarrow
$$
$$
\bar{h}_{WE_{(sk)}(A,\otimes )}(S_{1} \sqcup S_{2} , S_{3} )) \in Arr(WE_{Assoc(sk)}(A,\otimes))
$$

i.e. there exists an arrow $\sqcup_{(S_{1},S_{2})(S_{3})}$ such that the object functor sends it to the underlying map on sets, $Ob_{(sk,A,\otimes)}(\sqcup_{(S_{1},S_{2})(S_{3})})) = \sqcup_{0(S_{1},S_{2})(S_{3})}$.

\sss{Lemma} If $(A,\otimes)$ has an $(sk)$-unit, $I$, then any two Kan extensions are $(sk)$-equivalent, in the sense of (\ref{SkelQuot}).

\sss{Lemma}  For any $f,g \in Arr(WE_{Assoc(sk)}(A,\otimes)), \bar{sk}(f_{*} \cdot g_{*}) = \bar{sk}((g\cdot f)_{*})$, where $f_{*}\cdot g_{*}, (g\cdot f)_{*} : \bar{h}_{WE_{Assoc(sk)}(A,\otimes)}(codom(g),x) \rightarrow \bar{h}_{WE_{Assoc(sk)}(A,\otimes)}(dom(f),x)$ are as in the enrichment lemma.

\sss{Proposition}\label{AssocStell} Suppose that the tensor category $(A,\otimes)$ with $(sk)$-associator $\alpha$ has an $(sk)$-unit $I$ and $\bar{sk} : WE_{Assoc(sk)}(A,\otimes) \longrightarrow Q$ is the quotient functor given by identifying such $(A,\otimes)$-functors as should admit pairs of equivalences between them, given by elements of the sets given by applying $Yo^{opp}_{(A)}(I)$ to the hom objects in $\bar{h}_{WE_{Assoc(sk)}(A,\otimes)}(c,d)$, as in (\ref{SkelQuot}). Then for any $\bar{S},e_{1},e_{2},e_{3},i_{1},i_{2}$, etc. as in the previous definition, $(i) \Longrightarrow (ii)$.

(i). For any $a,b,c,d \in S$, for any diagram $(\varepsilon : D \longrightarrow WE_{Assoc(sk)}(A,\otimes))  \in Arr(U'-\mathfrak{Cat})$ in $WE_{Assoc(sk)}(A,\otimes)$ which includes a subcategory $D$ of the form
$$
(\{ I(a,b) \sqcup I(b,c) \sqcup I(c,d), I(a,b) \sqcup J(b,c,d), J(a,b,c) \sqcup I(c,d), I(a,b) \sqcup I(b,d), I(a,c) \sqcup I(c,d),
$$
$$
 J(a,b,d), J(a,c,d), I(a,d)  \},
$$
$$
\{ id_{I(a,b)} \sqcup ( e_{1}(b,c,d) \sqcup e_{2}(b,c,d)), id_{I(a,b)} \sqcup e_{3}(b,c,d), e_{1}(a,b,d) \sqcup e_{2}(a,b,d) , e_{3}(a,b,d),\alpha_{\sqcup},
$$
$$
( e_{1}(a,b,c) \sqcup e_{2}(a,b,c)) \sqcup id_{I(c,d)} ,  e_{3}(a,b,c) \sqcup id_{I(c,d)}, e_{1}(a,c,d) \sqcup e_{2}(a,c,d) , e_{3}(a,c,d),
$$
$$
id_{...},id_{...},... \}, Hom_{D}, ... )
$$

for any colimit $(L_{D},\lambda_{D})$ of $D$, for any arrows of enriched sets
$$
K_{\lambda_{D}(I(a,b) \sqcup J(b,c,d))} : \bar{h}_{WE_{Assoc(sk)}(A,\otimes)}(I(a,b) \sqcup J(b,c,d),\bar{S}) \rightarrow \bar{h}_{WE_{Assoc(sk)}(A,\otimes)}(L_{D},\bar{S}) 
$$
$$
K_{e_{1}(a,b,d) \sqcup e_{2}(a,b,d))} : \bar{h}_{WE_{Assoc(sk)}(A,\otimes)}(I(a,b) \sqcup I(b,d),\bar{S}) \rightarrow \bar{h}_{WE_{Assoc(sk)}(A,\otimes)}(J(a,b,d),\bar{S}) 
$$
$$
K_{\lambda_{D}(J(a,b,c) \sqcup I(c,d))} : \bar{h}_{WE_{Assoc(sk)}(A,\otimes)}(J(a,b,c) \sqcup I(c,d),\bar{S}) \rightarrow \bar{h}_{WE_{Assoc(sk)}(A,\otimes)}(L_{D},\bar{S}) 
$$
$$
K_{e_{1}(a,c,d) \sqcup e_{2}(a,c,d))} : \bar{h}_{WE_{Assoc(sk)}(A,\otimes)}(I(a,c) \sqcup I(c,d),\bar{S}) \rightarrow \bar{h}_{WE_{Assoc(sk)}(A,\otimes)}(J(a,c,d),\bar{S}) 
$$

with functions $\Phi_{\lambda_{D}(I(a,b) \sqcup J(b,c,d))} ,\Phi_{e_{1}(a,b,d) \sqcup e_{2}(a,b,d))}, \Phi_{\lambda_{D}(J(a,b,c) \sqcup I(c,d))}, \Phi_{e_{1}(a,c,d) \sqcup e_{2}(a,c,d))} \in$ \newline $Arr(U'-\mathfrak{Set})$, such that each of the pairs 
$$
(K_{\lambda_{D}(I(a,b) \sqcup J(b,c,d))}, \Phi_{\lambda_{D}(I(a,b) \sqcup J(b,c,d))}), (K_{e_{1}(a,b,d) \sqcup e_{2}(a,b,d))}, \Phi_{e_{1}(a,b,d) \sqcup e_{2}(a,b,d)}),
$$
$$
(K_{\lambda_{D}(J(a,b,c) \sqcup I(c,d))}, \Phi_{\lambda_{D}(J(a,b,c) \sqcup I(c,d))} ) , (K_{e_{1}(a,c,d) \sqcup e_{2}(a,c,d))}, \Phi_{e_{1}(a,c,d) \sqcup e_{2}(a,c,d))})
$$ 

is a Lan respectively of the arrows subscribed in its components, we have that
$$
\bar{sk}(\lambda_{D}(J(a,b,d))_{*} \cdot K_{\lambda_{D}(I(a,b) \sqcup J(b,c,d))}) = \bar{sk}(K_{e_{1}(a,b,d) \sqcup e_{2}(a,b,d))} \cdot (id_{I(a,b)} \sqcup e_{3}(b,c,d))_{*})
$$

and
$$
\bar{sk}(\lambda_{D}(J(a,c,d))_{*} \cdot K_{\lambda_{D}(J(a,b,c) \sqcup I(c,d))}) = \bar{sk}(K_{e_{1}(a,c,d) \sqcup e_{2}(a,c,d))} \cdot ( e_{3}(a,b,c) \sqcup id_{I(c,d)})_{*}).
$$

(ii). The enriched set $(S,\bar{h},\bar{\circ}) \in Ob(WE_{(\bar{sk})}(WE_{Assoc(sk)}(A,\otimes),\times_{WE_{Assoc(sk)}(A,\otimes)}))$ is $(\bar{sk})$-associative.

\sss{Remark} The condition (i) of the previous proposition expresses the notion that one can take both Kan extensions, necessary for the composition $\bar{h}(a,b)\times \bar{h}(b,c) \times \bar{h}(c,d) \rightarrow \bar{h}(a,d)$, at the same time. The composition, when done in either either order, should therefor result in a restriction of a functor whose domain is an $(A,\otimes)$-set glued together from the $I$ and $J$ sets.


\sss{Example} $n-\mathfrak{Cat}$. The Yoneda functor $Yo^{opp}_{(n-\mathfrak{Cat})}(I(n))$ of the unit for the product $\times(n)$ differs from the objects functor (given by $Yo^{opp}_{(n-\mathfrak{Cat})}(I_{noenrich})$, where $I_{noenrich}$, has one object, enriched by the empty category (even without a unit). Unless the category of $n$-functors is restricted to those preserving the identity, recording idempotent higher arrows.

\sss{Remark} Given an arrow $(F : \bar{S} \rightarrow \bar{T})$ of enriched sets, one might define an arrow $(\bar{F} : Stell(\bar{S},e,...) \rightarrow Stell(\bar{T},e,...)) \in Arr(WE_{(\bar{sk})}(WE_{Assoc(sk)}(A,\otimes),\times_{WE_{Assoc(sk)}(A,\otimes)}))$, between the left constellations (the data $e_{j},i_{k}$ being the same), if the ``pushforward" $F_{*} : \bar{h}(\bar{I},\bar{S}) \rightarrow \bar{h}(\bar{I},\bar{T})$ commutes with the Kan extensions used in the compositions.

\sss{Proposition}\label{FunctSys} For any $s : S \times S \rightarrow Arr(WE_{Assoc(sk)}(A,\otimes))$, $t_{12},t_{3} : S \times S \times S \rightarrow Arr(WE_{Assoc(sk)}(A,\otimes))$, 

\ref{FunctSys}.1. For any two sets of constellation data for $\bar{S}$, $(e_{1},e_{2},e_{3},i_{e1},i_{e2}),(f_{1},f_{2},f_{3},i_{f1},i_{f2})$, $\ulcorner(i) \Longrightarrow (ii)\urcorner \text{ and } (iii)$.

(i). For any functions $s',t_{12}',t_{3}',e'_{12},e'_{3},f'_{12},f'_{3} : S \times S \times S \rightarrow Arr(U-\mathfrak{Cat})$, for any $a,b \in S$, $s'(a,b) \text{ is a Lan}(s(a,b))$, 

and for any $a,b,c \in S$, 

$t'_{12}(a,b,c) \text{is a Lan}(t_{12}(a,b,c))$ and $t'_{3}(a,b,c) \text{is a Lan}(t_{3}(a,b,c))$ and

$e'_{12}(a,b,c) \text{ is a Lan}(e_{12})$ (for some coproduct arrow $e_{12} = e_{1}(a,b,c) \sqcup e_{2}(a,b,c)$) and \newline $e'_{3} \text{ is a Lan}(e_{3}(a,b,c))$ and 
$$
e_{3*}(a,b,c) : \bar{Hom}_{WE_{Assoc(sk)}(A,\otimes)}(codom(e_{3}(a,b,c)),S) \rightarrow
$$
$$
\bar{Hom}_{WE_{Assoc(sk)}(A,\otimes)}, dom(e_{3}(a,b,c)),S)
$$

is the pullback for $\bar{h}$, and

$f'_{12}(a,b,c) \text{ is a Lan}(f_{12})$ (for some coproduct arrow $f_{12} = f_{1}(a,b,c) \sqcup f_{2}(a,b,c)$) and \newline $f'_{3}(a,b,c) \text{ is a Lan}(f_{3}(a,b,c))$ and 
$$
f_{3*}(a,b,c) : \bar{Hom}_{WE_{Assoc(sk)}(A,\otimes)}(codom(f_{3}(a,b,c)),S) \rightarrow
$$
$$
\bar{Hom}_{WE_{Assoc(sk)}(A,\otimes)}, dom(f_{3}(a,b,c)),S)
$$

is the pullback for $\bar{h}$, imply that for any $a,b,c \in S$
$$
\bar{sk}(\bar{h}(e'_{12}(a,b,c),f_{3*}(a,b,c) \cdot t'_{12}(a,b,c))(e'_{3}(a,b,c)\cdot e_{3*}(a,b,c)))
$$
$$
=
$$
$$
\bar{sk}(\bar{h}(e'_{12}(a,b,c),f_{3*}(a,b,c) \cdot t'_{12}(a,b,c))(id))
$$
$$
and
$$
$$
\bar{sk}(\bar{h}(t'_{3}(a,b,c) \cdot e_{3*}(a,b,c) \cdot e'_{12}(a,b,c),id)(f_{3*}(a,b,c) \cdot f'_{3}(a,b,c)))
$$
$$
=
$$
$$
\bar{sk}(\bar{h}(t'_{3}(a,b,c) \cdot e_{3*}(a,b,c) \cdot e'_{12}(a,b,c),id)(id))
$$

(ii). The pair $F = (id_{S},s') : \bar{S} \rightarrow \bar{T}$ is an arrow in \newline $WE_{\bar{sk}}(WE_{Assoc(sk)}(A,\otimes),\times_{WE_{Assoc(sk)}(A,\otimes)})$, where $\bar{S}$ is the constellation formed of the $e_{1,2},e_{3}$ data and ... .

(iii). For any tuples $(s_{1},t_{112},t_{13},s'_{1},t'_{112},t'_{13})$ and $(s_{2},t_{212},t_{23},s'_{2},t'_{212},t'_{23})$, for any tuple $(s'_{3},t'_{312},t'_{33})$ such that the first two tuples and the tuple $(s_{2} \cdot s_{1}, t_{212}\cdot t_{112}, t_{23}\cdot t_{13}, s'_{3},t'_{313},t'_{33}) $ all satisfy condition (i) in the place of $(s,t_{12},t_{3},s',t'_{12},t'_{3})$,
$$
\bar{sk}(s'_{2} \cdot s'_{1}) = \bar{sk}(s'_{3}).
$$

\ref{FunctSys}.2. Dual, using Ran.


\sss{Lemma} For any $J \in Ob(A)$, for any arrow of functors $\rho$ for which $(Yo^{opp}(J),\rho) : (A,\otimes) \rightarrow (\mathfrak{Set},\times_{\mathfrak{Set}})$ is an arrow of tensor categories, for any objects $a,b \in Ob(S) = Ob(\bar{S})$,
$$
h_{WE(Ob,\rho_{0}) \cdot Stell_{(Bar_{(\lambda_{A},\rho_{A})}(J))} (\bar{S})}(a,b) \cong h_{WE(Yo^{opp}(J),\rho)(\bar{S})}(a,b)
$$

i.e. the enriched sets $WE(Ob,\rho_{0}) \cdot Stell_{(Bar_{(\lambda_{A},\rho_{A})}(J))} (\bar{S})$ and $WE(Yo^{opp}(J),\rho)(\bar{S})$ have isomorphic hom objects.

\sss{Remark}\label{remNewEnrich} Any $(A,\otimes)$-enriched set, $\bar{S} = (S,...)$ can be reconsidered in several ways as a $(WE_{Assoc(sk)}(A,\otimes),\times_{WE_{Assoc(sk)}(A,\otimes)})$-enriched set, with such data $(s,t_{3})$ as above (assignments of functors between composition data) allowing for their comparison. This might be inductively applied to induce an $((n,A)-\mathfrak{Cat},\times)$-enrichment with the same underlying $(\mathfrak{Set},\times)$-enrichment as the original enriched set $\bar{S}$.

\sss{Example} Trivial. Assign to $x,y \in Ob(C)$ the diagram $(x_{0} \rightarrow y_{0})$ . The composition is given by gluing two such diagrams into a diagram 
$$
((x_{0} \rightarrow y_{0}) \sqcup (y_{0} \rightarrow z_{0})) \hookrightarrow (x_{0} \rightarrow y_{0} \rightarrow z_{0}) \hookleftarrow (x_{0} \rightarrow z_{0})
$$

so that the first arrow is composed of the $e_{1},e_{2}$ arrows and the third is the $e_{3}$ arrow.

\sss{Example} ``Localization." Assign to $x,y \in Ob(C)$ the diagram with objects $x_{0}$,$a$, and $y_{0}$, and arrows $(x_{0} \rightarrow a)$, $(a \rightarrow x_{0})$, $(a \rightarrow y_{0})$ and $(x_{0} \rightarrow y_{0})$, with both arrows from $x_{0}$ to $y_{0}$ equal. The composition is given by gluing the individual diagrams at the middle object, and adding an object $c$, ``above both $a$ above $x_{0} \rightarrow y_{0}$ and $b$ above $y_{0} \rightarrow z_{0}$."

One might, in applying (\ref{FunctSys}), send $(x_{00} \rightarrow y_{00})$ as in the first example to the arrow $(x_{0} \rightarrow y_{0})$.

\sss{Example}\label{ExSysFib} Fibre functors, their adjoints. Given an admissibility structure $\varepsilon : E \hookrightarrow Fib$ on a category $\mathcal{T}$ (see \cite{Wi}), assign to $x,y \in Ob(\mathcal{T})$ the sub-category of diagrams which is the disjoint union of the arrow category over $y$ with itself, with an additional arrow $(u,1) \rightarrow (u,2)$ for each object $u \in Ob(E(y))$. The distinguished objects are $y_{0} = (id_{y},1)$ and $x_{0} = (id_{y},2)$.

\sss{Remark} One obtains $(k+1)$-arrows from $k$-arrows by appending a tuple of arrows of $C$, for the natural transformations.


\sss{$\bold{Theorem}$ (``Lens")}\label{Lens} For any arrows of enriched sets $u_{d},u_{c} : \bar{I} = (I,...) \rightarrow \bar{J} = (J,...)$ with two distinguished objects $i'_{1},i'_{2} \in I$, for any sub-enriched set 
$$
\lambda : \bar{L} = (L,h_{L},\circ_{L}) \hookrightarrow Stell(\bar{S},...)) \in Arr(WE_{(\bar{sk})}(WE_{Assoc(sk)}(A,\otimes)
$$

for any functions $c : L \rightarrow \bold{2}^{L},l_{d},l_{c} :  F \cdot \bar{Arr}(\bar{L}) \rightarrow Arr(WE_{Assoc(sk)}(A,\otimes)$ such that for any $k,l,m \in Ob(L)$, for any $\phi \in Ob(h_{L}(l,m))$, $\psi \in Ob(h_{L}(k,l))$, the following, (i) and (ii), hold,

(i). $dom(l_{d}(\phi)) = dom(l_{c}(\psi)) = \langle c(m) \rangle_{WEFull(\bar{L})} \subseteq \bar{L}$, and $codom(l_{d}(\phi)) = dom(e_{2}(k,l,m))$ and $codom(l_{c}(\psi)) = dom(e_{1}(k,l,m))$.

(ii). $\bar{sk}(\phi \cdot l_{d}(\phi)) = \bar{sk}(\psi \cdot l_{c}(\psi)$.

we construct arrows
$$
T : \bar{L} \rightarrow \bar{WE}({WE_{Assoc(sk)}(A,\otimes)})) \in Arr(WE_{(\bar{\bar{sk}})}(WE_{Assoc(sk)}(A,\otimes),\times)
$$
$$
T^{o} : \bar{L}^{opp} \rightarrow \bar{WE}({WE_{Assoc(sk)}(A,\otimes)})) \in Arr(WE_{(\bar{sk})}(WE_{Assoc(sk)}(A,\otimes),\times)
$$

by the following (\ref{Lens}.1), (\ref{Lens}.2), and (\ref{Lens}.3). For any l-constellation $Stell(\bar{S},...)$,

\ref{Lens}.1. For any $s \in L \subseteq S = Ob(Stell(\bar{S},...))$, the arrow
$$
u_{(s)} : T_{0}(s) := 
$$
$$
\langle \{ x \in Ob(\bar{h}_{(a,\otimes,sk)}(\bar{I},\bar{S}) ; \exists \phi \in F \circ \bar{Arr}(\bar{L}), \exists \tilde{x} \in Arr(WE_{Assoc(sk)}(A,\otimes)), 
$$
$$
\ulcorner\ulcorner sk(\phi \cdot l_{d}(\phi) \cdot \tilde{x}) = sk(x) \urcorner \text{ and } \ulcorner x_{(0)}(i'_{1}) = i_{0} \urcorner\urcorner \} \rangle 
$$
$$
\hookrightarrow \bar{h}_{(A,\otimes,sk)}(\bar{I},\bar{S})
$$ 

is the inclusion of the sub-enriched set generated by $(A,\otimes)$-functors $\bar{sk}$ factoring through some arrow $\phi \in F \circ \bar{Arr}(\bar{L})$, which agree on distinguished elements, the (domain and codomain enriched sets).

\ref{Lens}.2. For any $s,t \in L$, for any $\phi \in Ob(h_{L}(s,t))$, let 
$$
H_{(s,t)} :_{t}= 
$$
$$
\langle \{ x \in Ob(\bar{h}_{(a,\otimes,sk)}(\bar{J},dom(e_{1}(s,t,t))) ; \forall \phi \in Ob(h_{L}(s,t)),
$$
$$
\ulcorner\ulcorner  u_{d}^{*}(\phi \cdot x) \in T_{0}(dom(\phi)) \urcorner \text{ and } \ulcorner u_{c}^{*}(\phi \cdot x) \in T_{0}(codom(\phi)) \urcorner\urcorner \} \rangle
$$
$$
\hookrightarrow \bar{h}_{(A,\otimes)(sk)}(\bar{J},\bar{S})
$$

be the inclusion of the sub-enriched set generated by $(A,\otimes)$-functors \newline $x : \bar{J} \rightarrow dom(e_{1}(s,t,t))$ for which the composition $\phi \cdot x$ for any arrow $\phi \in Ob(h_{L}(s,t))$ restricts by both arrows $u_{d},u_{c} : \bar{I} \rightarrow \bar{J}$ to objects of the enriched sets $T_{0}(s)$ and $T_{0}(t)$ respectively.

$$
u_{d*} : \bar{h}_{(A,\otimes)(sk)}(\bar{J},\bar{S}) \rightarrow T_{0}(dom(\phi)) \subseteq \bar{h}_{(A,\otimes,sk)}(\bar{I},\bar{S})) \in Arr(WE_{Assoc(sk)}(A,\otimes)
$$
$$
u_{c*} : \bar{h}_{(A,\otimes)(sk)}(\bar{J},\bar{S}) \rightarrow T_{0}(codom(\phi)) \subseteq  \bar{h}_{(A,\otimes,sk)}(\bar{I},\bar{S})) \in Arr(WE_{Assoc(sk)}(A,\otimes)
$$

be the restrictions. Suppose that
$$
u_{d(s,t)}' : T_{0}(dom(\phi)) \rightarrow H_{(s,t)}
$$

is a left (right for r-constellations) adjoint of $u_{d*}$. Then 
$$
T_{10}(\phi) := u_{c(s,t)}^{*} \cdot \phi_{*} \cdot u_{d(s,t)}' \cdot l_{d(s,t)*} : T_{0}(dom_{L}(\phi)) \rightarrow T_{0}(codom_{L}(\phi))
$$

is the element of the set $Hom_{WE_{sk}(A,\otimes)}(T_{(0)}(dom(\phi)),T_{(0)}(codom(\phi)))$ assigned to $\phi$.

\ref{Lens}.3. For any $s,t \in S$, for any $\phi,\psi \in Ob(h_{L}(s,t))$ consider the projections
$$
\prod_{i \in Ob(dom(\phi))} h_{S}(\phi_{(0)}(i),\psi_{(0)}(i)) \rightarrow \prod_{f \in Ob(T_{0}(s))} \prod_{i \in Ob(dom(f))} h_{S}(T_{10}(\phi)_{(0)}(f)(i),T_{10}(\psi)_{(0)}(f)(i)
$$

from the full tuple of hom-objects (a sub-object of which is the object of $A$ specified in the enrichment lemma) to those which appear in the codomain components, i.e. between the images of objects $i$ under the functors $T(\phi)(f)$ and $T(\psi)(f)$.  Composition of this projection with the maps $x \rightarrow \prod_{i \in Ob(dom(\phi))} h_{S}(\phi_{(0)}(i),\psi_{(0)}(i))$ associated to each object in the colimit diagram induces arrows $x \rightarrow \prod_{i \in Ob(dom(f))} h_{S}(T_{10}(\phi)_{(0)}(f)(i),T_{10}(\psi)_{(0)}(f)(i)$ compatible with the composition requirement, and therefore of the latter colimit diagram, and provided with an arrow into the latter colimit. The colimit map for the enrichment object between $\phi$ and $\psi$ induces an arrow,
$$
\bar{h}(dom(e(s,t)),\bar{S})(\phi,\psi) \rightarrow \prod_{f \in Ob(T_{(0)}(s))} \bar{h}(\bar{I},\bar{S})(T_{10}(\phi)_{(0)}(f),T_{10}(\psi)_{(0)}(f)),
$$

which similarly induces an arrow, which we define to be
$$
T_{11}(s,t)(\phi,\psi) : \bar{h}(dom(e(s,t)),\bar{S})(\phi,\psi) \rightarrow \bar{h}(T_{0}(s),T_{0}(t))(T_{10}(\phi),(T_{10}(\psi)).
$$

\ref{Lens}.4. If the condition on arrows $\tau$ of (\ref{EnrichProd}) holds, then an arrow $c \in Hom_{A}(I_{0},I)$, where $I_{0}$ is a unit for the product tensor structure $\times_{A}$ and $I$ is a unit for $\otimes$ (the arrow $c$ is as that of (\ref{PreCurry}.3), and is implicitly used thereby, in the below invocation) define an arrow
$$
\bar{h}(T_{0}(r),\bar{h}(\bar{J},dom(e_{3}(r,s,t)))) \times \bar{h}(h_{L}(r,s) \times h_{L}(s,t) , \bar{h}(dom(e_{3}(r,s,t)),\bar{S}) )
$$
$$
\rightarrow \bar{h}(h_{L}(r,s) \times h_{L}(s,t) , \bar{h}(T_{0}(r),T_{0}(t)))
$$

by pulling back the product maps,
$$
\bar{h}(T_{0}(r),\bar{h}(\bar{J},dom(e_{3}(r,s,t)))) \times \bar{h}(h_{L}(r,s) \times h_{L}(s,t) , \bar{h}(dom(e_{3}(r,s,t)),\bar{S}) )
$$
$$
\rightarrow \bar{h}(T_{0}(r) \times h_{L}(r,s) \times h_{L}(s,t),\bar{h}(\bar{J},dom(e_{3}(r,s,t)))) \times
$$
$$
\bar{h}(T_{0}(r) \times h_{L}(r,s) \times h_{L}(s,t),\bar{h}(dom(e_{3}(r,s,t)),\bar{S}))) \rightarrow
$$

by the lemma on the product enriched set, (\ref{EnrichProd}),
$$
\bar{h}(T_{0}(r) \times h_{L}(r,s) \times h_{L}(s,t),\bar{h}(\bar{J},dom(e_{3}(r,s,t))) \times \bar{h}(dom(e_{3}(r,s,t)),\bar{S})) \rightarrow
$$

by the lemma on composition functors, (\ref{PushPull}.3), with the product unit
$$
\bar{h}(T_{0}(r) \times h_{L}(r,s) \times h_{L}(s,t),\bar{h}(\bar{J},\bar{S})) \rightarrow
$$

by the restriction $u_{c}^{*}$ from the arrow $u_{c} : \bar{I} \rightarrow \bar{J}$, with the product unit,
$$
\bar{h}(T_{0}(r) \times h_{L}(r,s) \times h_{L}(s,t),T_{0}(t)) \rightarrow
$$

by the pre-Curry arrow, (\ref{PreCurry}.3),
$$
\bar{h}(h_{L}(r,s) \times h_{L}(s,t),\bar{h}(T_{0}(r),T_{0}(t))).
$$

\ref{Lens}.4.1. For any colimit $(\bar{J}',\lambda_{J})$ of the diagram $D : (\{ 0,1,2 \} , \{ (0,1), (0,2), id_{0},... \},...)$ \newline  $\longrightarrow WE_{Assoc(sk)}(A,\otimes)$ with two non-identity arrows $u_{d},u_{c} : \bar{I} \rightarrow \bar{J}$ and three objects, so that
$$
D(0) = \bar{I} \text{ and } D(1) = D(2) = \bar{J} \text{ and }
$$
$$
D((0,1)) = u_{c} \text{ and } D((0,2)) = u_{d}.
$$

For any arrow of enriched sets $d : \bar{J} \rightarrow \bar{J}'$, using the notation of (\ref{Lens}.2), we temporarily define a sub-enriched set
$$
H_{comp(r,s,t)} :_{t}= \langle \{ x \in Ob(\bar{h}(\bar{J}',\bar{S})) ;
$$
$$
\ulcorner\ulcorner \lambda_{J}(1)^{*}(x) \in Ob(H_{(r,s)}) \urcorner \text{ and } \ulcorner \lambda_{J}(2)^{*}(x) \in Ob(H_{(s,t)}) \urcorner\urcorner \} \rangle
$$
$$
\hookrightarrow \bar{h}(\bar{J}',\bar{S})
$$

take the arrow
$$
u_{dd*} : H_{comp(r,s,t)} \rightarrow T_{0}(r)
$$

to be the restriction, and the arrow
$$
u_{dd}' : T_{0}(r) \rightarrow H_{comp(r,s,t)}
$$

to be the left (right for r-constellations) Kan extension of $u_{dd*}$. Applying the arrow of enriched sets of (\ref{Lens}.4) to objects constructed from this, temporarily defining
$$
D :_{t}= \bar{h}(T_{0}(r),\bar{h}(\bar{J},dom(e_{3}(r,s,t)))) (e_{3}(r,s,t)_{*} \cdot u_{d}', d^{*} \cdot (e_{1}(r,s,t) \sqcup_{\lambda_{J}} e_{2}(r,s,t))_{*} \cdot u_{dd}')
$$

and
$$
C :_{t}= \bar{h}(h_{L}(r,s) \times h_{L}(s,t) , \bar{h}(dom(e_{3}(r,s,t),\bar{S}) )(e_{1}(r,s,t) \sqcup_{\lambda_{J}} e_{2}(r,s,t))^{*} \cdot K , id)
$$

we obtain an arrow in $A$
$$
D \times_{A} C \rightarrow \bar{h}(h_{L}(r,s) \times h_{L}(s,t) , \bar{h}(T_{0}(r),T_{0}(t)))(\circ \cdot (T_{10}(r,s) \times T_{10}(s,t)) , T_{10}(r,t) \cdot \circ_{L})
$$

where $u_{d}' = u_{d(r,t)}'$ as for $H_{(r,t)}$ and $K$ is the Lan (Ran) functor used for composition. 

\ref{Lens}.4.2. In particular, given an $(sk)$-unit $(I_{0},\lambda_{0},\rho_{0}) \in Ob(A)$ a pair of arrows $(I_{0} \rightarrow D),(I_{0} \rightarrow C) \in Arr(A)$ i.e. ``natural transformations"
$$
e_{3}(r,s,t)_{*} \cdot u_{d}' \rightarrow d^{*} \cdot (e_{1}(r,s,t) \sqcup e_{2}(r,s,t))_{*} \cdot u_{dd}'
$$

and
$$
(e_{1}(r,s,t) \sqcup_{\lambda_{J}} e_{2}(r,s,t))^{*} \cdot K \rightarrow id
$$

determine an arrow
$$
I_{0} \rightarrow \bar{h}(h_{L}(r,s) \times h_{L}(s,t) , \bar{h}(T_{0}(r),T_{0}(t)))(\circ \cdot (T_{10}(r,s) \times T_{10}(s,t)) , T_{10}(r,t) \cdot \circ_{L})
$$

i.e. a ``natural transformation"
$$
\circ(T_{0}(r),T_{0}(s),T_{0}(t)) \cdot (T_{10}(r,s) \times T_{10}(s,t)) \rightarrow T_{10}(r,t) \cdot \circ_{L}(r,s,t)
$$

so that, if the former two arrows are $(sk)$-equivalences, the latter is an $(sk)$-equivalence.

\ref{Lens}.4.3. If, for every $r,s,t \in L$, there are such $(sk)$-equivalences as in (\ref{Lens}.4.2), then (\ref{Lens}.1),(\ref{Lens}.2) and (\ref{Lens}.3) define an arrow $T : \bar{L} \rightarrow \bar{WE}({WE_{Assoc(sk)}(A,\otimes)})$. An arrow $T^{o} : \bar{L}^{opp} \rightarrow \bar{WE}({WE_{Assoc(sk)}(A,\otimes)})$ would be similarly defined, differing in that the ``functors" $T_{10}(\phi)$ are defined by taking the Kan extensions of functors with ``images" in the codomain enriched sets.

\begin{proof}
Articles (\ref{Lens}.1) and (\ref{Lens}.2) are definitions.

The proof of (\ref{Lens}.3) is the argument that the objects $p : x \rightarrow \prod_{i \in dom(e_{1}(s,t,t)} h_{S}(\phi(i),\psi_{i})$ naturally determine objects in the $P$-category which determines the object \newline $h_{\bar{h}(T_{0}(s),T_{0}(t))}(T_{10}(\phi),T_{10}(\psi)) \in Ob(A)$. The arrows $T_{10}(\phi)(x,y)$ and $T_{10}(\psi)(x,y)$ factor through $\phi$ and $\psi$ respectively so that their tensors with the component arrows for $h_{...}(\phi,\psi)$ are commutative with respect to the composition in $\bar{S}$.

Article (\ref{Lens}.4) is a definition composed of the listed lemmas. Its subsections are corollaries of the existence of the constructed arrow.
\end{proof}

\sss{Corollary} If the condition (i) of (\ref{AssocStell}) is satisfied, then any sub-enriched set of $(n+1)-\mathfrak{Cat}$ generated by the image of an arrow of enriched sets of the form $T$ or $T^{opp}$ of (\ref{Lens}), for $(A,\otimes) = (n-\mathfrak{Cat},\times)$, i.e. whose $k$-arrows are the images of those of the chosen $L$, is $(\bar{sk})$-associative.

\sss{Remark} Roughly speaking, the first arrow of (\ref{Lens}.4.2) expresses a ``natural transformation" between the arrow of enriched sets $T_{0}(r) \rightarrow \bar{h}(\bar{J},codom(e_{3}(r,s,t)))$ which finds the closest arrow factoring $\bar{J}$ directly through $dom(e_{3}(r,s,t))$, and that which finds the closest arrow factoring $\bar{J}$ through $dom(e_{1}(r,s,t))$ and $dom(e_{2}(r,s,t))$ together, through a duplication of itself in $\bar{J}'$. It could be thought of as a constellation data in miniature.

The second arrow of (\ref{Lens}.4.2) expresses a ``natural transformation" from the composition of the restriction with the Kan extension (that used for the composition) to the identity functor.

\sss{Remark} I imagine that one might construct functors $End : C \rightarrow WE(A,\otimes)$, borrowing the enrichment of $WE(WE(A,\otimes),\times)$ by the arrows $L,L^{o} : C,C^{opp} \rightarrow WE(A,\otimes)$, and mapping $c$ to $\bar{h}(L^{o}(c),L(c))$ or $\bar{h}(L(c),L(c))$.

\end{document}